\documentclass[english,15pt]{article}
\usepackage[T1]{fontenc}
\usepackage[latin1]{inputenc}
\usepackage{geometry}
\usepackage{float}
\usepackage{mathrsfs}
\usepackage{amsmath}
\usepackage{amssymb}
\usepackage{graphicx}
\usepackage{hyperref}
\usepackage[color=yellow]{todonotes}
\hypersetup{
    colorlinks=true,
    linkcolor=blue,
    filecolor=magenta,      
    urlcolor=cyan,
    citecolor = red,
}

\makeatletter
\usepackage{amsmath,tikz}

\usepackage[english]{babel}
\floatstyle{ruled}
\usepackage{algorithm}
 \usepackage{algorithmicx}

\usepackage{amsthm}

\usepackage{mathrsfs}

\usepackage{amsfonts}

\usepackage{epsfig}

\usepackage{bm}

\usepackage{mathrsfs}

\usepackage{enumerate}

\@ifundefined{definecolor}{\@ifundefined{definecolor}
 {\@ifundefined{definecolor}
 {\usepackage{color}}{}
}{}
}{}

\usepackage{subcaption}

\newtheorem{theorem}{Theorem}[section]
\newtheorem{lem}{Lemma}[section]
\newtheorem{rem}{Remark}[section]
\newtheorem{proposition}{Proposition}[section]

\newcounter{hypA}

\newcounter{hypB}

\newcounter{hypD}

\usepackage{babel}\date{}

\newcommand{\EE}{\mathbb{E}}

\newcommand{\Pa}{ {\cal P }}

\def \EE{\mathbb{E}}

\usepackage{mathtools}

\DeclarePairedDelimiter\floor{\lfloor}{\rfloor}

\numberwithin{equation}{section}

\makeatother

\begin{document}

\begin{center}

{\Large \textbf{Multilevel Estimation of Normalization Constants \\ using Ensemble Kalman--Bucy Filters}}

\vspace{0.5cm}

HAMZA RUZAYQAT,  NEIL K. CHADA \& AJAY JASRA

{\footnotesize Computer, Electrical and Mathematical Sciences and Engineering Division, \\ King Abdullah University of Science and Technology, Thuwal, 23955, KSA.} \\
{\footnotesize E-Mail:\,} \texttt{{\footnotesize hamza.ruzayqat@kaust.edu.sa, neil.chada@kaust.edu.sa, ajay.jasra@kaust.edu.sa}} \\

\end{center}

\begin{abstract}
{In this article we consider the application of multilevel Monte Carlo, for the estimation of normalizing constants. In particular we will make use of the filtering algorithm, the ensemble Kalman--Bucy filter (EnKBF), which is an $N$-particle representation of the Kalman--Bucy filter (KBF). The EnKBF is of interest as it coincides with the optimal filter in the continuous-linear setting, i.e. the KBF. This motivates our particular setup in the linear setting. The resulting methodology we will use is the multilevel ensemble Kalman--Bucy filter (MLEnKBF). We provide an analysis based on deriving $\mathbb{L}_q$-bounds for the normalizing constants using both the single-level, and the  multilevel algorithms, which is largely based on previous work deriving the MLEnKBF \cite{CJY20}. Our results will be highlighted through numerical results, where we firstly demonstrate the error-to-cost rates of the MLEnKBFs comparing it to the EnKBF on a linear Gaussian model. Our analysis will be specific to one variant of the MLEnKBF, whereas the numerics will be tested on different variants. We also exploit this methodology for parameter estimation, where we test this on the models arising in atmospheric sciences, such as the stochastic Lorenz 63 and 96 model.  }
 \\\\
\noindent\textbf{Keywords}: {Multilevel Monte Carlo, Filtering, Kalman--Bucy Filter, Normalizing Constant, \\ Parameter estimation} 
\end{abstract}

\section{Introduction}
\label{sec:intro}
Filtering \cite{BC09,CR11,PD04} is the mathematical discipline concerned with the conditional probability of an unobserved latent process, given sequentially observed data.  It can be found in a wide array of applications, most notably; numerical weather prediction, mathematical finance, geophysical sciences and more recently machine learning \cite{RB10,MW06,ORL08}. Mathematically, given a $d_x$-dimensional unobserved signal process $\{X_t\}_{t \geq 0}$, and a $d_y$-dimensional $\{Y_t\}_{t \geq 0}$ observed process, defined as
\begin{align}
\label{eq:dat}
dY_t &=  h(X_t) dt +  dV_t,   \\
\label{eq:sig}
dX_t &=   f(X_t) dt +  \sigma(X_t) dW_t, 
\end{align}
the aim of the filtering problem is to compute the following expectation $\mathbb{E}[\varphi(X_t)|\mathscr{F}_t]$, where $\varphi:\mathbb{R}^{d_x} \rightarrow \mathbb{R}$ is an appropriately integrable function and $\{\mathscr{F}_t\}_{t \geq 0}$ is the filtration generated by the observed process \eqref{eq:dat}. From \eqref{eq:dat} - \eqref{eq:sig}, $V_t$ and $W_t$ are independent $d_y$ and $d_x-$dimensional Brownian motions respectively, with $h:\mathbb{R}^{d_x} \rightarrow \mathbb{R}^{d_y}$ and $f:\mathbb{R}^{d_x} \rightarrow \mathbb{R}^{d_x}$ denoting potentially nonlinear functions, and $\sigma: \mathbb{R}^{d_x} \rightarrow  \mathbb{R}^{d_x \times d_x}$ acting as a diffusion coefficient. Aside from computing the filtering distribution, filtering can also be exploited to compute  normalizing constants associated with the filtering distribution \cite{BDM18,CDG11,GM98,KW17,RJP18}, i.e. the marginal likelihood, which is an important and useful computation in Bayesian statistics. It is useful as it can be used for model comparison which is commonly done through Bayes' factor. These quantities are particularly useful in mixture models such as time-series state space models, and hierarchical models.  Our motivation from this work is the estimation of the normalizing constants, associated with the filtering distribution in the continuous linear-Gaussian setting. It is well-known in this setting, that the optimal filter is the Kalman--Bucy filter (KBF) \cite{AJ70}. However, the use of such a filter can be challenging to work with, such as firstly actually simulating such processes, or secondly the associated computational cost. As a result an alternative to this filter is the ensemble Kalman--Bucy filter (EnKBF), which is an $N$-particle representation of the KBF, or also as approximations of conditional McKean--Vlasov-type diffusion processes. This filter has been analyzed extensively of recent through various pieces of work which include, but not limited to, understanding stability,  deriving uniform propagation of chaos bounds \cite{BD20,BD17,DT18} and acquiring a multilevel estimator \cite{CJY20,MBG08,MBG15}. However in the context of estimating normalizing constants, recent work has been done by Crisan et al. \cite{CDJ21} where the authors provide a way to do so using the EnKBF. However despite this, there is a computational burden associated with such costs, for example with the EnKBF or other Monte-Carlo algorithms, to attain a mean squared error (MSE) of a particular order. Therefore it would be of interest to apply techniques, to reduce the cost associated of attaining a particular MSE. This is the incentive behind the group of methods based on multilevel Monte Carlo (MLMC).

{MLMC is a numerical technique for stochastic computation, concerned with reducing the computational complexity of Monte Carlo. Specifically it reduces the cost of attaining an MSE of an order $\mathcal{O}(\epsilon^2)$, for $\epsilon>0$.  The ideas of MLMC go back to the original work of Giles and Hendrick \cite{MBG08,MBG15}, which uses mesh refinements, and a telescoping sum property, to reduce the cost. Since then it has been applied to numerous applications, and disciplines, with filtering being one of them. Specifically for Kalman-based filtering this has been applied to the discrete-EnKF, with particular variants also analyzed, and applications arising in reservoir modeling \cite{CHL20,FMS20,HLT16,HST21}. However more recently there has been the extension to the continuous-time setting, for the EnKBF \cite{CJY20}, entitled MLEnKBF. 

In this article we are interested in developing multilevel estimators related to Kalman filtering, for the computation of normalizing constant. In particular we will propose the use of various MLEnKBFs for the application of normalizing constant estimation. Our motivation for this, as mentioned, is that firstly applying ML techniques can reduce the cost, associated to attaining an MSE of a particular oder, compared to the normalizing constant estimator in \cite{CDJ21}. Due to this we will make use of the methodology and proof arguments discussed in \cite{CJY20}. Secondly as it coincides with KBF, it provides an incentive to work in the linear setting, which can be more cost-effective in terms of the cost of the algorithm, to other filtering techniques, which make use of the MLMC, such as the multilevel particle filter and sequential Monte Carlo sampler \cite{DJLZ17,JKL17,JKO18}. We emphasize with this paper, that we are not focused on such a comparison.}

\subsection{Contributions}
Our contributions of this manuscript are highlighted through the following points:
\begin{itemize}
\item We introduce and approximate a multilevel estimator for the computation of normalizing constants associated with EnKBFs. This formulation is based on the multilevel EnKBF introduced by Chada et al. \cite{CJY20}, which looks to extend the single-level estimator which was proposed by Crisan et al. \cite{CDJ21}.
\item Through our formulation we provide firstly a propagation of chaos result for the normalizing constants associated with the EnKBF, which requires appropriate $\mathbb{L}_q-$bounds, for $q \in [1,\infty)$. To achieve a MSE of order $\mathcal{O}(\epsilon^2)$, we require a cost of $\mathcal{O}(\epsilon^{-3})$, for $\epsilon>0$.  This is then extended to the multilevel setting, where we prove a similar result for an `ideal', or i.i.d., system of the MLEnKBF. In particular to achieve a MSE of order $\mathcal{O}(\epsilon^2)$, we require a cost of $\mathcal{O}(\epsilon^{-2}|\log(\epsilon)|)$, for $\epsilon>0$. The analysis is specific to the {vanilla} variant, which  will form our main result of this work and naturally follows from the analysis conducted in \cite{CJY20}.
\item We verify the analysis derived in the paper for various numerical experiments. We firstly demonstrate the rates attained on an Ornstein--Uhlenbeck process, and provide parameter estimation on the stochastic Lorenz 63 and Lorenz 96 models. Furthermore we test other variants of the EnKBF, where the analysis does not directly apply, such as the {deterministic} and {deterministic-transport} variants, to see how they perform.
\end{itemize}

\subsection{Outline}
The outline of this paper is presented as follows. In \autoref{sec:model} we review and discuss the KBF which will motivate the introduction of the EnKBF and the three variants we consider. This will lead onto \autoref{sec:NC} where we discuss normalizing constants associated to the filtering distribution, and how one can use the EnKBF and the MLEnKBF. Furthermore in this section we provide our main result and demonstrate it on a toy linear example. In \autoref{sec:param_estim} we provide an algorithm for parameter estimation using the MLEnKBF tested on various models arising in atmospheric sciences. We conclude with some remarks and suggest future directions of work in the \autoref{sec:conc}. Finally the proofs of our results are presented in the Appendix.

\section{Model and Background}
\label{sec:model}
In this section we provide an overview of the necessary background material required for the rest of the article. We begin by providing the continuous-time filtering problem through the Kalman--Bucy filters (KBF). This will lead onto a discussion of ensemble Kalman--Bucy filters (EnKBFs), which are $N-$particle representation of the KBF. Finally we will discuss the concept of multilevel Monte Carlo (MLMC), and review its extension to the EnKBF, referred to as the MLEnKBF. 
\subsection{Kalman--Bucy Filters}

Consider a linear-Gaussian filtering model of the following form
\begin{align}
\label{eq:data}
dY_t & =  CX_t dt + R^{1/2} dV_t, \\
\label{eq:signal}
dX_t & =  A X_t dt + Q^{1/2} dW_t,
\end{align}
where $(Y_t,X_t)\in\mathbb{R}^{d_y}\times\mathbb{R}^{d_x}$, $(V_t,W_t)$ is a $(d_y+d_x)-$dimensional standard Brownian motion, 
$A$ is a square $d_x\times d_x$ matrix, $C$ is a $d_y\times d_x$ matrix, $Y_0=0$, $X_0\sim\mathcal{N}_{d_x}(\mathcal{M}_0,\mathcal{P}_0)$
($d_x-$dimensional Gaussian distribution, mean $\mathcal{M}_0$, covariance matrix $\mathcal{P}_0$)
and $R^{1/2},Q^{1/2}$ are square (of the appropriate dimension) and symmetric and invertible matrices. It is well-known that, letting
$\{\mathscr{F}_t\}_{t\geq 0}$ be the filtration generated by the observations, the conditional probability of $X_t$ given $\mathscr{F}_t$ is a
Gaussian distribution with mean and covariance matrix 
$$
\mathcal{M}_t := \mathbb{E}[X_t|\mathscr{F}_t], \quad \mathcal{P}_t := \mathbb{E}\Big([X_t - \mathbb{E}(X_t|\mathscr{F}_t)][X_t - \mathbb{E}(X_t|\mathscr{F}_t)]^{\top}\Big),
$$
given by the Kalman--Bucy and Ricatti equations \cite{DT18}
\begin{align}
\label{eq:kbf}
d\mathcal{M}_t &= A \mathcal{M}_t dt + \mathcal{P}_tC^{\top}R^{-1}\Big(dY_t - C\mathcal{M}_t dt\Big), \\
\label{eq:ricc}
\partial_t\mathcal{P}_t &= \textrm{Ricc}(\mathcal{P}_t),
\end{align}
where the Riccati drift term is defined as
$$
\textrm{Ricc}(G) = AG + GA^{\top}-GSG + Q, \quad \textrm{with} \ S:=C^{\top}R^{-1}C.
$$
 A derivation of \eqref{eq:kbf} - \eqref{eq:ricc} can be found in \cite{AJ70}. KBF is viewed as the $\mathbb{L}_2$-optimal state estimator for an Ornstein--Uhlenbeck process, given the state is partially observed with linear and Gaussian assumptions. An alternative approach is to consider a conditional McKean-Vlasov type diffusion process. For this article we work with three different processes of the form
\begin{eqnarray}
d\overline{X}_t & = &A~\overline{X}_t~dt~+~Q^{1/2}~d\overline{W}_t+\Pa_{\eta_t}C^{\top}R^{-1}~\left[dY_t-\left(C\overline{X}_tdt+R^{1/2}~d\overline{V}_{t}\right)\right],\label{eq:vanilla_kbf}\\
d\overline{X}_t & = &A~\overline{X}_t~dt~+~Q^{1/2}~d\overline{W}_t+\Pa_{\eta_t}C^{\top}R^{-1}~\left[dY_t-\left(\frac{1}{2}C\left[\overline{X}_t+\eta_t(e)\right]dt\right)\right],
\label{eq:determ_kbf}\\
d\overline{X}_t & = &A~\overline{X}_t~dt~+~Q\Pa_{\eta_t}^{-1}\left(\overline{X}_t-\eta_t(e)\right)~dt+\Pa_{\eta_t}C^{\top}R^{-1}~\left[dY_t-\left(\frac{1}{2}C\left[\overline{X}_t+\eta_t(e)\right]dt\right)\right],
\label{eq:determ_transp_kbf}
\end{eqnarray}

where $(\overline{V}_t,\overline{W}_t,\overline{X}_0)$ are copies of the process of $(V_t,W_t,X_0)$ and covariance
$$
\mathcal{P}_{\eta_t} = \eta_t\Big([e-\eta_t(e)][e-\eta_t(e)]^{\top}\Big), \quad \eta_t:= \mathrm{Law}(\overline{X}_t|\mathscr{F}_t),
$$
such that $\eta_t$ is the conditional law of $\overline{X}_t$ given $\mathscr{F}_t$ and $e(x)=x$. We will explain the difference of each diffusion process, in succeeding subsections. It is important to note that the nonlinearity in \eqref{eq:vanilla_kbf}-\eqref{eq:determ_transp_kbf} does not depend on the distribution of the state $\mathrm{Law}(\overline{X}_t)$ but on the conditional distribution $\eta_t$, and $\mathcal{P}_{\eta_t}$ alone does not depend on $\mathscr{F}_t$. These processes are
commonly referred to as Kalman-Bucy (nonlinear) diffusion processes.
It is known that the conditional  expectations of the random states $\overline{X}_t$ and their conditional covariance matrices $\mathcal{P}_{\eta_t}$, with respect to $\mathscr{F}_t$, satisfy the Kalman--Bucy and the Riccati equations. In addition, for any $t\in\mathbb{R}^+$
$$
\eta_t:= \mathrm{Law}(\overline{X}_t|\mathscr{F}_t) = \mathrm{Law}({X}_t|\mathscr{F}_t).
$$
As a result, an alternative to recursively computing \eqref{eq:kbf} - \eqref{eq:ricc}, is to generate $N$ i.i.d.~samples from any of
\eqref{eq:vanilla_kbf}-\eqref{eq:determ_transp_kbf} processes and apply a Monte Carlo approximation, as mentioned which we now discuss.

\subsection{Ensemble Kalman--Bucy Filters}

Exact simulation from \eqref{eq:vanilla_kbf}-\eqref{eq:determ_transp_kbf} is typically not possible, or feasible, as one cannot compute $\mathcal{P}_{\eta_t}$ exactly. 
The ensemble Kalman--Bucy filter (EnKBF) can be used to deal with this issue. The EnKBF coincides with the mean-field particle interpretation of the Kalman-Bucy diffusion processes. EnKBFs is an $N-$particle system simulated for the $i^{th}-$particle, $i\in\{1,\dots,N\}$. The first variant of the EnKBF we consider is given as
\begin{equation}
d\xi_t^i  = A~\xi_t^i~dt~+~Q^{1/2}~d\overline{W}^i_t+P^N_tC^{\top}R^{-1}~\left[dY_t-\left(C\xi_t^idt+R^{1/2}~d\overline{V}^i_{t}\right)\right],
\label{eq:enkbf}
\end{equation}
which is known as the vanilla EnKBF (VEnKBF). This is the standard EnKBF used in theory and practice which contains perturbed observations, through the brownian motion $\overline{V}_t$.
It is a continuous-time derivation of the EnKF \cite{GE09}.
The second EnKBF given, defined as,
\begin{equation}
d\xi_t^i  = A~\xi_t^i~dt~+~Q^{1/2}~d\overline{W}^i_t+P^N_tC^{\top}R^{-1}~\left[dY_t-\left(\frac{1}{2}C\left[\xi_t^i+ \mathcal{M}_t \right]dt\right)\right],
\label{eq:denkbf}
\end{equation}
which is referred to as the deterministic EnKBF (DEnKBF), where unlike \eqref{eq:enkbf}, it contains no perturbed observations. This is the continuous-limiting object
of the deterministic EnKF defined in \cite{SO08}. It is well-known in ensemble data assimilation, that deterministic filters can perform better, partially due to the containing exact
observations with no noise. Our final variant of the EnKBF is the
deterministic transport EnKBF (DTEnKBF),
\begin{equation}
d\xi_t^i  = A~\xi_t^i~dt~+~Q P_t^{-1}\left(\xi_t^i-\mathcal{M}_t\right)~dt+P^N_t C^{\top}R^{-1}~\left[dY_t-\left(\frac{1}{2}C\left[\xi_t^i+\mathcal{M}_t\right]dt\right)\right],
\label{eq:dtenkbf}
\end{equation}
where the modification is that it does not contain  $\overline{V}_t$ and $\overline{W}_t$, implying it is completely deterministic. This filter, is motivated from the use of optimal transport
methodologies within data assimilation \cite{SR19}. Such processes are discussed in further detail in \cite{BD20}.
For all the variants of the EnKBF  the sample mean and covariances are defined as
\begin{align*}
{P}_t^N & =  \frac{1}{N -1}\sum_{i=1}^N (\xi_t^i-m^N_t)(\xi_t^i-m^N_t)^{\top},\\
m^N_t & =  \frac{1}{N }\sum_{i=1}^N \xi_t^i,
\end{align*}
where $\xi_0^i\stackrel{\textrm{i.i.d.}}{\sim}\mathcal{N}_{d_x}(\mathcal{M}_0,\mathcal{P}_0)$. Note that when $C=0$, \eqref{eq:enkbf} and \eqref{eq:denkbf} reduce to $N$--independent copies of an Ornstein--Uhlenbeck process. 


In practice, one will not have access to an entire trajectory of observations. Thus numerically, one often works with a time discretization, such as an Euler-type discretization.
Let $\Delta_l=2^{-l}$ denote our level of discretization, then we will generate the system for $(i,k)\in\{1,\dots,N\}\times\mathbb{N}_0=\mathbb{N}\cup\{0\}$ as

\begin{rem}
As mentioned in the introduction, the function $f$ in \eqref{eq:sig} can be nonlinear. If we assume in addition that $h$ is the identity function, then one can modify equations \eqref{eq:enkbf}-- \eqref{eq:dtenkbf} by replacing the term $A~\overline{X}_t$ with $f(\overline{X}_t)$ and equations \eqref{eq:enkbf1}-- \eqref{eq:dtenkbf1} by replacing the term $A~\xi_{k\Delta_l}^i$ with $f(\xi_{k\Delta_l}^i)$.
\end{rem}

\begin{rem}
When mentioning perturbed observations, we are referring to the additional noise in the innovation term. However, we emphasize it is not the observation being perturbed, but
rather the projected state. hence once can think of this as a stochastic filter.
\end{rem}
\begin{eqnarray}
   \xi_{(k+1)\Delta_l}^i & = & \xi_{k\Delta_l}^i + A  \xi_{k\Delta_l}^i  \Delta_l + Q^{1/2} \big\{\overline{W}_{(k+1)\Delta_l}^i-\overline{W}_{k\Delta_l}^i \big\} + \nonumber \\
   \label{eq:enkbf1}
 & & P^N_{k\Delta_l} C^{\top} R^{-1} \Big( \big\{Y_{(k+1)\Delta_l} - Y_{k\Delta_l} \big\} - \Big[ C \xi_{k\Delta_l}^i  \Delta_l + R^{1/2} \big \{\overline{V}_{(k+1)\Delta_l}^i - \overline{V}_{k\Delta_l}^i \big\} \Big] \Big),\\
 \xi_{(k+1)\Delta_l}^i & =&  \xi_{k\Delta_l}^i + A  \xi_{k\Delta_l}^i  \Delta_l + Q^{1/2} \big\{\overline{W}_{(k+1)\Delta_l}^i-\overline{W}_{k\Delta_l}^i \big\} + \nonumber \\
    \label{eq:denkbf1}
 & & P^N_{k\Delta_l} C^{\top} R^{-1} \left( \big\{Y_{(k+1)\Delta_l} - Y_{k\Delta_l} \big\} - C \left(\dfrac{ \xi_{k\Delta_l}^i + m^N_{k\Delta_l}}{2} \right)\Delta_l \right),\\
 \xi_{(k+1)\Delta_l}^i & = & \xi_{k\Delta_l}^i + A  \xi_{k\Delta_l}^i  \Delta_l + Q\left(P_{k\Delta_l}\right)^{-1} \left(\xi_{k\Delta_l}^i - m^N_{k\Delta_l}  \right) \Delta_l + \nonumber\\
     \label{eq:dtenkbf1}
 & & P^N_{k\Delta_l} C^{\top} R^{-1} \left( \big\{Y_{(k+1)\Delta_l} - Y_{k\Delta_l} \big\} - C \left(\dfrac{ \xi_{k\Delta_l}^i + m^N_{k\Delta_l}}{2} \right)\Delta_l \right),
\end{eqnarray}


such that
\begin{align*}
P_{k\Delta_l}^N & =  \frac{1}{N -1}\sum_{i=1}^N (\xi_{k\Delta_l}^i-m^N_{k\Delta_l})(\xi_{k\Delta_l}^i-m_{k\Delta_l}^N)^{\top}, \\
m_{k\Delta_l}^N & =  \frac{1}{N }\sum_{i=1}^N \xi_{k\Delta_l}^i,
\end{align*}
and $\xi_0^i\stackrel{\textrm{i.i.d.}}{\sim}\mathcal{N}_{d_x}(\mathcal{M}_0,\mathcal{P}_0)$. For $l\in\mathbb{N}_0$ given, denote by $\eta_t^{N,l}$ as the $N-$empirical
probability measure of the particles $(\xi_t^1,\dots,\xi_t^N)$, where $t\in\{0,\Delta_l,2\Delta_l,\dots\}$. For $\varphi:\mathbb{R}^{d_x}\rightarrow\mathbb{R}^{d_x}$
we will use the notation $\eta_t^{N,l}(\varphi):=\tfrac{1}{N}\sum_{i=1}^N\varphi(\xi_t^{i})$.

\begin{rem}
\label{rem:iid}
For ensemble based data assimilation, a common recurring assumption is that the particles are i.i.d.. While this is true, through ensemble filtering methodologies the assumption is practically false. This is based on the fact that ensemble covariance is computed from all ensemble members, which introduces a dependence, and in some cases the EnKF can be applied to nonlinear functions, ensuring the ensemble is not normally distributed. 
\end{rem}

\subsection{Multilevel EnKBFs}

Let us define $\pi$ to be be a probability on a measurable space $(\mathsf{X},\mathscr{X})$ and for $\pi-$integrable $\varphi:\mathsf{X}\rightarrow\mathbb{R}$
consider the problem of estimating $\pi(\varphi)=\mathbb{E}_{\pi}[\varphi(X)]$. Now let us assume that we only have access to a sequence of approximations of $\pi$,  $\{\pi_l\}_{l\in\mathbb{N}_0}$, also each defined on $(\mathsf{X},\mathscr{X})$ and we are now interested in estimating $\pi_l(\varphi)$, such that  $\lim_{l\rightarrow\infty}|[\pi_l-\pi](\varphi)|=0$. Therefore one can use the telescoping sum
\begin{equation}
\label{eq:sum}
\pi_L(\varphi) = \pi_0(\varphi) + \sum^L_{l=1}[\pi_l-\pi_{l-1}](\varphi),
\end{equation}
as we know that the approximation error between $\pi$ and $\pi_l$ gets smaller as $l \rightarrow \infty$. The idea of multilevel Monte Carlo (MLMC) \cite{MBG08,MBG15} is to construct a coupled system, related to the telescoping sum, such that the mean squared error can be reduced,  relative to i.i.d.~sampling from $\pi_L$. Therefore our MLMC approximation of $\mathbb{E}_{\pi_L}[\varphi(X)]$ is 
$$
\pi_L^{ML}(\varphi) := \frac{1}{N_0}\sum_{i=1}^{N_0}\varphi(X^{i,0}) + \sum_{l=1}^L \frac{1}{N_l}\sum_{i=1}^{N_l}\{\varphi(X^{i,l})-\varphi(\tilde{X}^{i,l-1})\},
$$
where $N_0\in\mathbb{N}$ i.i.d.~samples from $\pi_0$ as $(X^{1,0},\dots,X^{N_0,0})$ and for $l\in\{1,\dots,L\}$, $N_l\in\mathbb{N}$ samples
from a coupling of $(\pi_l,\pi_{l-1})$ as $((X^{1,l},\tilde{X}^{1,l-1}),\dots,(X^{N_l,l},\tilde{X}^{N_l,l-1}))$.
Then the MSE is then
$$
\mathbb{E}[(\pi_L^{ML}(\varphi)-\pi(\varphi))^2] = \underbrace{\mathbb{V}\textrm{ar}[\pi_L^{ML}(\varphi)]}_{\textrm{variance}} + 
[\underbrace{\pi_L(\varphi)-\pi(\varphi)}_{\textrm{bias}}]^2.
$$
It is important to note that the telescoping sum \eqref{eq:sum}, despite its simplicity and usefulness, is not an optimal choice to construct the multilevel estimator. In particular Giles \cite{MBG15} makes this assumption, but states that a potentially more robust weighted sum would be more optimal.
The work of Chada et al. \cite{CJY20} proposed the application of MLMC for the VEnKBF, where we will now briefly review this. If we consider the discretized VEnKBF \eqref{eq:enkbf1}, then, for  $l\in\mathbb{N}$ and $(i,k)\in\{1,\dots,N\}\times\mathbb{N}_0$ the ML adaption is given as

\[
\textbf{(F1)}
\begin{cases}
\begin{aligned}
\xi_{(k+1)\Delta_l}^{i,l} & = \xi_{k\Delta_l}^{i,l} + A\xi_{k\Delta_l}^{i,l}\Delta_l + Q^{1/2} [\overline{W}_{(k+1)\Delta_l}^i-\overline{W}_{k\Delta_l}^i]  \\
&+ P_{k\Delta_l}^{N,l}C^{\top}R^{-1}\Big([Y^{{i}}_{(k+1)\Delta_l}-Y^{{i}}_{k\Delta_l}]-\Big[C\xi_{k\Delta_l}^{i,l}\Delta_l + R^{1/2}[\overline{V}_{(k+1)\Delta_l}^i-\overline{V}_{k\Delta_l}^i]\Big]\Big),  \\
\xi_{(k+1)\Delta_{l-1}}^{i,l-1} & =  \xi_{k\Delta_{l-1}}^{i,l-1} + A\xi_{k\Delta_{l-1}}^{i,l-1}\Delta_{l-1} + Q^{1/2} [\overline{W}_{(k+1)\Delta_{l-1}}^i-\overline{W}_{k\Delta_{l-1}}^i] \\
 &+ P_{k\Delta_{l-1}}^{N,l-1}C^{\top}R^{-1}\Big([Y^{{i}}_{(k+1)\Delta_{l-1}}-Y^{{i}}_{k\Delta_{l-1}}] -\Big[C\xi_{k\Delta_{l-1}}^{i,l-1}\Delta_{l-1}+ R^{1/2}[\overline{V}_{(k+1)\Delta_{l-1}}^i-\overline{V}_{k\Delta_{l-1}}^i]\Big]\Big).
\end{aligned}
\end{cases}
\] \\
 Similarly for the deterministic variant \eqref{eq:denkbf1} we have
\[
\textbf{(F2)}
\begin{cases}
\begin{aligned} 
\xi_{(k+1)\Delta_l}^{i,l} & = \xi_{k\Delta_l}^{i,l} + A\xi_{k\Delta_l}^{i,l}\Delta_l + Q^{1/2} [\overline{W}_{(k+1)\Delta_l}^i-\overline{W}_{k\Delta_l}^i] \\
&+ P_{k\Delta_l}^{N,l}C^{\top}R^{-1}\left([Y^{{i}}_{(k+1)\Delta_l}-Y^{{i}}_{k\Delta_l}]-C \left(\dfrac{ \xi_{k\Delta_l}^{i,1} + m^{N,l}_{k\Delta_l}}{2} \right)\Delta_l \right),  \\
\xi_{(k+1)\Delta_{l-1}}^{i,l-1} & =  \xi_{k\Delta_{l-1}}^{i,l-1} + A\xi_{k\Delta_{l-1}}^{i,l-1}\Delta_{l-1} + Q^{1/2} [\overline{W}_{(k+1)\Delta_{l-1}}^i-\overline{W}_{k\Delta_{l-1}}^i] \\
 &+ P_{k\Delta_{l-1}}^{N,l-1}C^{\top}R^{-1}\left([Y^{{i}}_{(k+1)\Delta_{l-1}}-Y^{{i}}_{k\Delta_{l-1}}] -C \left(\dfrac{ \xi_{k\Delta_{l-1}}^{i,l-1} + m^{N,l-1}_{k\Delta_{l-1}}}{2} \right)\Delta_{l-1} \right),
\end{aligned}
\end{cases}
\] \\
and finally for the  deterministic-transport variant \eqref{eq:dtenkbf1}
\[
\textbf{(F3)}
\begin{cases}
\begin{aligned}
  \xi_{(k+1)\Delta_l}^{i,l} & =  \xi_{k\Delta_l}^{i,l} + A  \xi_{k\Delta_l}^{i,l}  \Delta_{l} + Q \left(P_{k\Delta_l}\right)^{-1} [\xi_{k\Delta_l}^{i,l} - m_{k\Delta_l}^{N,l} ] \Delta_l  \\
&+ P_{k\Delta_l}^{N,l}C^{\top}R^{-1}\left([Y^{{i}}_{(k+1)\Delta_l}-Y^{{i}}_{k\Delta_l}]-C \left(\dfrac{ \xi_{k\Delta_l}^{i,l} + m^{N,l}_{k\Delta_l}}{2} \right)\Delta_l \right),  \\
\xi_{(k+1)\Delta_{l-1}}^{i,l-1} & =  \xi_{k\Delta_{l-1}}^{i,l-1} + A\xi_{k\Delta_{l-1}}^{i,l-1}\Delta_{l-1} + Q\left(P_{k\Delta_{l-1}}\right)^{-1} [\xi_{k\Delta_{l-1}}^{i,l-1}-m_{k\Delta_{l-1}}^{N,l-1}] \Delta_{l-1} \\
  &+ P_{k\Delta_{l-1}}^{N,l-1}C^{\top}R^{-1}\left([Y^{{i}}_{(k+1)\Delta_{l-1}}-Y^{{i}}_{k\Delta_{l-1}}] -C \left(\dfrac{ \xi_{k\Delta_{l-1}}^{i-1} + m_{k\Delta_{l-1}}}{2} \right)\Delta_{l-1} \right),
\end{aligned}
\end{cases}
\] \\
where our sample covariances and means are defined accordingly as
\begin{align*}
P_{k\Delta_l}^{N,l} & =  \frac{1}{N -1}\sum_{i=1}^N (\xi_{k\Delta_l}^{i,l}-m_{k\Delta_l}^{N,l})(\xi_{k\Delta_l}^{i,l}-m_{k\Delta_l}^{N,l})^{\top} ,\\
m_{k\Delta_l}^{N,l} & =  \frac{1}{N }\sum_{i=1}^N \xi_{k\Delta_l}^{i,l},\\
P_{k\Delta_{l-1}}^{N,l-1} & =  \frac{1}{N -1}\sum_{i=1}^N (\xi_{k\Delta_{l-1}}^{i,l-1}-m_{k\Delta_{l-1}}^{N,l-1})(\xi_{k\Delta_{l-1}}^{i,l-1}-m_{k\Delta_{l-1}}^{N,l-1})^{\top}, \\
m_{k\Delta_{l-1}}^{N,l-1} & =  \frac{1}{N }\sum_{i=1}^N \xi_{k\Delta_{l-1}}^{i,l-1},
\end{align*}
and $\xi_0^{i,l}\stackrel{\textrm{i.i.d.}}{\sim}\mathcal{N}_{d_x}(\mathcal{M}_0,\mathcal{P}_0)$, $\xi_{0}^{i,l-1}=\xi_0^{i,l}$.
Then, one has the approximation of $[\eta_t^l-\eta_t^{l-1}](\varphi)$, $t\in\mathbb{N}_0$, $\varphi:\mathbb{R}^{d_x}\rightarrow\mathbb{R}$, given as
$$
[\eta_t^{N,l} -\eta_t^{N,l-1}](\varphi) = \frac{1}{N}\sum_{i=1}^N[\varphi(\xi_t^{i,l})-\varphi(\xi_t^{i,l-1})].
$$
Therefore for the multilevel estimation, one has an approximation for $t\in\mathbb{N}_0$
\begin{equation}\label{eq:main_est}
\eta_t^{ML}(\varphi):=\eta_t^{N_0,0}(\varphi) + \sum_{l-1}^L [\eta_t^{N_l,l} -\eta_t^{N_l,l-1}](\varphi).
\end{equation}
Similarly with \autoref{rem:iid}, the i.i.d. assumption regarding each particle does not also hold in the multilevel setting. Largely
due to the correlation of the particle in the multilevel setting.
\begin{rem}
For the above forms of the MLEnKBF, only (\textbf{F1}) - (\textbf{F2}) were introduced and numerically tested in \cite{CJY20}. 
Therefore this article aims to further test the deterministic-transport variant of \textbf{(F3)}.
\end{rem}

\section{Normalizing Constant Estimation}
\label{sec:NC}

In this section we introduce the notion of normalizing constants (NC), which is what we are concerned with estimating. We begin by recalling the NC estimator using the EnKBF, aiming to show its cost associated to attain an MSE of a particular order. We then present our NC estimator through the MLEnKBF. This will lead onto our main result of the paper, which is the error-to-cost ratio of our normalizing constant estimator, i.e. the MLEnKBF, which is presented through \autoref{theo:main}.

\subsection{Normalizing Constants}

We begin by defining the normalizing constant associated with the filtering distribution. This is defined as $Z_t:=\frac{{\mathcal L}_{X_{0:t},Y_{0:t}}}{{\mathcal L}_{X_{0:t},W_{0:t}}}$ to be the density of ${\mathcal L}_{X_{0:t},Y_{0:t}}$, the law of the process $(X,Y)$ and that of ${\mathcal L}_{X_{0:t},W_{0:t}}$, the law of the process $(X,W)$. That is, 
\[
{\mathbb E}[f(X_{0:t})g(Y_{0:t})]= {\mathbb E}[f(X_{0:t})g(W_{0:t})Z_{t}(X,Y)].
\]
One can show that (see Exercise 3.14 pp 56 in \cite{BC09})
\begin{equation}
\label{eq:initt}
Z_t(X,Y)=\exp{\left[\int_0^t \left[\langle CX_s, R^{-1} dY_s\rangle -\frac{1}{2}\langle X_s,SX_s\rangle ~ds\right]\right]},
\end{equation}
recalling that  $S:=C^{\top}R^{-1}C$. Now we let $\overline{Z}_t(Y)$ denote the likelihood function defined by
$$
\overline{Z}_t(Y):=\EE_Y\left(Z_t(X,Y)\right),
$$
where $\EE_Y\left(\cdot\right)$ stands for the expectation w.r.t. the signal process when the observation is fixed and independent of the signal.
From the work of Crisan et al. \cite{CDJ21}, the authors show that the normalizing constant is given by
\begin{align}
\label{eq:nc}
\overline{Z}_t(Y)=\exp{\left[\int_0^t 
\left[~\langle C\mathcal{M}_s, R^{-1} dY_s\rangle -\frac{1}{2}\langle \mathcal{M}_s,S\mathcal{M}_s\rangle~ds\right]
\right]},
\end{align}
and a sensible estimator of it is given as
\begin{align}
\label{eq:nc_estimate}
\overline{Z}_t^N(Y)=\exp{\left[\int_0^t 
\left[~\langle Cm_s^N, R^{-1} dY_s\rangle -\frac{1}{2}\langle m_s^N,Sm_s^N\rangle~ds\right]
\right]},
\end{align}
which follows from replacing the conditional mean of the signal process  with the sample mean associated with the EnKBF. To briefly describe how
\eqref{eq:nc} was attained in \cite{CDJ21}, the authors used a change of measure rule, through Girsanov's Theorem, then an application of Bayes'
Theorem which can be expressed through as the Kallianpur-Striebel formula defined on path space. The derivation is a standard approach, which 
is discussed in  Chapter 3 in \cite{BC09}.
\\\\
In practice, one must time-discretize the EnKBF, and the normalizing constant estimator \eqref{eq:nc_estimate} to yield for $t\in\mathbb{N}$ 
\begin{align}
\label{eq:disc_nc_estimate}
\overline{Z}_t^{N,l}(Y)=\exp\Bigg\{\sum_{k=0}^{t\Delta_l^{-1}-1} 
\left[~\langle Cm_{k\Delta_l}^N, R^{-1} [Y_{(k+1)\Delta_l}-Y_{k\Delta_l}]\rangle -\frac{\Delta_l}{2}\langle m_{k\Delta_l}^N,Sm_{k\Delta_l}^N\rangle\right]
\Bigg\}. 
\end{align}
Let $\overline{U}_t^{N,l}(Y)=\log(\overline{Z}_t^{N,l}(Y))$, where we now consider the estimation of log-normalization constants. To enhance the efficiency we consider a coupled ensemble Kalman--Bucy filter, as described in Section \ref{sec:model}. Let $l\in\mathbb{N}$ then we run the coupled system of the different MLEnKBFs. Our multilevel estimator of log-normalizing constant given in \autoref{alg:MLEnKBF_NC}.


\begin{rem}
In order to use the analysis derived in \cite{CJY20}, for the normalizing constant estimation, our results will be specific to the  vanilla MLEnKBF variant $\mathrm{\textbf{(F1)}}$. This is important
as we will use the results presented in \cite{CJY20}, however, in terms of the numerics, we will use all ML variants described for various parameter estimation experiments. 
\end{rem}

\subsection{Single-level EnKBF}
To present a multilevel NC estimator using the EnKBF, we first require to understand the single-level EnKBF. To aid our analysis we will consider the i.i.d.~particle system based upon the Euler discretization of EnKBF for \textbf{(F1)}, such that for $(i,k)\in\{1,\dots,N\}\times\mathbb{N}_0$
\begin{align}\label{eq:enkf_is}
\zeta_{(k+1)\Delta_l}^i &=(I+A\Delta_l)\zeta_{k\Delta_l}^i + Q^{1/2} [\overline{W}_{(k+1)\Delta_l}^i-\overline{W}_{k\Delta_l}^i]  \\
&+ P_{k \Delta_l} C^{\top}R^{-1}\Big([Y_{(k+1)\Delta_l}-Y_{k\Delta_l}] -
\Big[C\zeta_{k\Delta_l}^i\Delta_l+R^{1/2}[\overline{V}_{(k+1)\Delta_l}^i-\overline{V}_{k\Delta_l}^i]\Big]\Big), \nonumber
\end{align}
such that  $\zeta_{(k+1)\Delta_l}^i|\mathscr{F}_{(k+1)\Delta_l}\stackrel{\textrm{i.i.d.}}{\sim}\mathcal{N}_{d_x}(m_{(k+1)\Delta_l},P_{(k+1)\Delta_l})$, where the moments are
defined as
\begin{align}
m_{(k+1)\Delta_l} & =  m_{k\Delta_l} + Am_{k\Delta_l}\Delta_l + P_{k \Delta_l} C^{\top}R^{-1}\Big(
[Y_{(k+1)\Delta_l}-Y_{k\Delta_l}] -Cm_{k\Delta_l}\Delta_l
\Big),\label{eq:iid_mean_rec}\\
P_{(k+1)\Delta_l} & =  P_{k\Delta_l} + \textrm{Ricc}(P_{k\Delta_l})\Delta_l +  (A-P_{k\Delta_l}S)P_{k\Delta_l}(A^{\top}-SP_{k\Delta_l})\Delta_l^2,\label{eq:iid_cov_rec}
\end{align}
which are satisfied by the Kalman--Bucy diffusion \eqref{eq:vanilla_kbf}.
We now consider the NC estimator $\widehat{\overline{U}}_{t+k_1\Delta_l}^{N,l}$, which instead using the recursion of \eqref{eq:enkf_is}.

\begin{rem}
Our reason for introducing the new estimator, based on \eqref{eq:enkf_is} - \eqref{eq:iid_cov_rec}, is that the usual described estimator of 
$\overline{U}_T^{ML}(Y)$ will not hold, in terms of the analysis, due to the recursion of the EnKBF, which is mentioned and documented in \cite{CJY20}. However, 
as we will see, by using a number of arguments one can show that both the multilevel estimators are close in probability, as $N \rightarrow \infty$.
 Therefore our analysis will be specific to  ML NC estimator $\widehat{\overline{U}}_T^{ML}(Y)$ i.e.
$$
\widehat{\overline{U}}_T^{ML}(Y)= \widehat{\overline{U}}_T^{N_{l_*},l_*}(Y) + \sum_{l=l_*+1}^L \{\widehat{\overline{U}}_T^{N_l,l}(Y)-\widehat{\overline{U}}_T^{N_l,l-1}(Y)\},
$$
 while we continue to use $\overline{U}_T^{ML}(Y)$ for the numerical examples.
\end{rem}

 We now present our first result of the paper which is a propagation of chaos result for the single-level vanilla EnKBF NC estimator.
For a $d_x-$dimensional vector $x$ we denote $\|x\|_2=(\sum_{j=1}^{d_x}x(j)^2)^{1/2}$, where $x(j)$ is the $j^{th}-$element of $x$. 
We use the notation $[{\overline{U}}_{t+k_1\Delta_l}^{N,l}-{\overline{U}}_{t+k_1\Delta_l}^l](Y)={\overline{U}}_{t+k_1\Delta_l}^{N,l}(Y)-{\overline{U}}_{t+k_1\Delta_l}^l(Y)$.
\begin{proposition}
\label{prop:var_term1_sec_state}
{For any $(l,t,k_1)\in\mathbb{N}_0\times\mathbb{R}^+\{0,1,\dots, \Delta_l^{-1}\}$ almost surely:
$$
 \lim_{N\rightarrow\infty}[{\overline{U}}_{t+k_1\Delta_l}^{N,l}-{\overline{U}}_{t+k_1\Delta_l}^l](Y) = 0.
$$}
\end{proposition}
\begin{proof}
{The proof follows by using Proposition \ref{lem:sl_nc} in the appendix and the Marcinkiewicz--Zygmund inequality for i.i.d.~random variables, along with a standard first Borel--Cantelli lemma argument.}
\end{proof}
Then using the same arguments, as described in \cite{CJY20}, one can show, through the Markov inequality that
\begin{equation}\label{eq:prob_est}
\mathbb{P}\left(\left|{\overline{U}}_{t+k_1\Delta_l}^{N,l}(Y)-\widehat{\overline{U}}_{t+k_1\Delta_l}^{N,l}(Y)\right|>\varepsilon\right) \leq \frac{\mathsf{C}}{\varepsilon^{2q} N^{q/2}}.
\end{equation}
for any $\varepsilon>0$ and $q>0$, where $\mathsf{C}$ is a constant that depends on $(l,q,t,k_1)$ but not $N$. Related to this in  Proposition \ref{prop:var_term1}, in the appendix, we have shown that:
$$
\mathbb{E}\Big[\Big\|[\widehat{\overline{U}}_{t+k_1\Delta_l}^{N,l}-{\overline{U}}_{t+k_1\Delta_l}](Y)\Big\|_2^2\Big] \leq \mathsf{C}\Big(\frac{1}{N}+\Delta_l^2\Big),
$$
where $\mathsf{C}$ does not depend upon $l$ nor $N$. Thus to achieve an MSE of order $\mathcal{O}(\epsilon^2)$, for $\epsilon>0$, we require a cost of  $\mathcal{O}(\epsilon^{-3})$, under a suitable choice of $L$ and $N$, using the single-level EnKBF \cite{MBG08}.

\subsection{Multilevel EnKBF}
Now to analyze the multilevel estimator, we will aim to prove results for the i.i.d. (ideal) coupled system for $(i,k)\in\{1,\dots,N\}\times\mathbb{N}_0$:
\begin{align}
 \label{eq:iid1}
\zeta_{(k+1)\Delta_l}^{i,l} & = \zeta_{k\Delta_l}^{i,l} + A\zeta_{k\Delta_l}^{i,l}\Delta_l + Q^{1/2} [\overline{W}_{(k+1)\Delta_l}^i-\overline{W}_{k\Delta_l}^i]  
\\&+ P^{N,l}_{k \Delta_l} C^{\top}R^{-1}\Big([Y^{{i}}_{(k+1)\Delta_l}-Y^{{i}}_{k\Delta_l}]
-\Big[C\zeta_{k\Delta_l}^{i,l}\Delta_l + R^{1/2}[\overline{V}_{(k+1)\Delta_l}^i-\overline{V}_{k\Delta_l}^i]\Big]\Big),\nonumber \\
 \label{eq:iid2}
\zeta_{(k+1)\Delta_{l-1}}^{i,l-1} & =  \zeta_{k\Delta_{l-1}}^{i,l-1} + A\zeta_{k\Delta_{l-1}}^{i,l-1}\Delta_{l-1} + Q^{1/2} [\overline{W}_{(k+1)\Delta_{l-1}}^i-\overline{W}_{k\Delta_{l-1}}^i] 
\\&+ P^{N,l-1}_{k \Delta_{l-1}} C^{\top}R^{-1}\Big([Y^{{i}}_{(k+1)\Delta_{l-1}}-Y^{{i}}_{k\Delta_{l-1}}] -\Big[C\zeta_{k\Delta_{l-1}}^{i,l-1}\Delta_{l-1}+ R^{1/2}[\overline{V}_{(k+1)\Delta_{l-1}}^i-\overline{V}_{k\Delta_{l-1}}^i]\Big]\Big), \nonumber
\end{align}
Now we set
\begin{equation}\label{eq:main_est_iid}
\widehat{\overline{U}}_t^{ML}(Y):=\widehat{\overline{U}}_t^{N_0,0}(Y) + \sum_{l-1}^L \big[\widehat{\overline{U}}_t^{N_l,l} -\widehat{\overline{U}}_t^{N_l,l-1}\big](Y),
\end{equation}
Therefore using the same argument as in \cite{CJY20}, one can establish, almost surely
$$
\lim_{\min_l N_l\rightarrow\infty}[{\overline{U}}_t^{ML}-\widehat{\overline{U}}_t^{ML}](Y) = 0, \quad \mathbb{P}\left(\left|[{\overline{U}}_t^{ML}-\widehat{\overline{U}}_t^{ML}](Y)\right|>\varepsilon\right) \leq \frac{\mathsf{C}}{\varepsilon^{2q}}\left(\sum_{l=0}^L \frac{1}{N_l^{q/2}}\right),
$$
where $\mathsf{C}$ is a constant that can depend on $(L,q,t)$ but not $N_{0:L}$.
\\\\
We are now in a position to state our main result, concerned with using the vanilla MLEnKBF, for the i.i.d. system, as a NC estimator. 
\begin{theorem}
\label{theo:main}
For any $T\in\mathbb{N}$ fixed and $t\in[0,T-1]$ there exists a $\mathsf{C}<+\infty$ such that for any $(L,N_{0:L})\in\mathbb{N}\times\{2,3,\dots\}^{L+1}$,
$$
\mathbb{E}\left[\left\|[\widehat{\overline{U}}_t^{ML}-{\overline{U}}_t](Y)\right\|_2^2\right] \leq \mathsf{C}\left(
\sum_{l=0}^L \frac{\Delta_l}{N_l} + \sum_{l=1}^L\sum_{q=1, q\neq l}^L\frac{\Delta_l\Delta_q}{N_lN_q} + \Delta_L^2
\right).
$$
The above theorem translates as, in order to achieve an MSE of order $\mathcal{O}(\epsilon^2)$, for $\epsilon>0$, we have a cost of $\mathcal{O}(\epsilon^{-2} ~\log(\epsilon)^2)$. This implies a reduction in cost compared to the single-level NC estimator.
\end{theorem}

\begin{rem}
It is important to note that the rates obtained in \autoref{prop:var_term1_sec_state} and \autoref{theo:main} for the normalizing constants,
 are identical to those derived in \cite{CJY20}. However despite this it should be emphasized as we will see in the appendix, that the proofs are not trivial, and all 
 of them do not directly follow exactly.
 \end{rem}

\begin{algorithm}[h!]
\caption{(\textbf{MLEnKBF-NC}) Multilevel Estimation of Normalizing Constants}
\label{alg:MLEnKBF_NC}
\begin{enumerate}
\item \textbf{Input:} Target level $L\in\mathbb{N}$, start level $l_*\in \mathbb{N}$ such that $l_*<L$, the number of particles on each level $\{N_l\}_{l=l_*}^L$, the time parameter $T\in\mathbb{N}$ and initial independent ensembles $\Big\{\{\tilde{\xi}_0^{i,l_*}\}_{i=1}^{N_{l_*}}, \cdots, \{\tilde{\xi}_0^{i,L}\}_{i=1}^{N_{L}}\Big\}$.
\item \textbf{Initialize:} Set $l=l_*$. For $(i,k)\in \{1,\cdots,N_l\}\times \{0,\cdots,T\Delta_l^{-1}-1\}$, set $\{\xi_0^{i,l}\}_{i=1}^{N_l} = \{\tilde{\xi}_0^{i,l}\}_{i=1}^{N_l}$ and simulate any of the cases \eqref{eq:enkbf1}-\eqref{eq:dtenkbf1} to return $\{m^{N_l,l}_{k\Delta_l}\}_{k=0}^{T\Delta_l^{-1}-1}$. Then using \eqref{eq:disc_nc_estimate}, return $\overline{Z}_T^{N_l,l}(Y)$ or $\overline{U}_T^{N_l,l}(Y)$.
\item \textbf{Iterate:} For $l \in \{l_*+1,\cdots,L\}$ and $(i,k)\in \{1,\cdots,N_l\} \times \{0,\cdots,T\Delta_l^{-1}-1\}$, set $\{\xi_0^{i,l-1}\}_{i=1}^{N_l}=\{\xi_0^{i,l}\}_{i=1}^{N_l}= \{\tilde{\xi}_0^{i,l}\}_{i=1}^{N_l}$, and simulate any of the coupled ensembles \textbf{(F1)}-\textbf{(F3)} (the one which corresponds to the case used in Step 2.) to return $\{m^{N_l,l-1}_{k\Delta_{l-1}}\}_{k=0}^{T\Delta_{l-1}^{-1}-1}$ and $\{m^{N_l,l}_{k\Delta_l}\}_{k=0}^{T\Delta_l^{-1}-1}$. Then using \eqref{eq:disc_nc_estimate}, return $\overline{Z}_T^{N_l,l-1}(Y)$ \& $\overline{Z}_T^{N_l,l}(Y)$ or $\overline{U}_T^{N_l,l-1}(Y)$ \& $\overline{U}_T^{N_l,l}(Y)$.
\item \textbf{Output:} Return the multilevel estimation of the normalizing constant:
\begin{align}
\label{eq:MLEnKBF_NC}
\overline{Z}_T^{ML}(Y) = \overline{Z}_T^{N_{l_*},l_*}(Y) + \sum_{l=l_*+1}^L \{\overline{Z}_T^{N_l,l}(Y)-\overline{Z}_T^{N_l,l-1}(Y)\}.
\end{align}
or its logarithm:
\begin{align}
\label{eq:MLEnKBF_Log_NC}
\overline{U}_T^{ML}(Y) = \overline{U}_T^{N_{l_*},l_*}(Y) + \sum_{l=l_*+1}^L \{\overline{U}_T^{N_l,l}(Y)-\overline{U}_T^{N_l,l-1}(Y)\}.
\end{align}
\end{enumerate}
\end{algorithm}

\section{Numerical simulations}
\label{sec:param_estim}
In this section we provide various numerical experiments, which include both verifying the ML rates obtained, and for parameter estimation on both linear and nonlinear models.
The former will be tested on an linear example, which the later will include a toy linear Gaussian example, a stochastic Lorenz 63 model and  Lorenz 96 model, which are common models 
 that arise in atmospheric sciences. Our parameter estimation will be conducted through a combination of using recursive maximum 
 likelihood estimation, and simultaneous perturbation stochastic approximation, which is a gradient-free methodology.

\subsection{Verification of multilevel rates}

We now seek to verify our theory from \autoref{prop:var_term1_sec_state} and \autoref{theo:main} on an Ornstein--Uhlenbeck process,
taking the form of  \eqref{eq:data} - \eqref{eq:signal}. We will compare the error-to-cost rates of the MLEnKBF estimation versus the EnKBF estimation of log-normalizing constant. This will be numerically tested using all variants \textbf{(F1)} - \textbf{(F3)}, and their respective multilevel counterparts. 
We will take $d_x=d_y=5$, $A=-0.8~Id$, $\mathcal{M}_0=0.1~\textbf{1}$, 
$\mathcal{P}_0=0.05~Id$, $R^{1/2}=2~Id$, $C$ to be a random matrix and $Q^{1/2}$ is a tri-diagonal matrix defined as 
$$ Q^{1/2} = 
\begin{bmatrix}
2/3  & 1/3     &       &     \hdots    & 0   \\
1/3  & 2/3 & \ddots &  &         \\
 & \ddots & \ddots & \ddots &    \\
\vdots &        &  & \ddots & 1/3 \\
0 &   \hdots    &       & 1/3  & 2/3
\end{bmatrix},
$$
where $Id$ is the identity matrix of appropriate dimension and $\textbf{1}$ is a vector of ones. We computed the MSE for the target levels $L\in\{7,8,9,10\}$ using 416 simulations with $t=1$. In order to better estimate the quantity in \eqref{eq:disc_nc_estimate}, we used $l_*=6$. The number of samples on each level is given by
\begin{align}
\label{eq:num_of_samples}
N_l = \floor*{C_0~ 2^{2L-l} ~(L-l_*+1)}.
\end{align}
The reference value is the mean of 208 simulations computed at a discretization level $l=11$. For the NC estimator, using the EnKBF, the cost of the estimator is $N_{L}\Delta_L^{-1}$.
When using the MLEnKBF-NC, the cost of the estimator is $ \sum_{l=l_*}^LN_{l}\Delta_l^{-1}$.
\begin{figure}[!htb]
	\begin{subfigure}[c]{0.49\textwidth}
	\includegraphics[width=1.\textwidth]{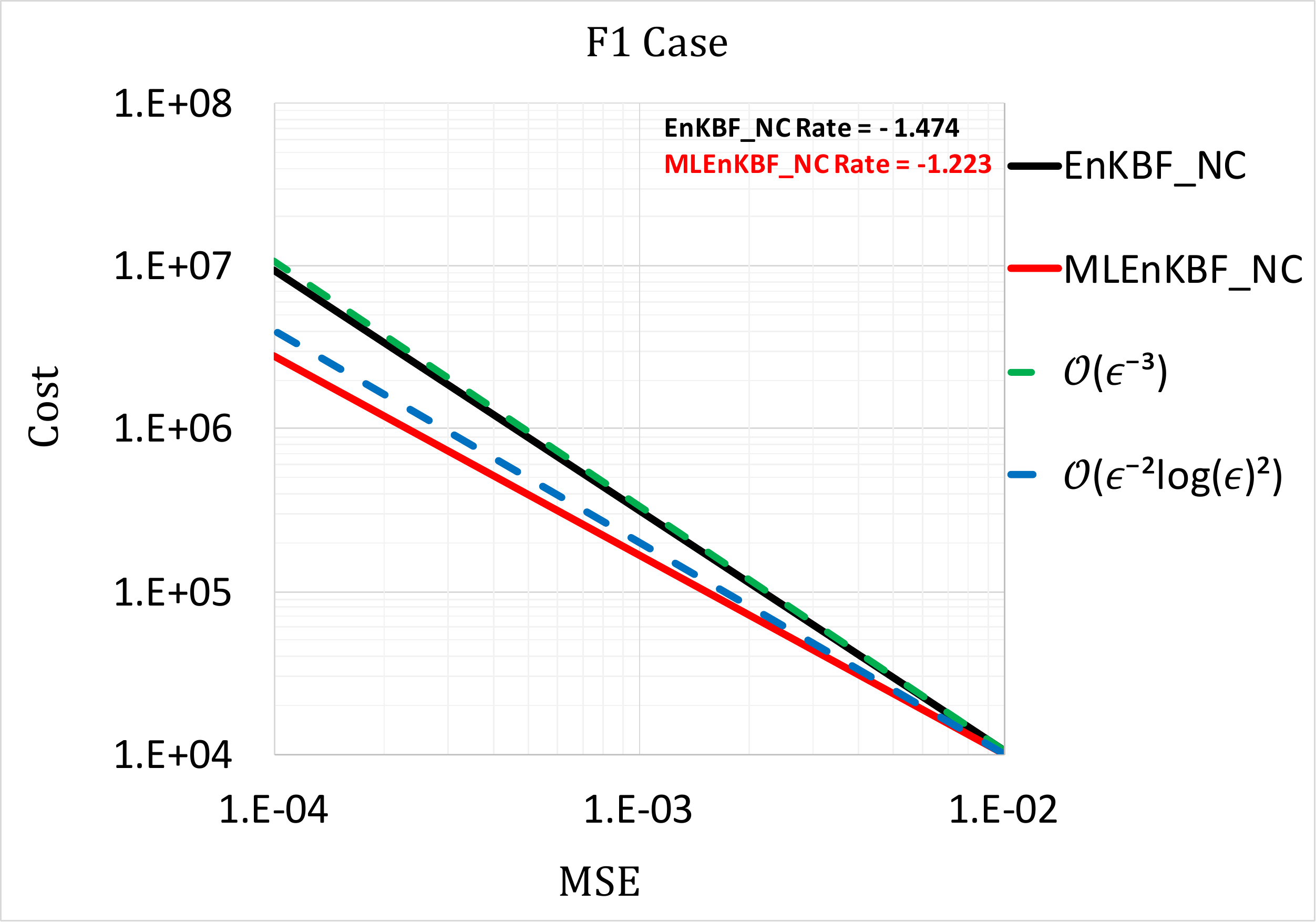}
	\end{subfigure}
	\begin{subfigure}[c]{0.5\textwidth}
	\includegraphics[width=1.\textwidth]{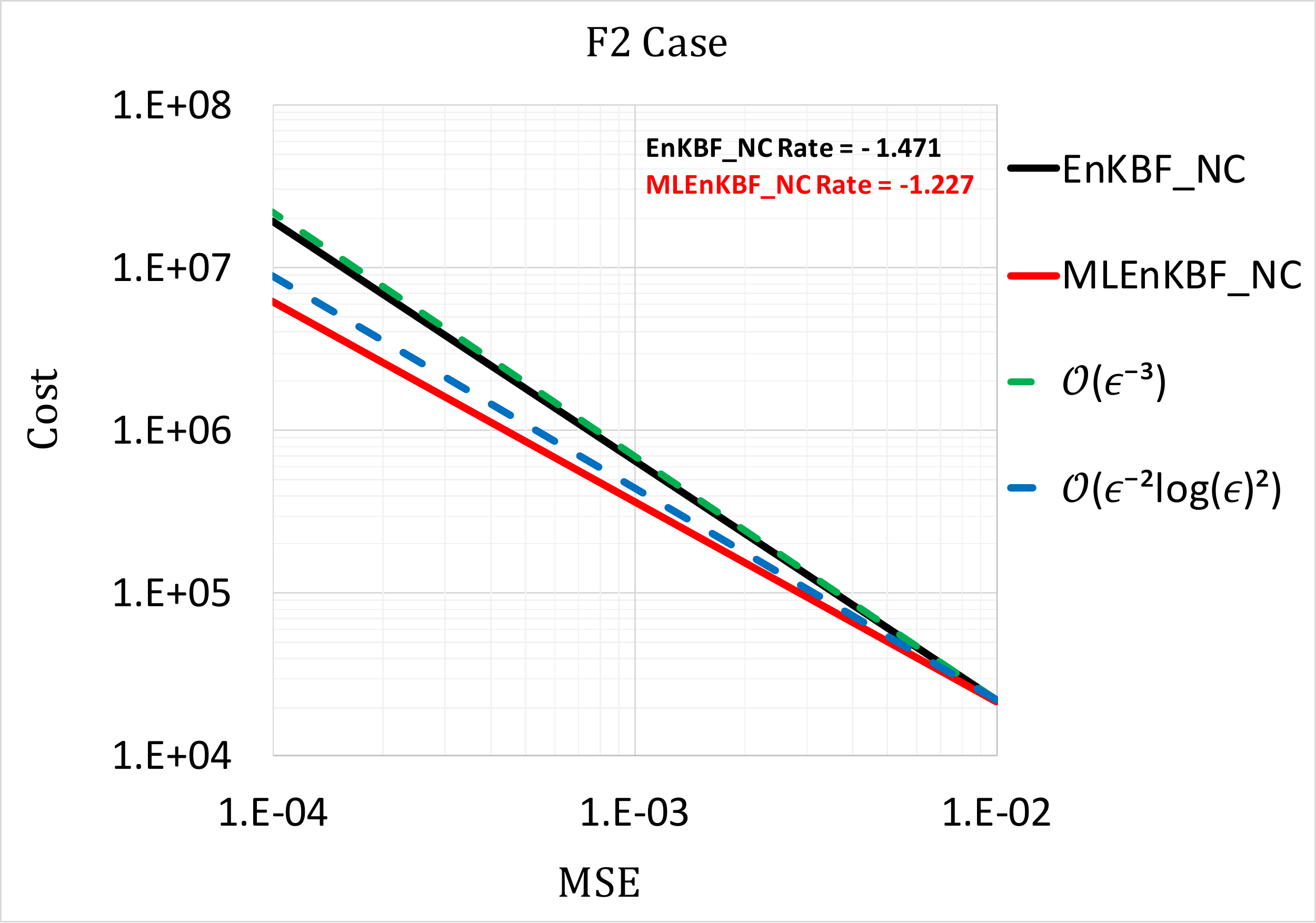}
	\end{subfigure}
	\\ \centering
	\begin{subfigure}[c]{0.5\textwidth}
	\includegraphics[width=1.\textwidth]{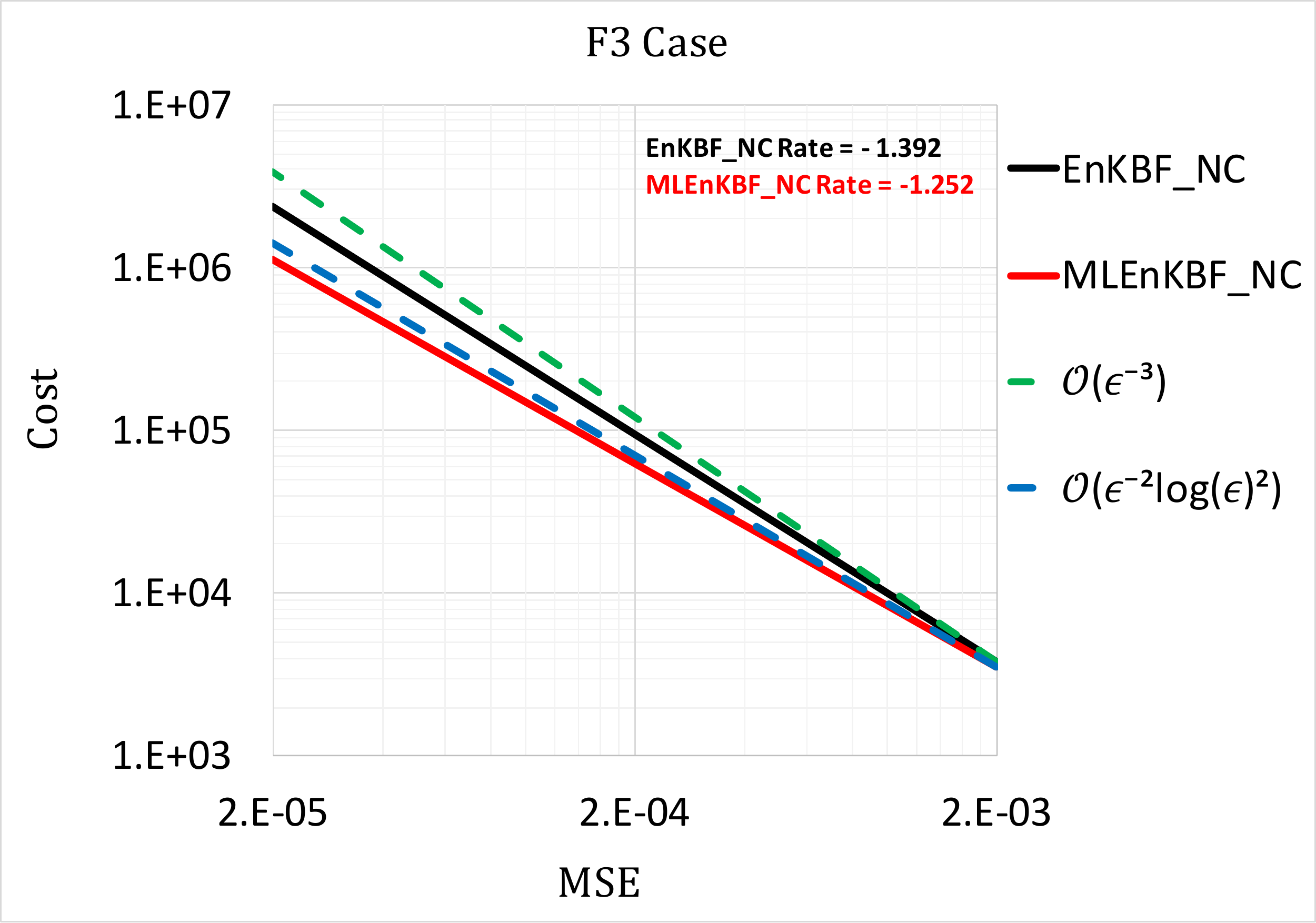}
	\end{subfigure}
    \caption{MSE vs Cost plots on a log-log scale for the linear Gaussian model \eqref{eq:data} - \eqref{eq:signal}, comparing EnKBF and MLEnKBF estimators of the normalizing constant.}
    \label{fig:rates}
\end{figure} 
Our numerical results are presented in \autoref{fig:rates}, where we have plotted the ML and single-level (SL) rates for each variant, within each subplot. Let us first consider \textbf{F(1)}, which is the vanilla MLEnKBF. The results obtained match that from both \autoref{prop:var_term1_sec_state} and \autoref{theo:main}, which suggest error-to-cost rates of $\mathcal{O}(\epsilon^{-3})$ and $\mathcal{O}(\epsilon^{-2}\log(\epsilon)^2)$. This is indicated through the slopes, which show that to attain an order of MSE, the MLEnKBF-NC is computationally cheaper than that of the EnKBF-NC. Our next subplot concerns \textbf{F(2)}. Despite no existing theory for this case, or its SL counterpart, we see that the results are similar to the vanilla variant. Specifically, rates given for both slopes are similar, where there is little
distinguishment. Finally for the final variant of \textbf{F(2)}, the results obtained are quite different. Firstly we notice the slopes are more different to the other variants, but also that the theory of the rate $\mathcal{O}(\epsilon^{-3})$ does not hold as convincingly. This could be related to the lack of stochasticity within the methodology. However from all subplots, one can see that by applying MLMC, one does attain a lower cost for a particular order of MSE.

\subsection{Parameter estimation}
 
Let us firstly assume that the model \eqref{eq:dat} - \eqref{eq:sig} contains unknown parameters $\theta\in \Theta \subseteq \mathbb{R}^{d_\theta}$. To account for this, we rewrite the normalizing constant \eqref{eq:nc} with the additional subscript $\theta$ as $\overline{Z}_{t,\theta}(Y)$. To estimate these parameters we focus on maximum likelihood inference and stochastic gradient
methods that are performed in an online manner. Mainly, we follow a recursive maximum likelihood (RML) method, which has been proposed originally in \cite{AM90}, in \cite{LM97} for finite spaces, and in \cite{BCJKR20, DDS10, PDS11} in the context of sequential Monte Carlo (SMC) approximations. Let 
$$
U_{t,t+1,\theta}(Y) := \log \frac{\overline{Z}_{t+1,\theta}(Y) }{\overline{Z}_{t,\theta}(Y) }.
$$
RML relies on the following update scheme at any time $t\in \mathbb{N}$:
\begin{align*}
\theta_{t+1} &= \theta_t + a_t \left(  \nabla_\theta \log \overline{Z}_{t+1,\theta_t}(Y)  - \nabla_\theta \log \overline{Z}_{t,\theta_t}(Y)  \right) \\
&= \theta_t + a_t \nabla_\theta~ U_{t,t+1,\theta_t}(Y),
\end{align*}
where $\{a_t\}_{t\in\mathbb{N}}$ is a sequence of positive real numbers such that we assume the usual Robbins--Munro conditions, i.e. $\sum_{t\in\mathbb{N}} a_t =\infty$ and $\sum_{t\in\mathbb{N}} a_t^2 < \infty$. Given an initial $\theta_0\in\Theta$, this formula enables us to update $\theta$ online as we obtain a new observation path in each unit time interval. Computing
the gradients in the above formula can be expensive (see e.g. \cite{BCJKR20}), therefore we use a gradient-free method that is based on some type of finite differences with simultaneous perturbation stochastic approximation (SPSA) \cite{JCS03,JCS92}. In a standard finite difference approach, one perturbs $\theta_t$ in the positive and negative directions of a unit vector $\textbf{e}_k$ (a vector of zeros in all directions except $k$ it is 1). This means evaluating $U_{t,t+1,\theta_t}(Y)$ $2d_\theta$--times. 

Whereas in SPSA, we perturb $\theta_t$ with a magnitude of $b_t$ in the positive and negative directions of a $d_\theta$-dimensional random vector $\Psi_t$. The numbers $\{b_t\}_{t\in \mathbb{N}}$ are a sequence of positive real numbers such that $b_t \to 0$, $\sum_{t\in \mathbb{N}} a_t^2/b_t^2<\infty$, and for $k\in\{1\cdots,d_\theta\}$, $\Psi_t(k)$ is sampled from a Bernoulli distribution with success probability $1/2$ and support $\{-1,1\}$. Therefore, this method requires only 2 evaluations of $U_{t,t+1,\theta_t}(Y)$ to estimate the gradient. The MLEnKBF-NC estimator, presented in \autoref{alg:MLEnKBF_NC}, will be used to estimate the ratio $U_{t,t+1,\theta_t}(Y)$. In \autoref{alg:param_est} we illustrate how to implement these approximations in order to estimate the model's static parameters.

\begin{algorithm}[h!]
\caption{Parameter Estimation: using MLEnKBF-NC and RML-SPSA}
\label{alg:param_est}
\begin{enumerate}
\item \textbf{Input:} Target level $L\in\mathbb{N}$, start level $l_*\in \mathbb{N}$ such that $l_*<L$, the number of particles on each level $\{N_l\}_{l=l_*}^L$, the number of iterations $M\in\mathbb{N}$, initial $\theta_0\in\Theta$, step size sequences of positive real numbers $\{a_t\}_{t\in\mathbb{N}}$, $\{b_t\}_{t\in\mathbb{N}}$ such that $a_t,b_t \to 0$, $\sum_{t\in\mathbb{N}} a_t=\infty$, $\sum_{t\in\mathbb{N}} a_t^2/b_t^2<\infty$, and initial ensembles $\{\tilde{\xi}_0^{i}\}_{i=1}^{N_{tot}} \stackrel{i.i.d.}{\sim} \mathcal{N}(\mathcal{M}_0,\mathcal{P}_0)$, where $N_{tot}=\sum_{l=l_*}^L N_l$.

\item \textbf{Iterate:} For $t \in \{0,\cdots,M-1\}$: 
\begin{itemize}
\item Set $\{\xi_0^{i,l_*}\}_{i=1}^{N_{l_*}}=\{\tilde{\xi}_0^{i}\}_{i=1}^{N_{l_*}}$, $\cdots$, $\{\xi_0^{i,L}\}_{i=1}^{N_L}=\{\tilde{\xi}_0^{i}\}_{i=N_{L-1}+1}^{N_L}$.
\item For $k\in \{1,\cdots,d_\theta\}$, sample $\Psi_t(k)$ from a Bernoulli distribution with \\ success probability $1/2$ and support $\{-1,1\}$. 
\item Set $\theta_t^+= \theta_t + b_{t+1} \Psi_t$ and $\theta_t^-= \theta_t - b_{t+1} \Psi_t$.
\item Run \autoref{alg:MLEnKBF_NC} twice, with $T=1$ and initial ensembles $\Big\{ \{\xi_0^{i,l_*}\}_{i=1}^{N_{l_*}},\cdots,\{\xi_0^{i,L}\}_{i=1}^{N_L} \Big\}$, \\ to generate the estimates $U_{t,t+1,\theta_t^+}(Y)$ and $U_{t,t+1,\theta_t^-}(Y)$.
\item Set for $k\in\{1,\cdots,d_\theta\}$,
\begin{align*}
\theta_{t+1}(k) &= \theta_t(k) + \frac{a_{t+1}}{2b_{t+1}\Psi_t(k)} \Big[ U_{t,t+1,\theta_t^+}(Y)-U_{t,t+1,\theta_t^-}(Y) \Big].
\end{align*}
\item Run the EnKBF up to time 1 under the new parameter $\theta_{t+1}$ with discretization level $L$ and initial ensembles $\{\tilde{\xi}_0^{i}\}_{i=1}^{N_{tot}}$ to return $\{\tilde{\xi}_1^{i}\}_{i=1}^{N_{tot}}$.
\item Set $\{\tilde{\xi}_0^{i}\}_{i=1}^{N_{tot}}=\{\tilde{\xi}_1^{i}\}_{i=1}^{N_{tot}}$.
\end{itemize}
\end{enumerate}
\end{algorithm}

We will now introduce our different models we work with. This includes a linear and two nonlinear models, namely, the stochastic Lorenz 63 and Lorenz 96.

\subsubsection{Linear Gaussian model}
\label{subsec:lin_model}

Our first numerical example is a linear Gaussian model with $d_x=d_y=2$, based on \eqref{eq:data} - \eqref{eq:signal}. Specifically we choose the parameter values $A = \theta_1 Id$ and $Q^{1/2} = \theta_2 \mathcal{Q}$
where $\mathcal{Q}$ is a tri-diagnonal matrix defined as 
$$\mathcal{Q} = 
\begin{bmatrix}
1  & 1/2     &   0    &       \hdots & 0  \\
1/2  & 1 & 1/2 &  &   \vdots     \\
 0  &  \ddots & \ddots & \ddots  & 0 \\
\vdots &        & \ddots & \ddots & 1/2 \\
0 &   \hdots    &        0   & 1/2  & 1
\end{bmatrix}.
$$
We also take $C$ to be a uniform random matrix and $R^{1/2} =0.556~ Id$. The parameters of interest we aim to estimate are $(\theta_1,\theta_2) \in \mathbb{R}^2$. We choose the target level as $L=9$ and the start level as $l_*=7$, and specify an initial state of $X_0 \sim \mathcal{N}(4\mathbf{1},Id)$, where $\mathbf{1}$ is a vectors of 1's. In all cases, \textbf{(F1)} - \textbf{(F3)}, we take
\begin{align}
\label{eq:num_of_samples}
N_l = \floor*{0.04~ 2^{2L-l} ~(L-l_*+1)}.
\end{align}
The results are displayed in \autoref{fig:param_est_Lin}. In order to reduce the variance between the different simulations, we used the same Brownian increments (in the cases \textbf{(F1)} and \textbf{(F2)}) that were needed to generate $U_{t,t+1,\theta_t^+}(Y)$ and $U_{t,t+1,\theta_t^-}(Y)$ as well. We applied this trick only on the linear model, but it can be extended to the nonlinear models below. We observe from \autoref{fig:param_est_Lin} that all the variants learn of the true parameters $(\theta_1^*,\theta_2^*)=(-2,1)$, and from the bottom panel, $(\textbf{F3})$ has the smallest standard deviation.

\begin{rem}
 When testing  \textbf{(F3)} in the linear example, there were difficulties in learning $\theta_2$. This is because there is no natural coupling,  
hence why we had to choose a specific step-size as described in \autoref{fig:param_est_Lin}.
However for the nonlinear examples we did not experience this, where no specific modification of the step-size was required.
\end{rem}

\begin{figure}[!htb]
	\begin{subfigure}[c]{0.49\textwidth}
	\includegraphics[width=1.\textwidth]{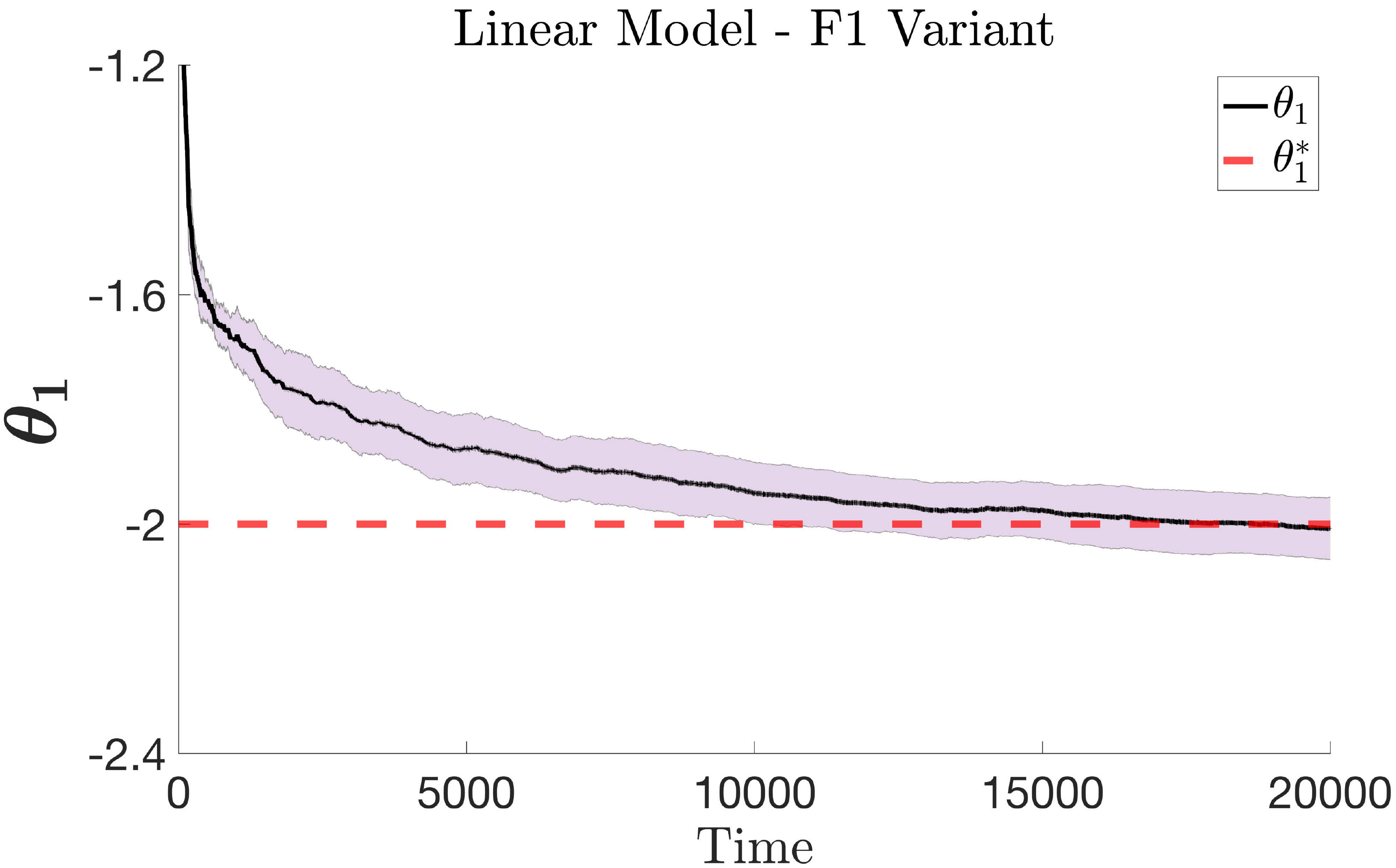}
	\end{subfigure}
	\begin{subfigure}[c]{0.5\textwidth}
	\includegraphics[width=1.\textwidth]{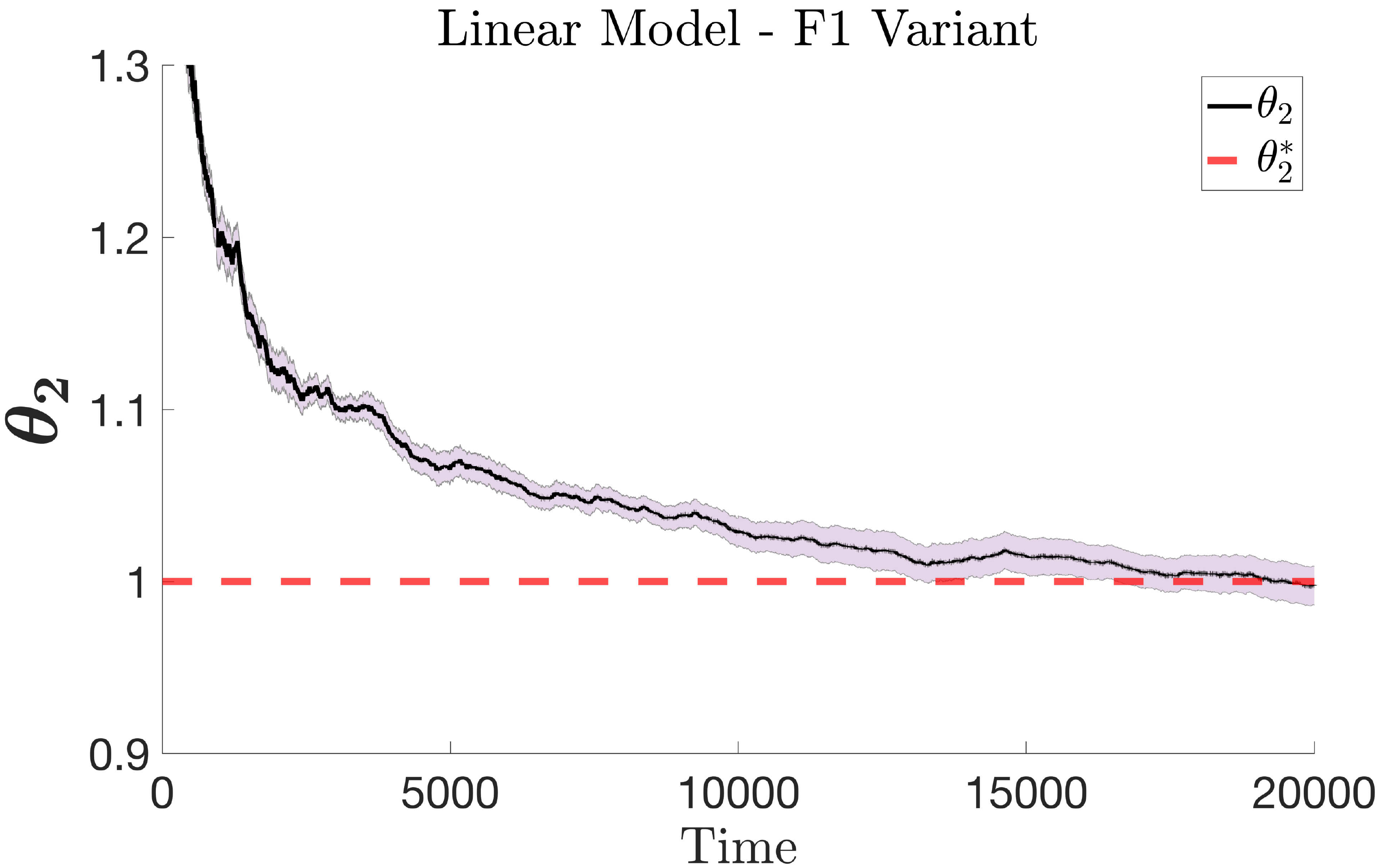}
	\end{subfigure}\\
	\begin{subfigure}[c]{0.49\textwidth}
	\includegraphics[width=1.\textwidth]{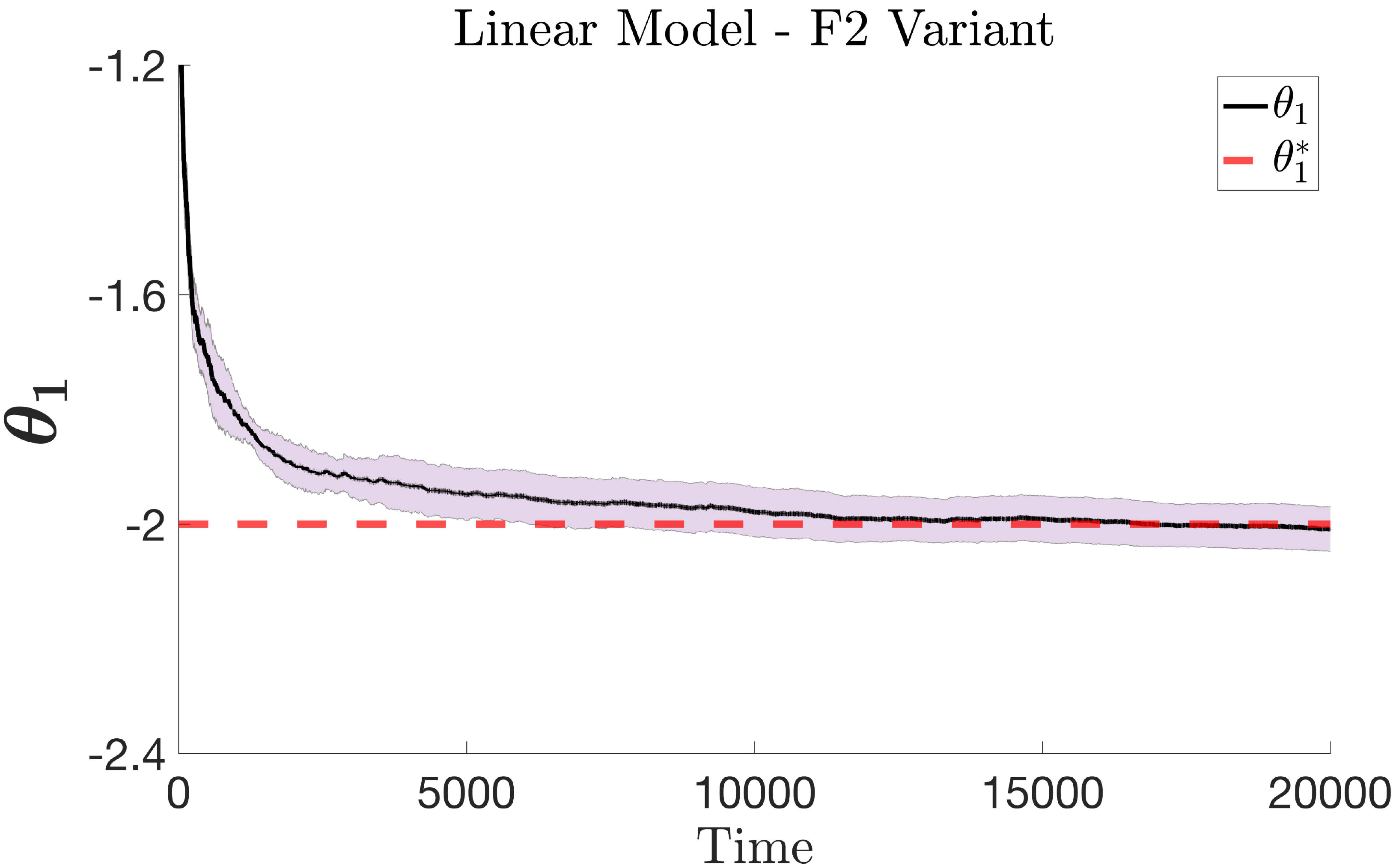}
	\end{subfigure}
	\begin{subfigure}[c]{0.5\textwidth}
	\includegraphics[width=1.\textwidth]{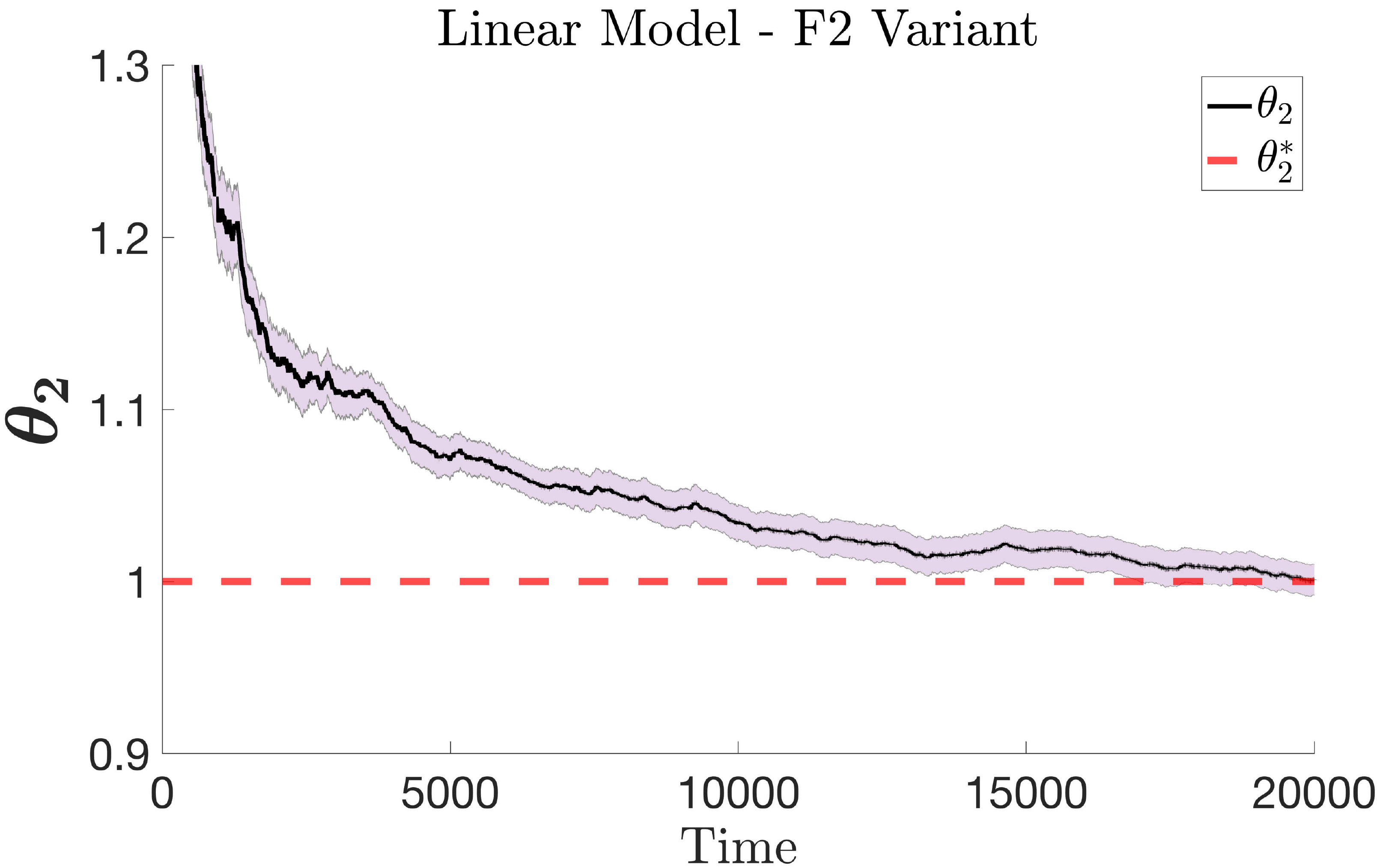}
	\end{subfigure}\\
	\begin{subfigure}[c]{0.49\textwidth}
	\includegraphics[width=1.\textwidth]{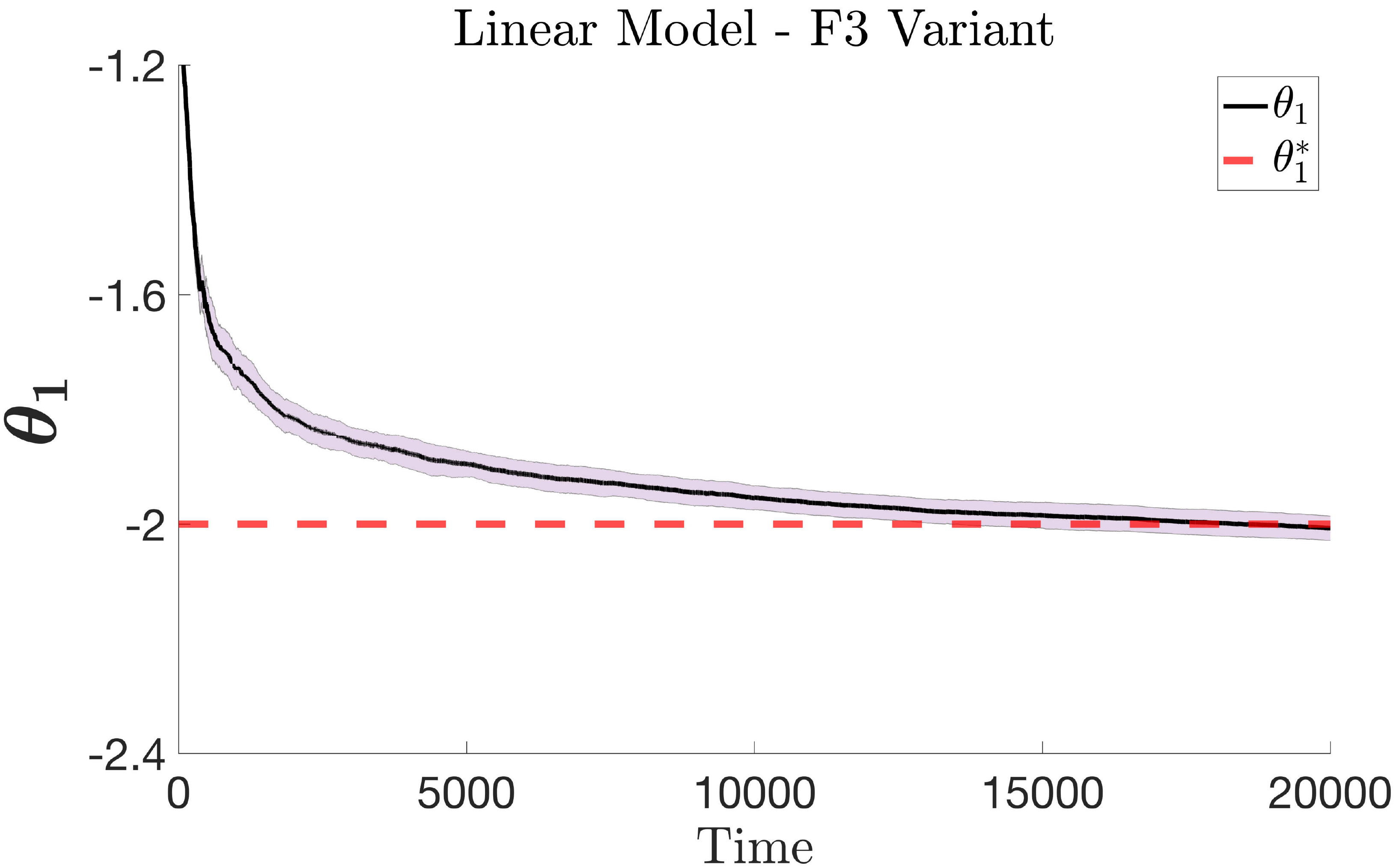}
	\end{subfigure}
	\begin{subfigure}[c]{0.5\textwidth}
	\includegraphics[width=1.\textwidth]{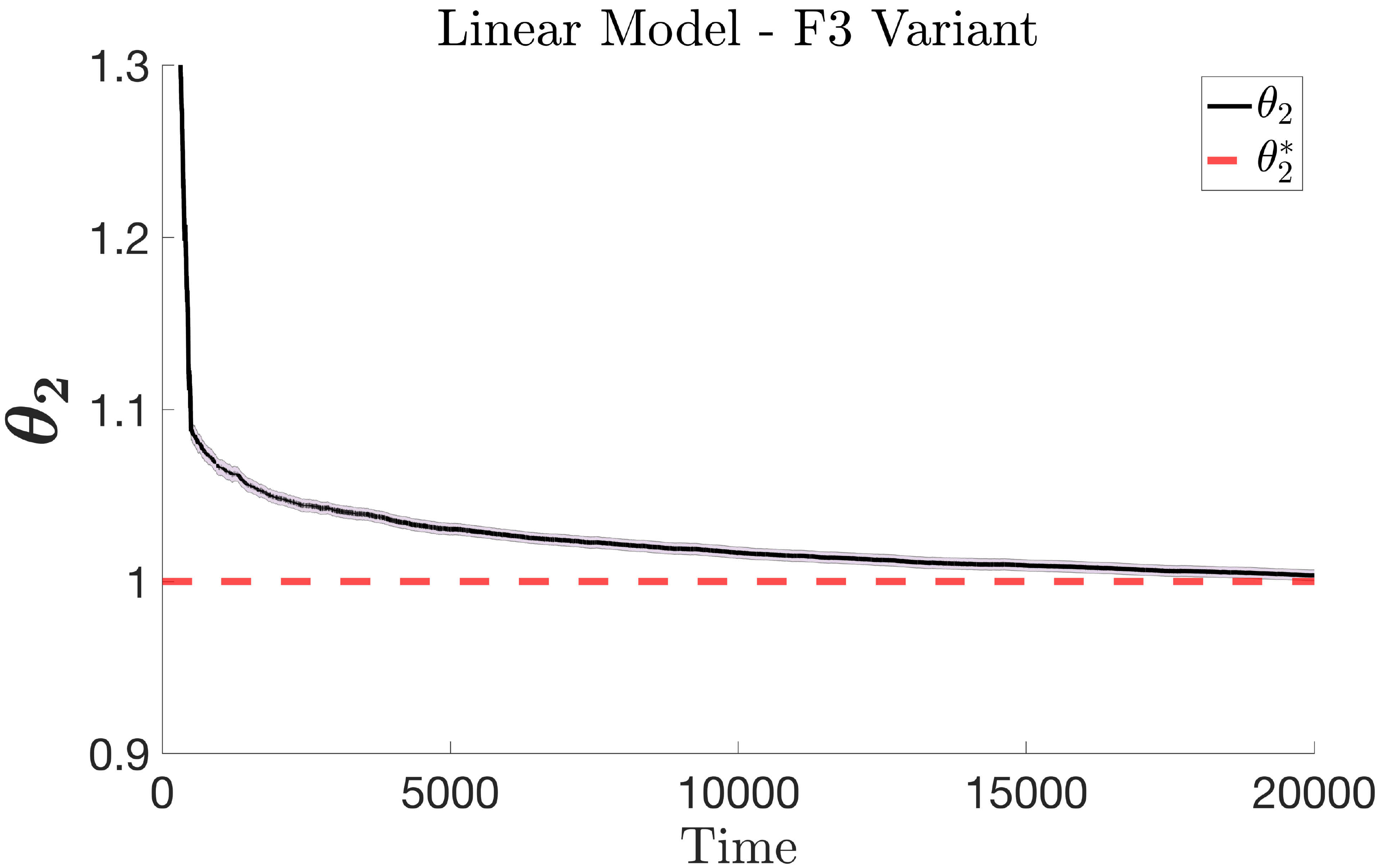}
	\end{subfigure}
 \caption{Linear results: The outcomes of running \autoref{alg:param_est} for the estimation of $(\theta_1,\theta_2)$ in the cases \textbf{(F1)} (top), \textbf{(F2)} (middle) and \textbf{(F3)} (bottom). The black curve is the average of 6 independent runs and the shaded area is the mean $\pm$ the standard deviation. The initial values of the parameters are $(-1, 2)$. The dashed lines represent the true parameters values $(\theta^*_1,\theta^*_2) = (-2,1)$. In all cases we take $b_t=t^{-0.1}$ for all $t\in\mathbb{N}$. In \textbf{(F1)} \& \textbf{(F2)} cases, we set $a_t=0.02$ when $t\leq 50$ and $a_t = t^{-0.75}$ (for $\theta_1$), $a_t = t^{-0.82}$ (for $\theta_2$) when $t>50$. In \textbf{(F3)} case, we set $a_t=0.02$ when $t\leq 500$ and $a_t =t^{-0.88}$ (for $\theta_1$), $a_t = 0.2~t^{-0.95}$ (for $\theta_2$) when $t>500$. }
    \label{fig:param_est_Lin}
\end{figure} 

\begin{figure}[!htb]
	\begin{subfigure}[c]{0.33\textwidth}
	\includegraphics[width=1.\textwidth]{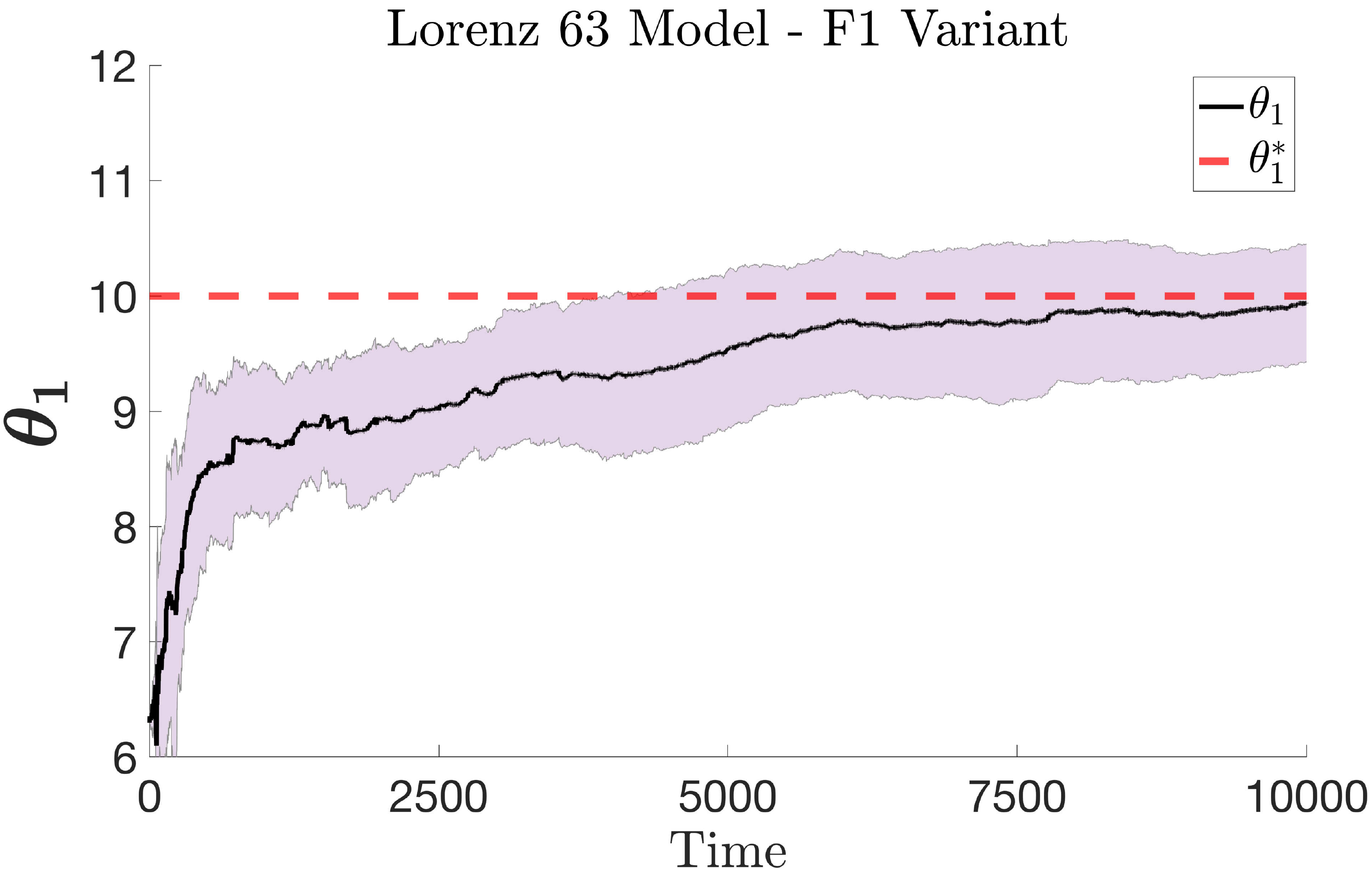}
	\end{subfigure}
	\begin{subfigure}[c]{0.33\textwidth}
	\includegraphics[width=1.\textwidth]{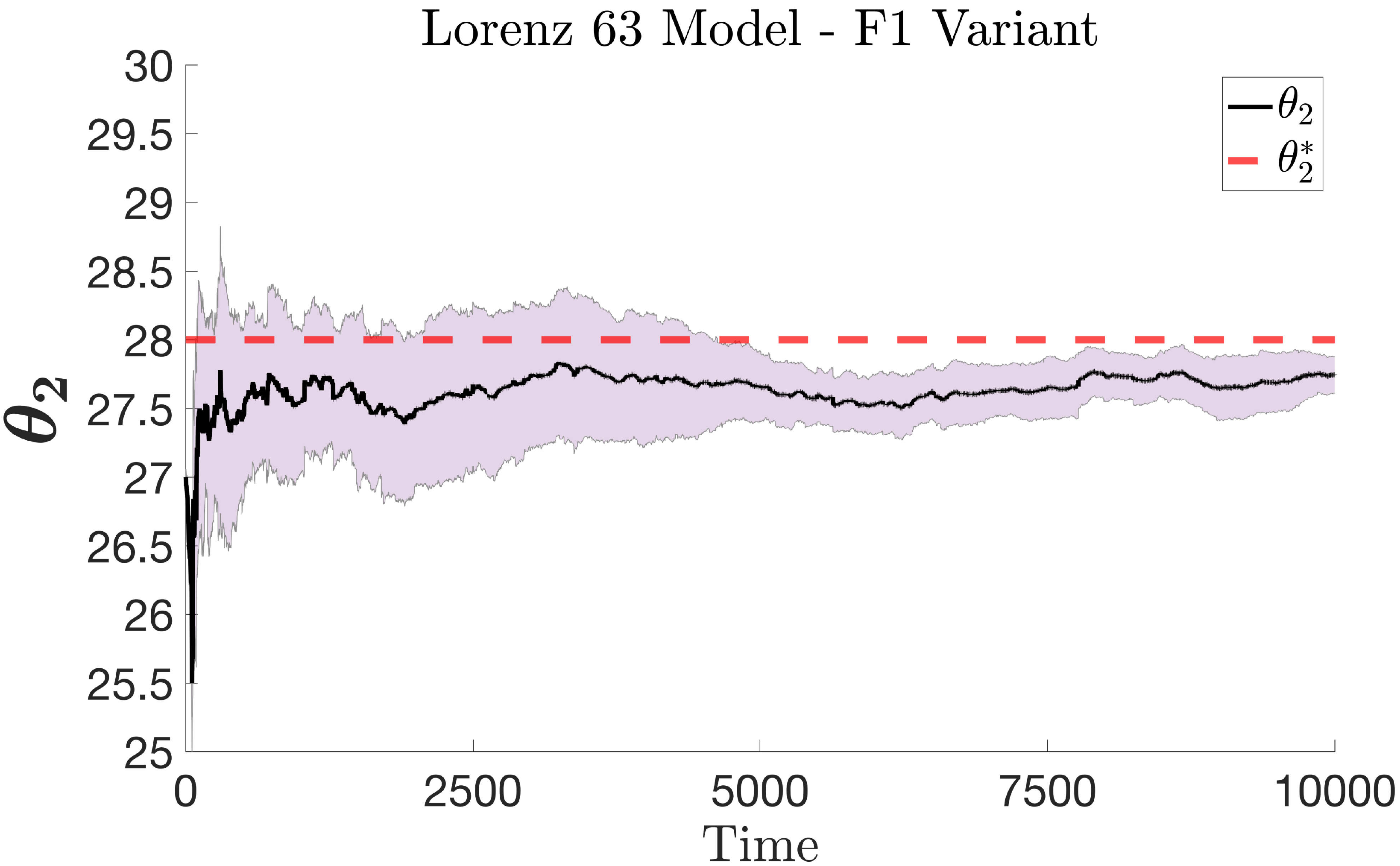}
	\end{subfigure}
	\begin{subfigure}[c]{0.33\textwidth}
	\includegraphics[width=1.\textwidth]{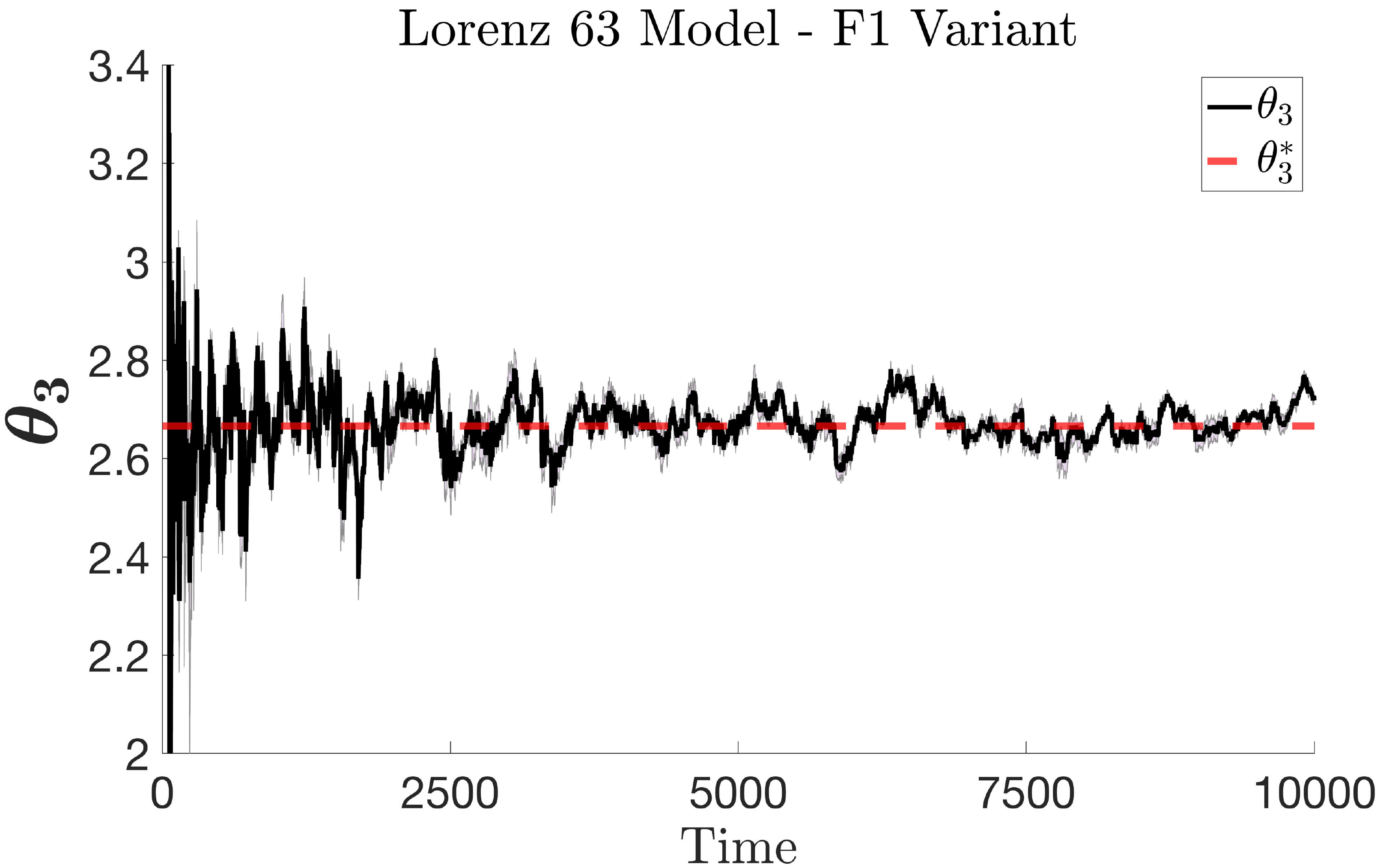}
	\end{subfigure}\\
	\begin{subfigure}[c]{0.33\textwidth}
	\includegraphics[width=1.\textwidth]{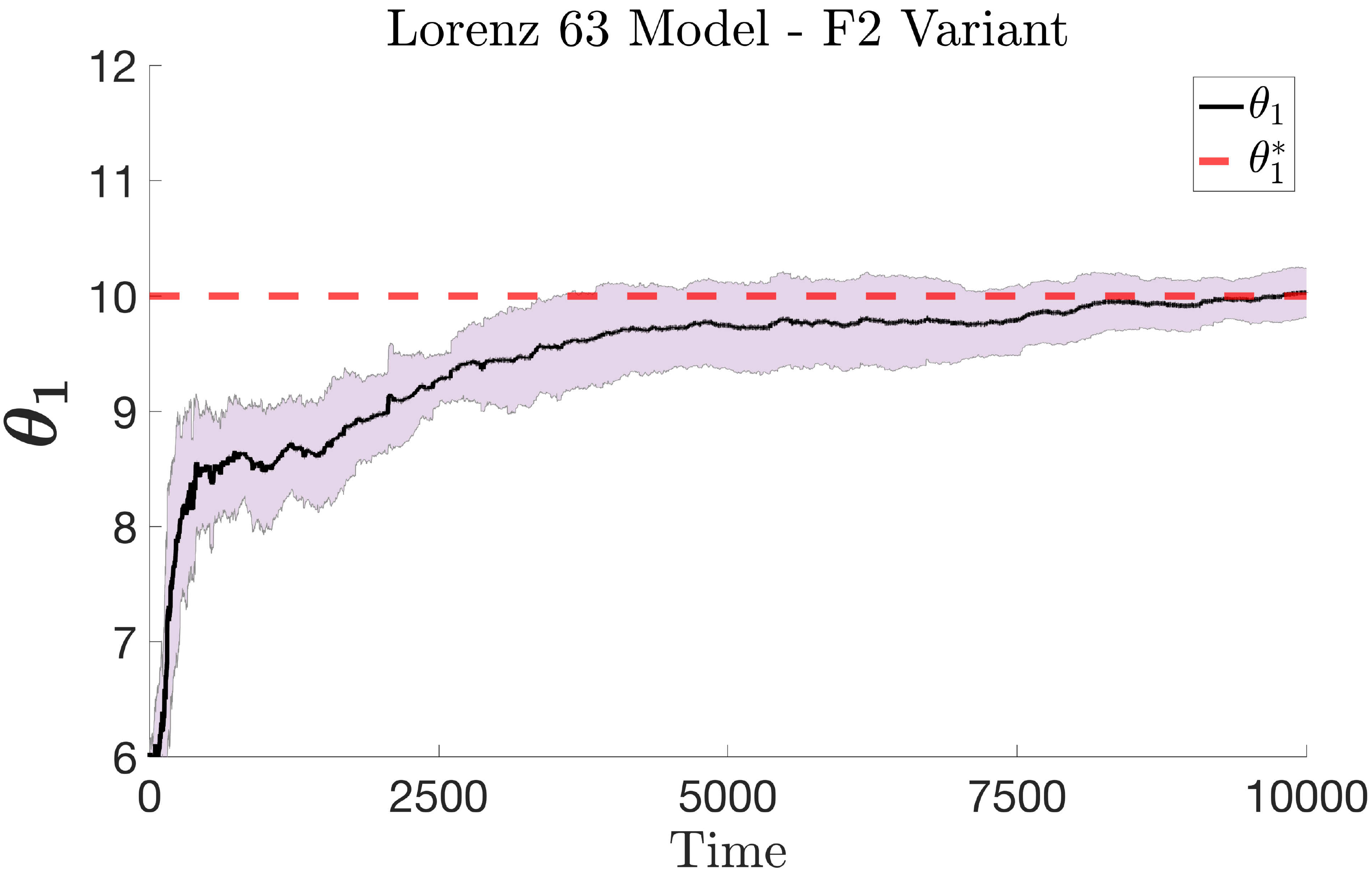}
	\end{subfigure}
	\begin{subfigure}[c]{0.33\textwidth}
	\includegraphics[width=1.\textwidth]{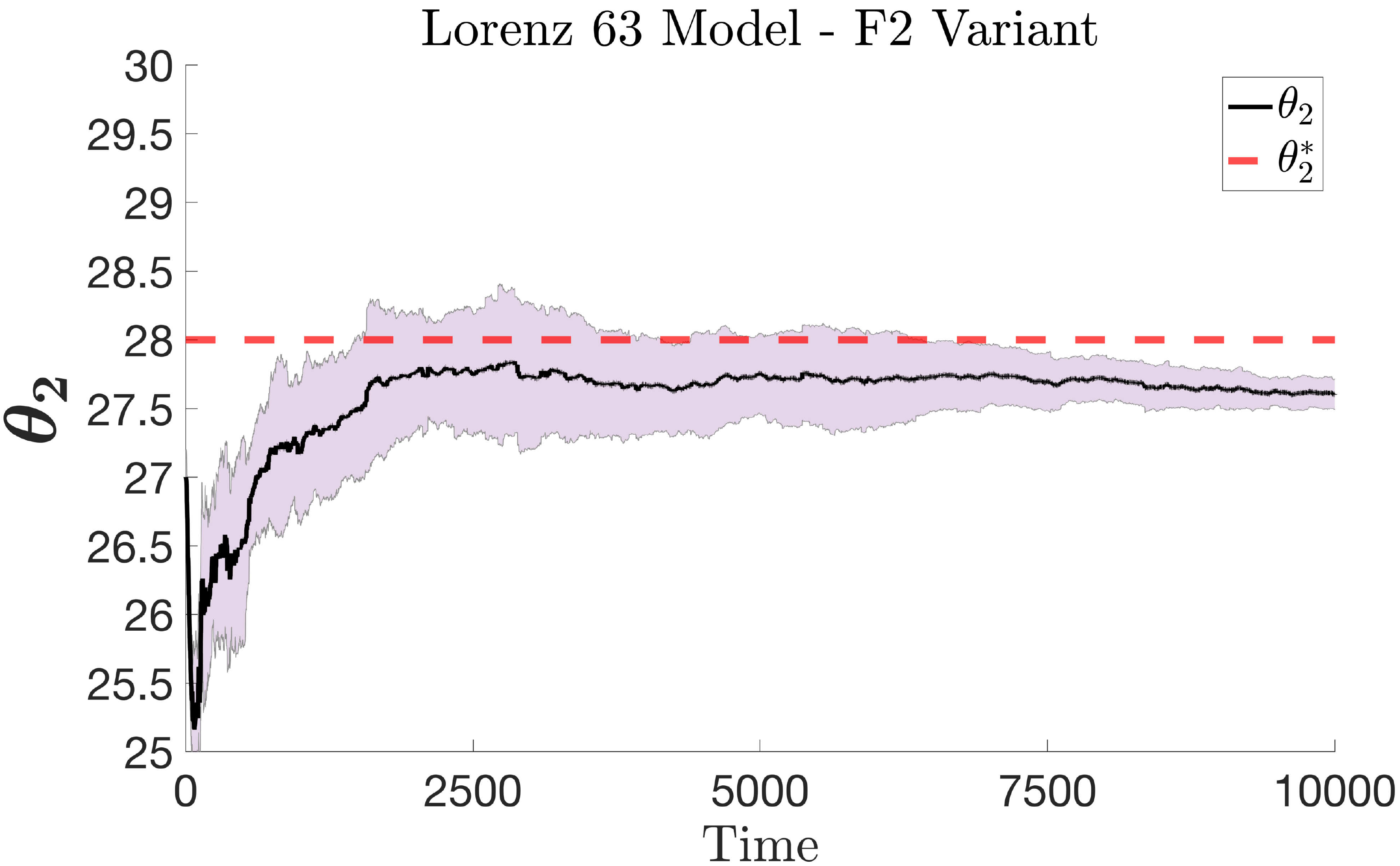}
	\end{subfigure}
	\begin{subfigure}[c]{0.33\textwidth}
	\includegraphics[width=1.\textwidth]{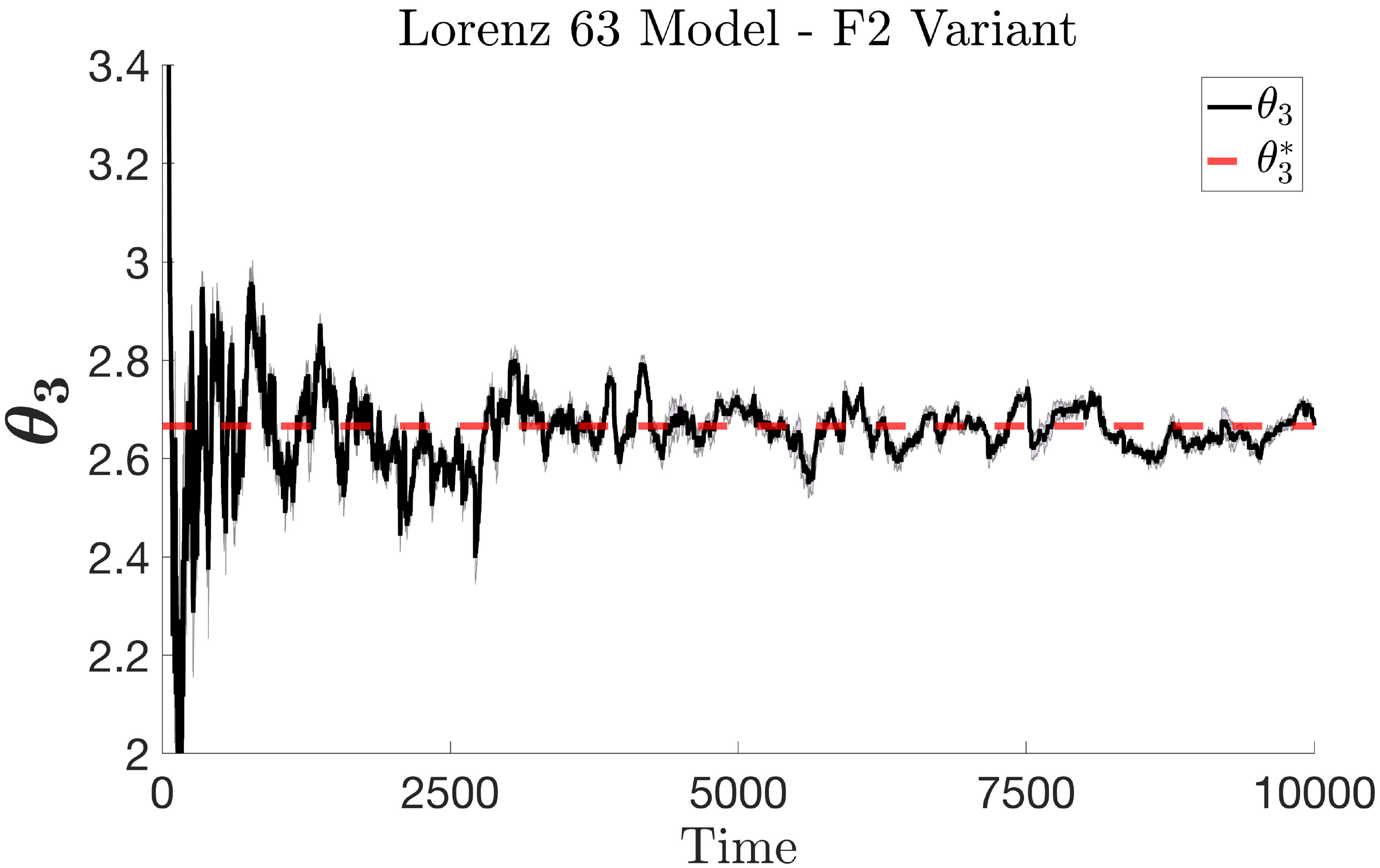}
	\end{subfigure}\\
	\begin{subfigure}[c]{0.33\textwidth}
	\includegraphics[width=1.\textwidth]{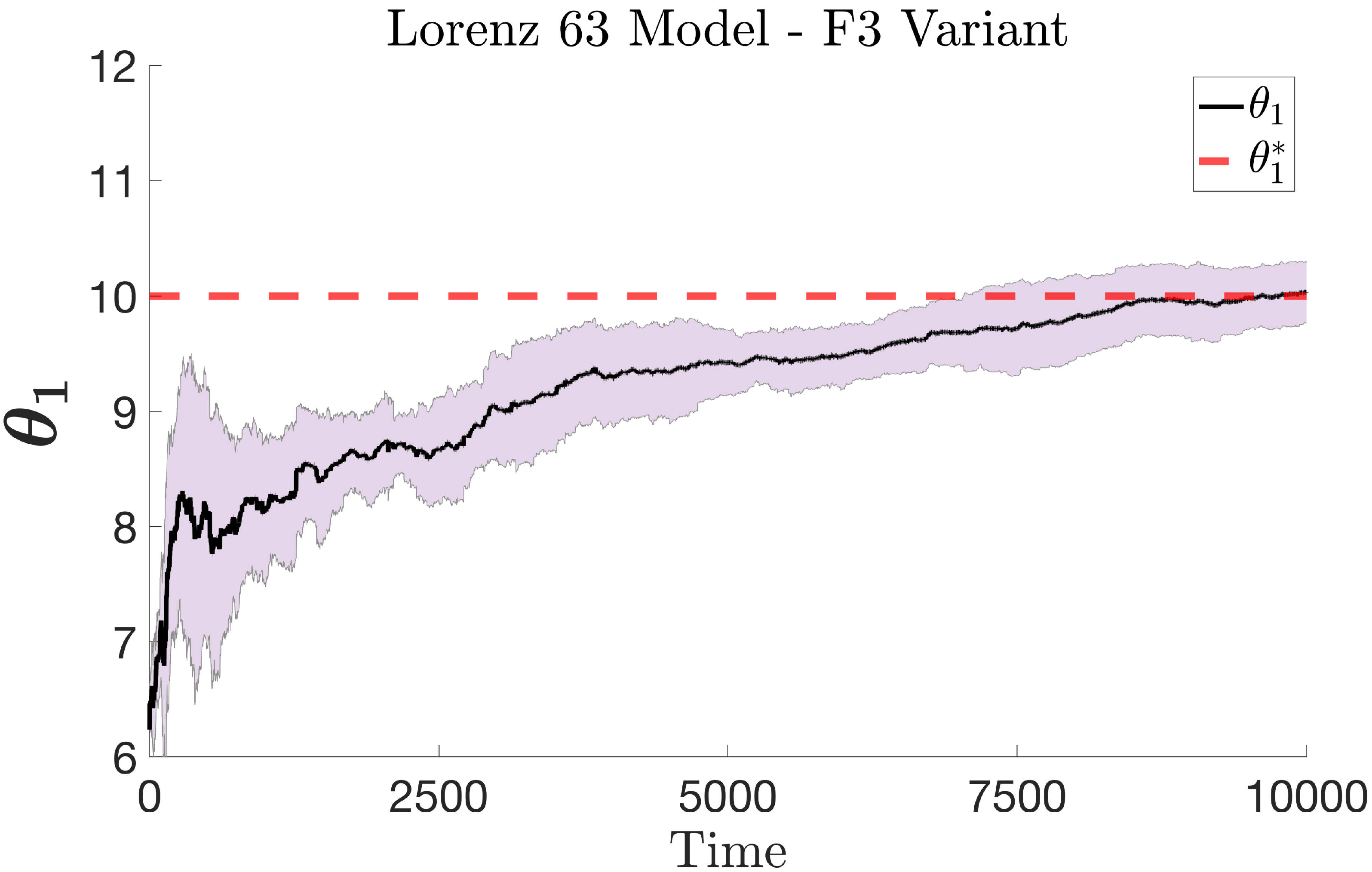}
	\end{subfigure}
	\begin{subfigure}[c]{0.33\textwidth}
	\includegraphics[width=1.\textwidth]{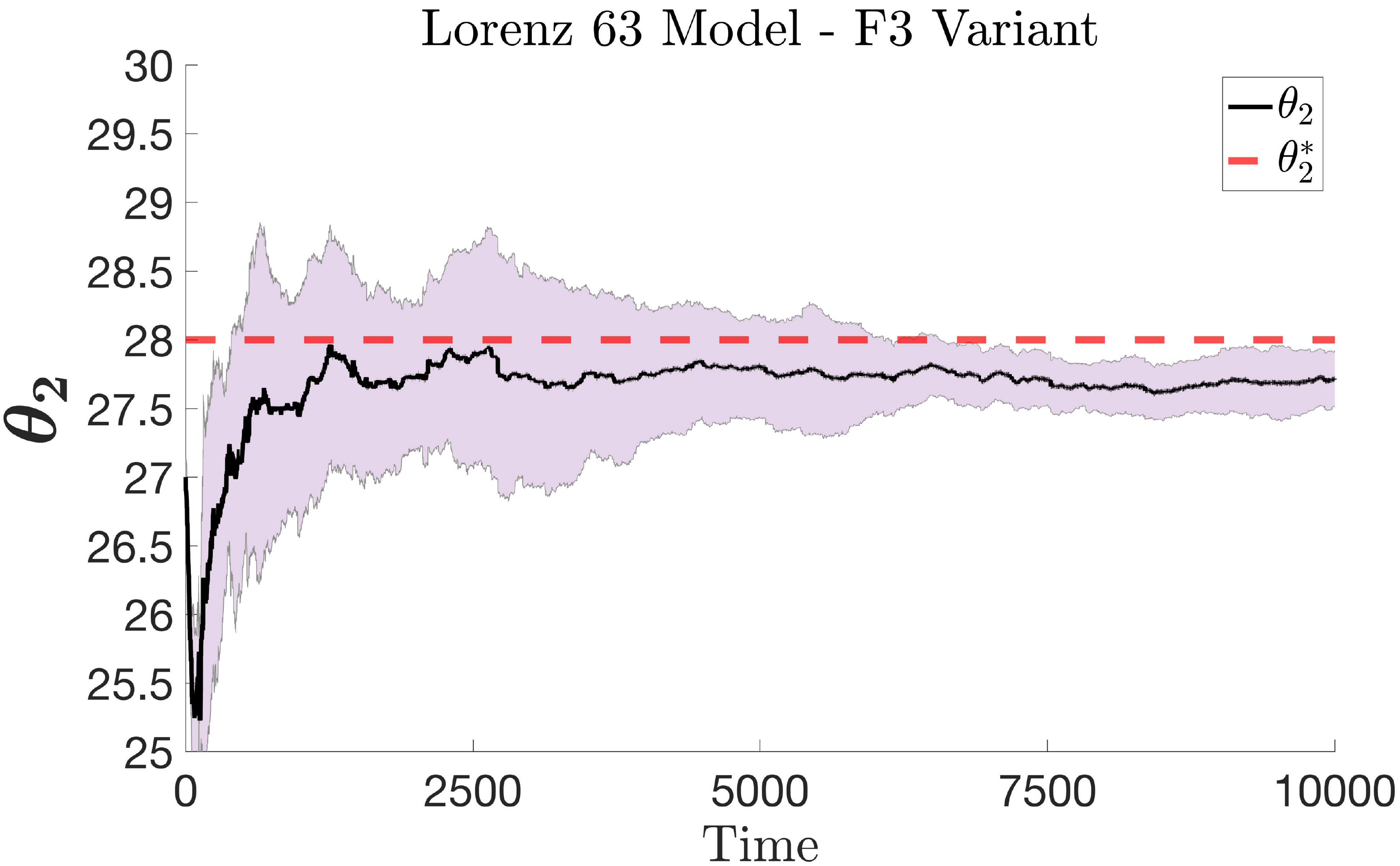}
	\end{subfigure}
	\begin{subfigure}[c]{0.33\textwidth}
	\includegraphics[width=1.\textwidth]{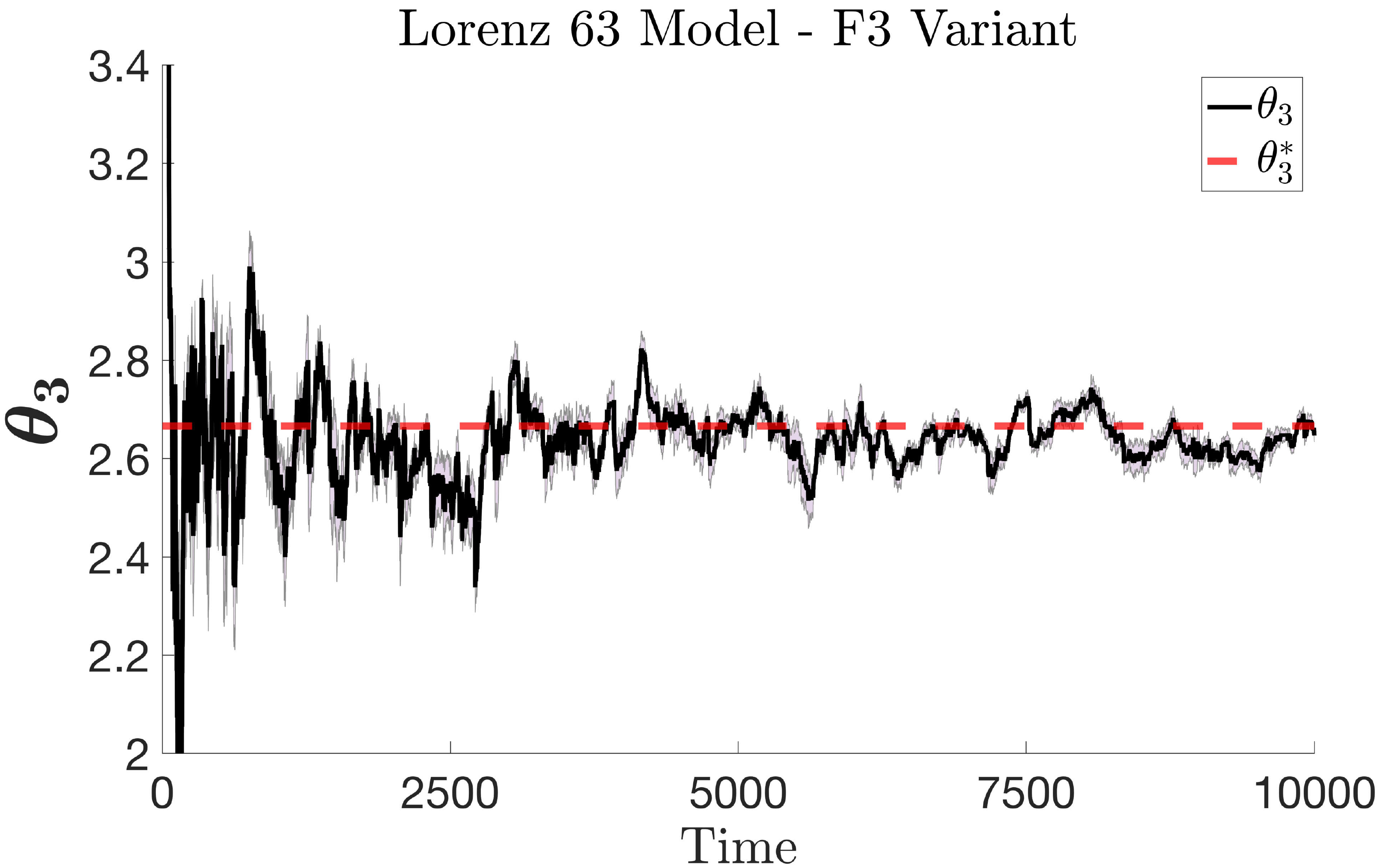}
	\end{subfigure}
    \caption{Lorenz 63 results: The outcomes of running \autoref{alg:param_est} for the estimation of $(\theta_1,\theta_2, \theta_3)$ in the cases \textbf{(F1)} (top), \textbf{(F2)} (middle) and \textbf{(F3)} (bottom). The black curve is the average of 6 independent runs and the shaded area is the mean $\pm$ the standard deviation. The initial values of the parameters are $(6, 27, 6.5)$. The dashed lines represent the true parameters values $(\theta^*_1,\theta^*_2, \theta_3^*) = (10,28,8/3)$. The green horizontal line represent the true parameters values $(\theta^*_1,\theta^*_2, \theta_3^*) = (10,28,8/3)$. In all cases, we set $b_t=t^{-0.1}$ for all $t\in\mathbb{N}$, $a_t = 0.01$ when $t\leq 100$ and $a_t=t^{-0.75}$ for $t>100$.}
    \label{fig:param_est_L63}
\end{figure} 

\subsubsection{Stochastic Lorenz 63 model}
\label{subsec:L63_model}

Our next example is the Lorenz 63 model \cite{ENL63} with $d_x=d_y=3$, which is a model for atmospheric convection. The model is based on three ordinary differential equations, where now we have three parameters of interest to estimate, i.e. $(\theta_1,\theta_2,\theta_3) \in \mathbb{R}^3$. The stochastic Lorenz 63 model is given as
\begin{align*}
dX_t &=   f(X_t) dt +  Q^{1/2} dW_t, \\
dY_t &=  C X_t dt + R^{1/2} dV_t,   \\
\end{align*}
such that 
\begin{align*}
f_1(X_t) &= \theta_1(X_t(2) - X_t(1)), \\
f_2(X_t) &= \theta_2X_t(1) - X_t(2) - X_t(1)X_t(3), \\
f_3(X_t) & = X_t(1)X_t(2)-\theta_3X_t(3),
\end{align*}
where $X_t(i)$ is the $i^{th}$ component of $X_t$. Furthermore we have that $Q^{1/2} = Id$ and the variable $C$ is specified as
$$
C = 
\begin{cases}
\frac{1}{2}, \quad &\mathrm{if} \ i=j,\\
\frac{1}{2}, \quad &\mathrm{if} \ i=j-1, \quad j \in \{1,2,3\},\\
0, \quad &\mathrm{otherwise},
\end{cases}
$$
and $(R^{1/2})_{ij}=2~q(\frac{2}{5}\min\{|i-j|,r_2-|i-j| \}), \ i,j \in \{1,2,3\}$ such that 
$$
q(x) = 
\begin{cases}
1-\frac{3}{2}x+\frac{1}{2}x^3, \quad &\mathrm{if} \ 0 \leq x \leq 1,\\
0, \quad &\mathrm{otherwise}.
\end{cases}
$$
In \autoref{fig:param_est_L63}, we show the results for the parameters estimation of $(\theta_1,\theta_2,\theta_3)$ using the cases \textbf{(F1)} - \textbf{(F3)}. We set the target level to be $L=9$, the start level $l_*=7$, and specify the initial state $X_0\sim \mathcal{N}(\textbf{1},0.5~Id)$. The number of samples on each level are the same as in \eqref{eq:num_of_samples}. From \autoref{fig:param_est_L63} it is clear that all variants of the MLEnKBF perform similarly for this inference problem. Interestingly we notice that the learning of the parameter $\theta^*_3 = 8/3$ seems the most accurate, while taking the least amount of time to reach the value of $\theta^*_3$. This is unlike the learning of the other parameters, which require at least a time of $T=5000$ to get close to $\theta^*_1$ and $T=3000$ to get close to $\theta^*_2$. Part of the reason for this could be that $\theta_3$ acts as a coefficient only for $X_t(3)$, unlike for $X_t(1)$ which depends on both $\theta_1$ and $\theta_2$. In addition, we see that the learning of $\theta^*_2$ is biased towards 27.6 in all variants.

\begin{figure}[!htb]
	\begin{subfigure}[c]{0.33\textwidth}
	\includegraphics[width=1.\textwidth]{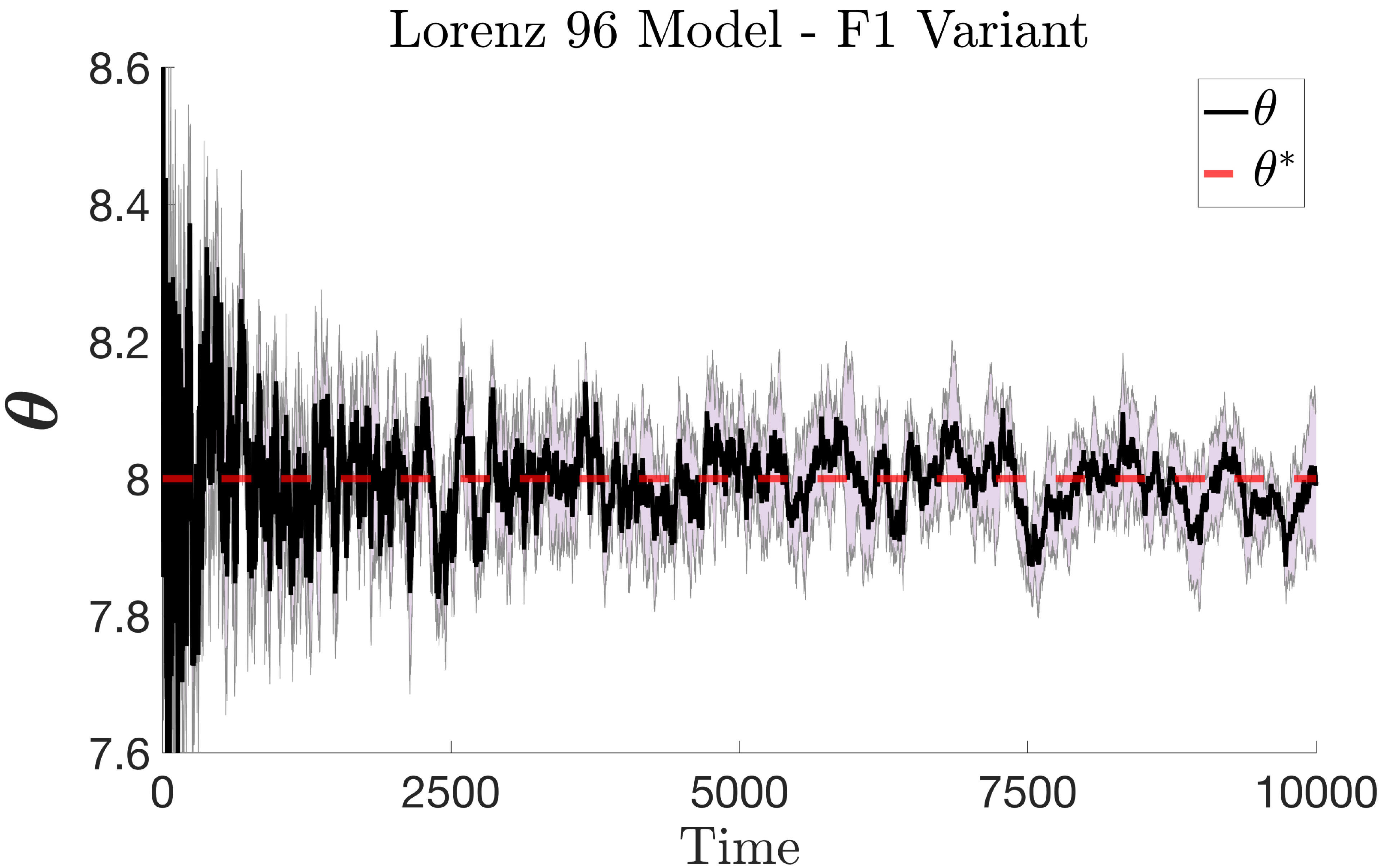}
	\end{subfigure}
	\begin{subfigure}[c]{0.33\textwidth}
	\includegraphics[width=1.\textwidth]{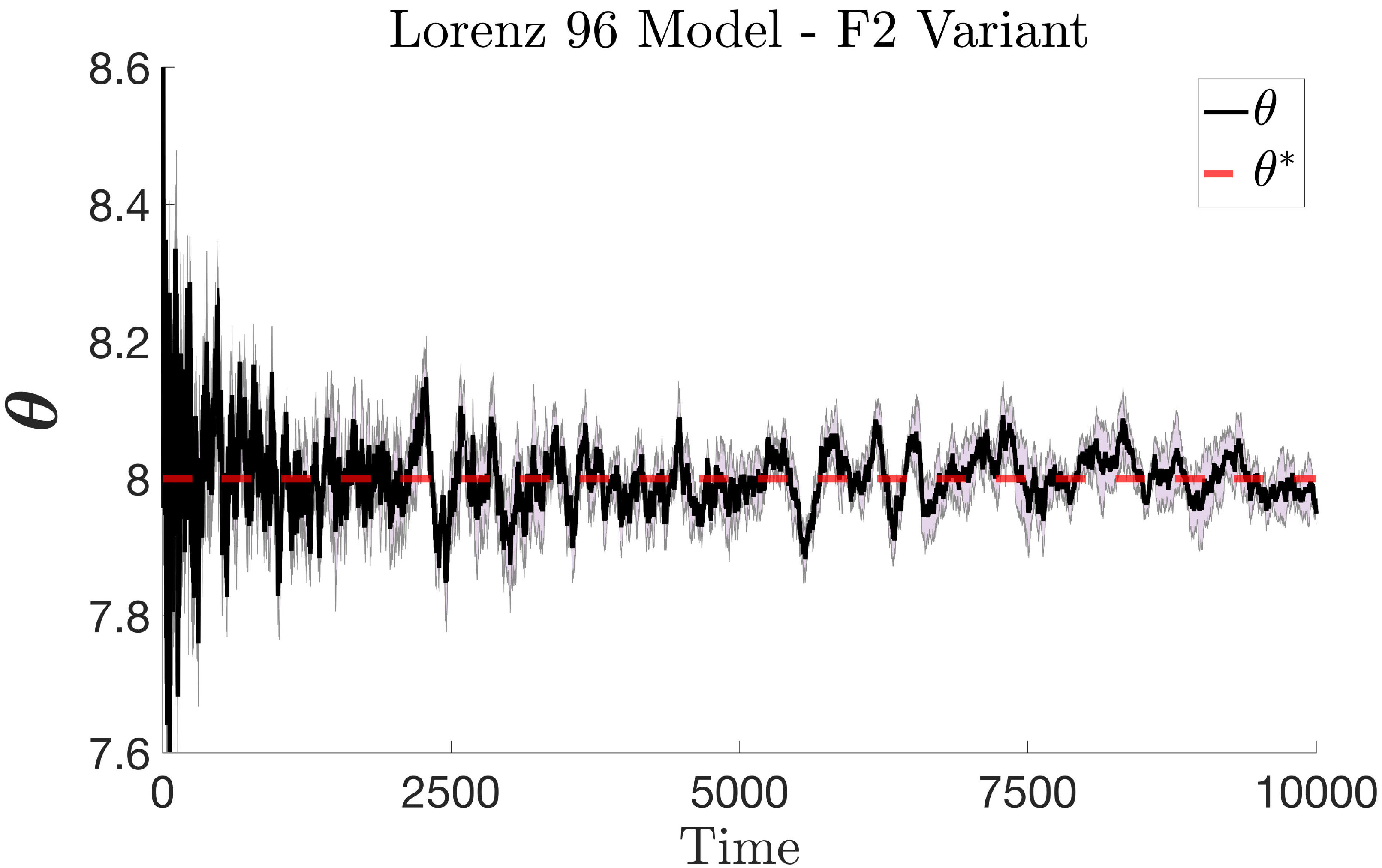}
	\end{subfigure}
	\begin{subfigure}[c]{0.33\textwidth}
	\includegraphics[width=1.\textwidth]{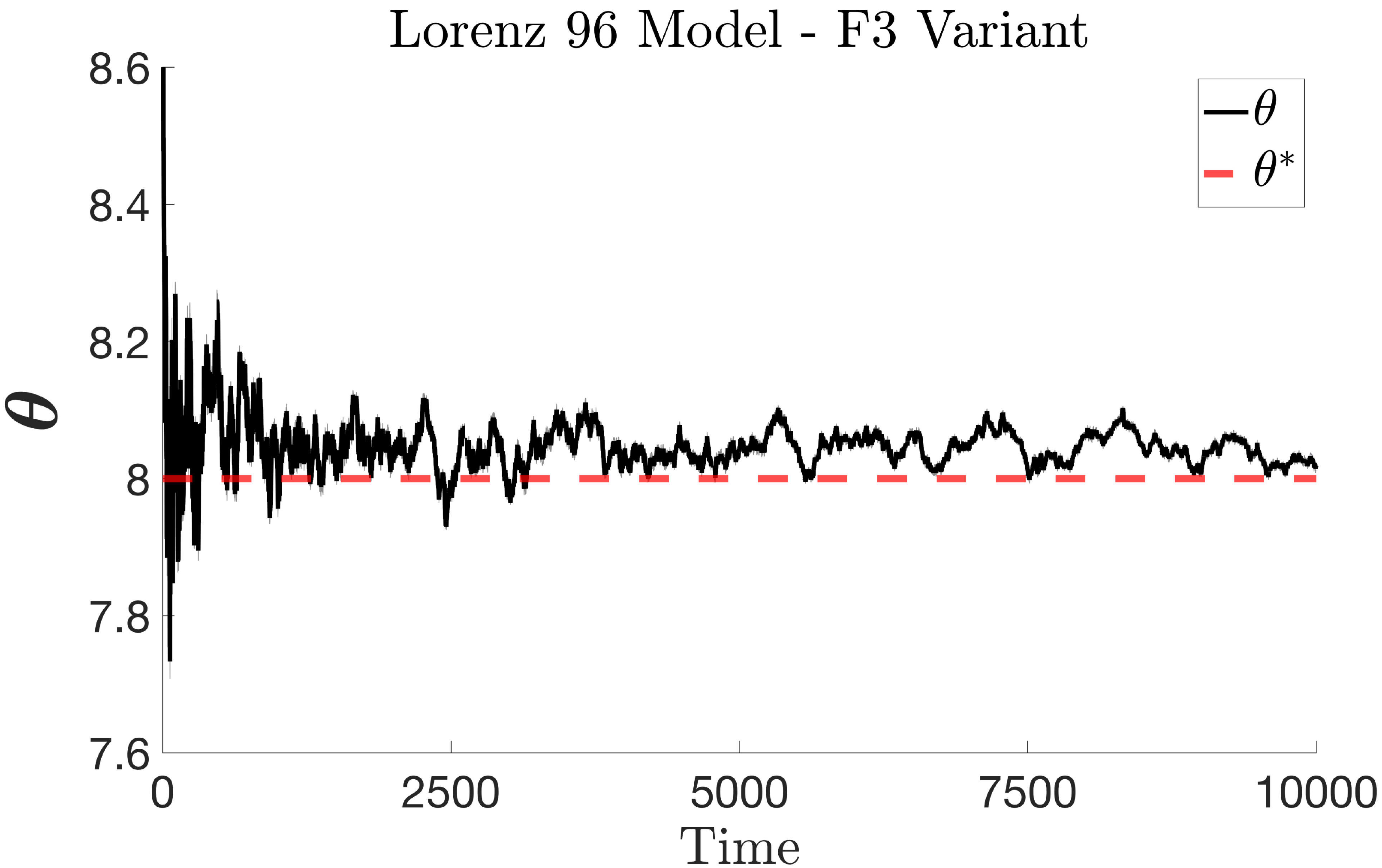}
	\end{subfigure}
    \caption{Lorenz 96 results: The outcomes of running \autoref{alg:param_est} for the estimation of $\theta$ in the cases \textbf{(F1)} (left), \textbf{(F2)} (middle) and \textbf{(F3)} (right). The black curve is the average of 6 independent runs and the shaded area is the mean $\pm$ the standard deviation. The initial value of $\theta$ is 10. The dashed line represents the true value $\theta^* = 8$. In all cases, we set $b_t=t^{-0.1}$ for all $t\in\mathbb{N}$, $a_t = 0.03$ when $t\leq 50$ and $a_t=t^{-0.75}$ for $t>50$.}
    \label{fig:param_est_L96}
\end{figure}

\subsubsection{Stochastic Lorenz 96 model}
Our final test model will be the Lorenz 96 model \cite{ENL96} with $d_x=d_y=40$, which is a dynamical system designed to describe equatorial waves in atmospheric science. The stochastic Lorenz 96 model takes the form
\begin{align*}
dX_t &=   f(X_t) dt +  Q^{1/2} dW_t, \\
dY_t &=  C X_t dt + R^{1/2}dV_t,  
\end{align*}
such that 
\begin{equation*}
f_i(X_t) = (X_t(i+1)-X_t(i-2))X_t(i-1)-X_t(i)+\theta,
\end{equation*}
where again $X_t(i)$ is the $i^{th}$ component of $X_t$, and we assume that $X_t(-1) = X_t(d_x-1)$, $X_t(0) = X_t(d_x)$ and that $X_t(d_x+1) = X_t(1)$. We specify our parameter values as $Q^{1/2} = \sqrt{2}~Id$ and $R^{1/2} = 0.5~ Id$. The parameter $\theta$ is the external force in the system, while $(X_t(i+1)-X_t(i-2))X_t(i-1)$ is the advection term and $-X_t(i)$ is the damping term. In \autoref{fig:param_est_L96}, we show the results for the parameter estimation of $\theta$ using the cases \textbf{(F1)}-\textbf{(F3)}. We set the target level to be $L=9$, and the start level $l_*=7$. In \textbf{(F1)} and \textbf{(F2)}, we specify the initial state as follows: We set $X_0(1) = 8.01$ and $X_0(i)=8$ for $1<i\leq d_x$. In \textbf{(F3)} case, to prevent the matrix $P_0^{N,l}$ from being
equal to zero, we set $X_0 \sim \mathcal{N}(8~\textbf{1},0.05~Id)$. The number of samples on each level are the same as in \eqref{eq:num_of_samples}. As we observe from \autoref{fig:param_est_L96}, all MLEnKBF variants learn the true value of $\theta^*=8$. The fluctuations again can be seen more clearly in the first two variants, as they are stochastic, i.e. contain either or both of $Q$ and $R$. This is not the case for \textbf{(F3)} whose dynamics are completely deterministic. In general the results are more stable that those conducted for the Lorenz 63 model, as the function $f:\mathbb{R}^{d_x} \rightarrow \mathbb{R}^{d_x}$, does not contain $\theta$ as a coefficient of the process $X_t$. 

\section{Conclusion}
\label{sec:conc}

The purpose of this work was to apply MLMC strategies for normalizing constant (NC) estimation. In particular our aim was to extend the work of \cite{CDJ21}, which used EnKBF, to its multilevel counterpart which is the MLEnKBF \cite{CJY20}. As stated, our motivation is primarily in the linear setting, for which the optimal filter is the Kalman--Bucy filter, where we could provide propagation of chaos-type results. This has the advantage of being more computationally feasible than using other filtering strategies such as the particle filter or SMC methods. From our results, we analyzed and provided various $\mathbb{L}_q$ bounds associated to the i.i.d. normalizing constant estimator. This was for both the single-level and multilevel cases. Numerical experiments were conducted which showed the improvement, related to cost, of adopting the multilevel methodology for the EnKBF. This was further tested on various models for parameter estimation, where we compared different variants of MLEnKBFs, introduced earlier on the stochastic Lorenz 63 and 96 models.

This work naturally leads to different fruitful and future directions of work to consider. The first potential direction is to aim to provide the same analysis for other MLEnKBFs, such as the deterministic and deterministic-transport EnKBF \cite{BD20}. This is of interest as from the numerical results, they suggest alternative rates for \textbf{(F3)}, different to that of the other variants. One could also provide an unbiased estimation of the NC \cite{RJP18}, which has connections to MLMC \cite{CFJ20,MV18}. In order to do so, one would require a modified multilevel analysis of the EnKBF, where one has uniform upper bounds, with respect to levels $l=1,\ldots, L$. As stated, a comparison
of both normalizing estimators using both the MLEnKBF and MLPF would be interesting, but to make such a comparison, one could exploit advanced methodologies, such as in \cite{BCJKR20}, and appropriate examples in low dimensions. Finally it would be of interest to develop theory for the multilevel estimator $\overline{U}^{ML}_t$. We have only tested this computationally which seems to match the rates attained for the i.i.d. (ideal) ML estimator. To do so, one requires more sophisticated mathematics to hold initially for the problem-setting in \cite{CJY20}.

\section*{Appendix}
\label{a:appendix}

\appendix

 \section{Analysis for EnKBF NC Estimator}
 \label{Appendix:A}

{
For the appendices, \autoref{Appendix:A} will cover the propagation of chaos result, which is required for the variance of the single-level EnKBF NC estimator. We will then proceed to  \autoref{Appendix:B} which discusses various discretization biases of the diffusion process related to both EnKBF and the NC estimator. Finally our main theorem is proved in  \autoref{Appendix:C}. All of our results will be specific to the vanilla variant of the ENKBF, \textbf{F(1)}.

Before proceeding to our results, we will introduce the following assumptions which will hold from herein, but not be added to any statements. 
For a square matrix, $B$ say, we denote by $\mu(B)$ as the maximum eigenvalue of $\textrm{Sym}(B)$.
\begin{enumerate}
\item{We have that $\mu(A)<0$.}
\item{There exists a $\mathsf{C}<+\infty$ such that for any $(k,l)\in\mathbb{N}_0^2$ we have that 
\begin{equation}\label{eq:bound_deter_cov}
\max_{(j_1,j_2)\in\{1,\dots,d_x\}^2}|P_{k\Delta_l}(j_1,j_2)|\leq\mathsf{C}.
\end{equation}
}
\end{enumerate}
We note that 1.~is typically used in the time stability of the hidden diffusion process $X_t$, see for instance \cite{DT18}. In the case of 2.~we expect that it can be verified under 1.,
that $S=\mathsf{C}I$ with $\mathsf{C}$ a positive constant and some controllability and observability assumptions (e.g.~\cite[eq.~(20)]{DT18}). Under such assumptions, the Riccati equation has a solution and moreover, by \cite[Proposition 5.3]{DT18} $\mathcal{P}_t$ is exponentially stable w.r.t.~the Frobenius norm; so that this type of bound exists in continuous time.}

Throughout the appendix we will make use of the $C_q-$inequality. For two real-valued random variables
$X$ and $Y$ defined on the same probability space, with expectation operator $\mathbb{E}$, suppose that for some fixed $q\in(0,\infty)$, $\mathbb{E}[|X|^q]$
and $\mathbb{E}[|Y|^q]$ are finite, then the $C_q-$inequality is
$$
\mathbb{E}[|X+Y|^q] \leq \mathsf{C}_q\Big(\mathbb{E}[|X|^q]+ \mathbb{E}[|Y|^q]\Big),
$$
where $\mathsf{C}_q=1$, if $q\in(0,1)$ and $\mathsf{C}_q=2^{q-1}$ for $q\in[1,\infty)$.
\\\\
In order to verify some of our claims for the analysis, we will rely on various results derived in \cite{CJY20}. For convenience-sake we will state these below, which are concerned with various $\mathbb{L}_q-$bounds, where


\begin{lem}
\label{lem:y_inc}
For any $q\in[1,\infty)$ there exists a $\mathsf{C}<+\infty$ such that for any $(k,l,j)\in\mathbb{N}_0^2\times\{1,\dots,d_y\}$:{
$$
\mathbb{E}[|[Y_{(k+1)\Delta_l}-Y_{k\Delta_l}](j)|^q]^{1/q} \leq \mathsf{C}\Delta_l^{1/2}.
$$}
\end{lem}

\begin{lem}
\label{lem:mean}
For any $(q,t,k,l)\in[1,\infty)\times\mathbb{N}_0^3$ there exists a $\mathsf{C}<+\infty$ such that for any $(j,N)\in\{1,\dots,d_x\}\times\{2,3,\dots\}$:
$$
\mathbb{E}\Big[\Big|m_{t+k_1\Delta_l}^N(j)-m_{t+k_1\Delta_l}(j)\Big|^q\Big]^{1/q} \leq \frac{\mathsf{C}}{\sqrt{N}}.
$$
\end{lem}

\begin{lem}\label{lem:xi_lq}
For any $(q,k,l)\in(0,\infty)\times\mathbb{N}_0^2$ there exists a $\mathsf{C}<+\infty$ such that for any $N \geq 2$ and  $i\in \{1,\dots,N\}$:
$$
\mathbb{E}[|\xi_{k\Delta_l}^i(j)|^q]^{1/q} \leq \mathsf{C},
$$
where $\xi_{k\Delta_l}$ is defined through \eqref{eq:enkbf1}.
\end{lem}

We now present our first result for the single-level EnKBF NC estimator, which is presented as an $\mathbb{L}_q-$ error bound.

\begin{proposition}
\label{lem:sl_nc}
For any $(q,t,k_1,l)\in [1,\infty) \times \mathbb{N}_0^3$ there exists a $\mathsf{C}<+\infty$ such that for any $N \in \{2,3,\ldots\}$ we have:
$$
\mathbb{E}\Big[\Big|[\overline{U}_{t+k_1\Delta_l}^{N,l}(Y) - \overline{U}_{t+k_1\Delta_l}^{l}(Y)]\Big|^q\Big]^{1/q}\leq \frac{\mathsf{C}}{\sqrt{N}}.
$$
\end{proposition}

\begin{proof}
Let us first consider $\overline{U}_{t+k_1\Delta_l}^{N,l}(Y) - \overline{U}_{t+k_1\Delta_l}^{l}(Y)$, which for every $l \in \mathbb{N}_0$, we can decompose through a martingale remainder-type decomposition,
\begin{equation}
\label{eq:diff}
\overline{U}_{t+k_1\Delta_l}^{N,l}(Y) - \overline{U}_{t+k_1\Delta_l}^{l}(Y) = M_{t+k_1\Delta_l}^l(Y) + R_{t+k_1\Delta_l}^l,
\end{equation}
such that 
\begin{eqnarray*}
M_{t+k_1\Delta_l}^l(Y) & = & \sum_{k=0}^{t\Delta_l^{-1}+k_1-1} 
\langle Cm_{k\Delta_l}^N, R^{-1} [Y_{(k+1)\Delta_l}-Y_{k\Delta_l}]\rangle -
\sum_{k=0}^{t\Delta_l^{-1}+k_1-1} 
\langle Cm_{k\Delta_l}, R^{-1} [Y_{(k+1)\Delta_l}-Y_{k\Delta_l}]\rangle, \\
R_{t+k_1\Delta_l}^l & = & -\frac{\Delta_l}{2}\sum_{k=0}^{t\Delta_l^{-1}+k_1-1}  \langle m_{k\Delta_l}^N,Sm_{k\Delta_l}^N\rangle +
\frac{\Delta_l}{2}\sum_{k=0}^{t\Delta_l^{-1}+k_1-1}  \langle m_{k\Delta_l},Sm_{k\Delta_l}\rangle.
\end{eqnarray*}
We can decompose the martingale term from \eqref{eq:diff} further through 
$$ 
M_{t+k_1\Delta_l}^l(Y) =  M_{t+k_1\Delta_l}^l(1) + R_{t+k_1\Delta_l}^l(1),
$$
where, 
by setting ${k=t\Delta_l^{-1}+k_1-1}$,
\begin{eqnarray}
\label{eq:M}
M_{t+k_1\Delta_l}^l(1) & = & \sum_{k=0}^{t\Delta_l^{-1}+k_1-1}\sum_{j_1=1}^{d_y}\sum_{j_2=1}^{d_x}\sum_{j_3=1}^{d_y} C(j_1,j_2)\{m_{k\Delta_l}^N(j_2)-m_{k\Delta_l}(j_2)\}
R^{-1}(j_1,j_3)\times \\ & & \{[Y_{(k+1)\Delta_l}-Y_{k\Delta_l}](j_3)-CX_{k\Delta_l}(j_3)\Delta_l\}, \nonumber \\
\label{eq:R}
R_{t+k_1\Delta_l}^l(1) & = & \Delta_l\sum_{k=0}^{t\Delta_l^{-1}+k_1-1}\sum_{j_1=1}^{d_y}\sum_{j_2=1}^{d_x}\sum_{j_3=1}^{d_y}C(j_1,j_2)\{m_{k\Delta_l}^N(j_2)-m_{k\Delta_l}(j_2)\} \times \\&& R^{-1}(j_1,j_3)CX_{k\Delta_l}(j_3). \nonumber
\end{eqnarray}
In order to proceed we construct a martingale associated with the term of $M_t^l$. Let us first begin with the $M_t^l(1)$ term \eqref{eq:M}, where we construct the filtration $(\Omega,\mathscr{F}, \mathscr{F}_{k \Delta_l}, \mathbb{P})$ for our discrete-time martingale $(M^l_t(1), \mathscr{F}_{k \Delta_l})$. 
\\\\
Then by using H\"{o}lder's inequality 
\begin{eqnarray}
\mathbb{E}[|M^l_{t+k_1\Delta_l}(1)|^q]^{1/q} & = &   \mathbb{E} \Big[\Big|\sum_{k=0}^{t\Delta_l^{-1}+k_1-1}\sum_{j_1=1}^{d_y}\sum_{j_2=1}^{d_x}\sum_{j_3=1}^{dy} C(j_1,j_2)\{m_{k\Delta_l}^N(j_2)-m_{k\Delta_l}(j_2)\}
R^{-1}(j_1,j_3) \nonumber  \times 
\\ && \{[Y_{(k+1)\Delta_l}-Y_{k\Delta_l}](j_3)-CX_{k\Delta_l}(j_3)\Delta_l\}\Big|^q\Big]^{1/q} \nonumber \\
 & \leq &   \mathbb{E} \Big[\Big|\sum_{k=0}^{t\Delta_l^{-1}+k_1-1} \sum_{j_1=1}^{d_y}\sum_{j_2=1}^{d_x}\sum_{j_3=1}^{d_y} C(j_1,j_2)\{m_{k\Delta_l}^N(j_2)-m_{k\Delta_l}(j_2)\}
R^{-1}(j_1,j_3)\Big|^{2q}\Big]^{1/2q} \nonumber  \times
 \\ &&  \mathbb{E} \Big[\Big| \sum_{k=0}^{t\Delta_l^{-1}+k_1-1} \sum_{j_3=1}^{dy} \{[Y_{(k+1)\Delta_l}-Y_{k\Delta_l}](j_3)-CX_{k\Delta_l}(j_3)\Delta_l\} \Big|^{2q} \Big]^{1/2q} \nonumber \\
 &=:& T_1 \times T_2.
\label{eq:bound1}
 \end{eqnarray}
For $T_1$ we can apply the Minkowski inequality and  \autoref{lem:mean} to yield
\begin{eqnarray*}
T_1 &\leq& \sum_{k=0}^{t\Delta_l^{-1}+k_1-1} \sum_{j_1=1}^{d_y} \sum_{j_2=1}^{d_x} \sum_{j_3=1}^{d_y} C(j_1,j_2) R^{-1}(j_1,j_3)\mathbb{E} \Big[\Big|\{m_{k\Delta_l}^N(j_2)-m_{k\Delta_l}(j_2)\} \Big|^{2q}\Big]^{1/2q} \\
&\leq &  \frac{\mathsf{C}}{\sqrt{N}}.
\end{eqnarray*}
For $T_2$ we know that the expression $\{[Y_{(k+1)\Delta_l}-Y_{k\Delta_l}](j_3)-CX_{k\Delta_l}(j_3)\Delta_l\}$ is a Brownian motion \\ increment, using the formulae \eqref{eq:data} - \eqref{eq:signal}. Therefore by using the Burkholder--Davis--Gundy inequality, along with Minkowski, for $\tilde{q} = 2q$, we have
\begin{eqnarray*}
T_2 &\leq& \sum_{j_3=1}^{d_y} \mathbb{E}\Big[\Big| \sum^{t\Delta^{-1}_l+k_1-1}_{k=0} [V_{(k+1)\Delta_l} -V_{k\Delta_l}](j_3)\Big|^{\tilde{q}}\Big]^{1/\tilde{q}} \\
&\leq&  \sum_{j_3=1}^{d_y} \sum^{t\Delta^{-1}_l+k_1-1}_{k=0} \mathsf{C}_{\tilde{q}} \mathbb{E}\Big[ \Big|[V_{(k+1)\Delta_l} -V_{k\Delta_l}]^2(j_3)\Big|^{\tilde{q}/2}\Big]^{1/\tilde{q}}  \\
&\leq&  \sum_{j_3=1}^{d_y}  \sum^{t\Delta^{-1}_l+k_1-1}_{k=0}  \mathsf{C}_{\tilde{q}} \Big( \mathbb{E}\Big[\Big|[V_{(k+1)\Delta_l} -V_{k\Delta_l}](j_3)\Big|^{\tilde{q}}\Big]^{2/\tilde{q}} \Big)^{1/2}
\end{eqnarray*}
Then using that fact that $\mathbb{E}[|[V_{(k+1)\Delta_l}-V_{k\Delta_l}]|^{\tilde{q}}] = O(\Delta_l^{\tilde{q}/2})$, and with the summation it is of order $\mathcal{O}(\Delta^{(1/2-1/\tilde{q})})$, we can conclude $T_2$ is of order $\mathcal{O}(1)$, as we have used the case of when $q \geq 1$.
\\\\
For the $R_t^l(1)$ term, it follows similarly to $M_t^l(1)$, where we require the use of \autoref{lem:mean}. 
\\\\
Again we make use of the Minkowski and H\"{o}lder inequality, and \autoref{lem:mean},
\begin{eqnarray}
\nonumber
\mathbb{E}[|R^l_{t+k_1\Delta_l}(1)|^q]^{1/q} & = & \mathbb{E} \Big[ \Big|  \sum_{k=0}^{t\Delta_l^{-1}+k_1-1} \sum_{j_1=1}^{d_x}\sum_{j_2=1}^{d_x}\sum_{j_3=1}^{d_y}C(j_1,j_2)\{m_{k\Delta_l}^N(j_2)-m_{k\Delta_l}(j_2)\}R^{-1}(j_1,j_3)CX_{k\Delta_l}(j_3) \Delta_l \Big|^q \Big]^{1/q} \\
\nonumber
& \leq &  \sum_{k=0}^{t\Delta_l^{-1}+k_1-1}\sum_{j_1=1}^{d_y}\sum_{j_2=1}^{d_x}\sum_{j_3=1}^{d_y}\sum^{d_x}_{j_4=1}C(j_1,j_3)R^{-1}(j_1,j_3) \Big( \mathbb{E}\Big[\Big|m^N_{k \Delta l}(j_2) - m_{k \Delta l}(j_2)\Big|^{2q}\Big]^{1/2q}  \times \\
\nonumber&& \mathbb{E} \Big[\Big|X_{k\Delta_l}(j_4)\Delta_l\Big|^{2q}\Big]^{1/2q}\Big) \\
 \label{eq:bound}
 &\leq&  \sum_{k=0}^{t\Delta_l^{-1}+k_1-1} \sum^{d_x}_{j_4=1}  \frac{\mathsf{C}}{\sqrt{N}} \mathbb{E} \Big[\Big|X_{k\Delta_l}(j_4)\Delta_l\Big|^{2q}\Big]^{1/2q}.
\end{eqnarray}
For the final term of \eqref{eq:bound} we can show it is of order $\mathcal{O}(1)$, using the Cauchy--Schwarz \\ and Jensen's inequality
\begin{eqnarray*}
\mathbb{E} \Big[\Big|X_{k\Delta_l}(j_4)\Delta_l\Big|^{2q}\Big]^{1/2q}  &=& \Delta_l  \mathbb{E}\Big[\Big| \int^{(k+1)\Delta_l}_{k\Delta_l}{\Delta_l^{-1}}X_s(j_4) ds\Big|^{2q}\Big]^{1/2q} \\
&\leq& \Delta_l \mathbb{E} \Big[ \Big( \int^{(k+1)\Delta_l}_{k\Delta_l}\Delta^{-2}_l ds \Big)^q \Big( \int^{(k+1)\Delta_l}_{k\Delta_l} |X_s(j_4)|^2ds\Big)^q \Big]^{1/2q}  \\ 
& =&  \Delta^{1/2}_l \mathbb{E} \Big[  \Big( \int^{(k+1)\Delta_l}_{k\Delta_l} |X_s(j_4)|^2ds\Big)^q \Big]^{1/2q}\\
& \leq &  \Delta_l^{1/2}  \mathbb{E}\Big[{\Delta_l^{q-1}\int^{(k+1)\Delta_l}_{k\Delta_l}}\Big| X_s(j_4) \Big|^{2q} ds\Big]^{1/2q} \\
& \leq & \Delta_l^{1-1/2q}\Big(\int^{(k+1)\Delta_l}_{k\Delta_l}{\mathbb{E}[|X_s(j_4)|^{2q}]}_{}ds\Big)^{1/2q} \\
&\leq& \mathsf{C} \Delta_l^{1-1/2q} (\Delta_l)^{1/2q} \quad  \textrm{(By \cite{DT18} Eq. 54)} \\
&=& \mathsf{C} \Delta_l,
\end{eqnarray*}
therefore combing this with the summation in \eqref{eq:bound}, the above quantity is of order $\mathcal{O}(1)$, resulting in $\mathbb{E}[|R^l_{t+k_1\Delta_l}(1)|^q]^{1/q} \leq \frac{\mathsf{C}}{\sqrt{N}}$. All that is left is the $R^l_{t+k_1\Delta_l}$ term. Before proceeding we can express, or rewrite, the $R_{t+k_1\Delta_l}^l$ term as
\begin{eqnarray*}
R_{t+k_1\Delta_l}^l & = & -\frac{\Delta_l}{2}\sum_{k=0}^{t\Delta_l^{-1}+k_1-1} \langle m_{k\Delta_l}^N,Sm_{k\Delta_l}^N\rangle + \frac{\Delta_l}{2}\sum_{k=0}^{t\Delta_l^{-1}-1} \langle m_{k\Delta_l},Sm_{k\Delta_l}\rangle \\
& =& -\frac{\Delta_l}{2}\sum_{k=0}^{t\Delta_l^{-1}+k_1-1}\sum_{j_1=1}^{d_x}\sum_{j_2=1}^{d_x}\{m_{k\Delta_l}^N(j_1)S(j_1,j_2)m_{k\Delta_l}^N(j_2)-m_{k\Delta_l}(j_1)S(j_1,j_2)m_{k\Delta_l}(j_2)\} \\
&=&  -\frac{\Delta_l}{2}\sum_{k=0}^{t\Delta_l^{-1}+k_1-1}\sum_{j_1=1}^{d_x}\sum_{j_2=1}^{d_x} \{m_{k\Delta_l}^N(j_1)-m_{k\Delta_l}(j_1)\}S(j_1,j_2)m_{k\Delta_l}^N(j_2) \\
& +& m_{k\Delta_l}(j_1)S(j_1,j_2)\{m_{k\Delta_l}^N(j_2)-m_{k\Delta_l}(j_2)\}. 
\end{eqnarray*}
By taking its associated $\mathbb{L}_q$-bound, from Minkowski's inequality and  \autoref{lem:mean} - \autoref{lem:xi_lq}, we have
\begin{eqnarray*}
\mathbb{E}[|R_{t+k_1\Delta_l}^l|^q]^{1/q} & = & \mathbb{E} \Big[ \Big| -\frac{\Delta_l}{2}\sum_{k=0}^{t\Delta_l^{-1}+k_1-1}\sum_{j_1=1}^{d_x}\sum_{j_2=1}^{d_x} \{m_{k\Delta_l}^N(j_1)-m_{k\Delta_l}(j_1)\}S(j_1,j_2)m_{k\Delta_l}^N(j_2) \\
& +& m_{k\Delta_l}(j_1)S(j_1,j_2)\{m_{k\Delta_l}^N(j_2)-m_{k\Delta_l}(j_2)\} \Big|^q\Big]^{1/q} \\
& \leq & -\frac{\Delta_l}{2}\sum_{k=0}^{t\Delta_l^{-1}+k_1-1}\sum_{j_1=1}^{d_x}\sum_{j_2=1}^{d_x} S(j_1,j_2)\Big( \mathbb{E}\Big[\Big|\{m_{k\Delta_l}^N(j_1)-m_{k\Delta_l}(j_1)\}m_{k\Delta_l}^N(j_2)\Big|^{q}\Big]^{1/q} \\
& +& \mathbb{E}\Big[\Big|m_{k\Delta_l}(j_1)\{m_{k\Delta_l}^N(j_2)-m_{k\Delta_l}(j_2)\}\Big|^{q}\Big]^{1/q} \Big) \\
& \leq&  \frac{\mathsf{C}}{\sqrt{N}}.
\end{eqnarray*}
Finally by using the Minkowski inequality, we can deduce that
\begin{eqnarray*}
\mathbb{E}\Big[\Big|[\overline{U}_{t+k_1\Delta_l}^{N,l}(Y) - \overline{U}_{t+k_1\Delta_l}^{l}(Y)]\Big|^q\Big]^{1/q} &=& \mathbb{E}\Big[\Big| M_{t+k_1\Delta_l}^l(1) + R_{t+k_1\Delta_l}^l(1) + R_{t+k_1\Delta_l}^l \Big| ^q \Big]^{1/q} \\
& \leq& \mathbb{E}[|M^l_{t+k_1\Delta_l}(1)|^q]^{1/q} + \mathbb{E}[|R^l_{t+k_1\Delta_l}(1)|^q]^{1/q} + \mathbb{E}[|R^l_{t+k_1\Delta_l}|^q]^{1/q} \\
& \leq & \frac{\mathsf{C}}{\sqrt{N}}.
\end{eqnarray*}

\end{proof}

\section{Analysis for Discretized Diffusion Process}
 \label{Appendix:B}
In this appendix we consider deriving analysis for the discretized diffusion process. This will include both the discretized i.i.d. particle system
\begin{align}
\nonumber
\zeta_{(k+1)\Delta_l}^i &=(I+A\Delta_l)\zeta_{k\Delta_l}^i + Q^{1/2} [\overline{W}_{(k+1)\Delta_l}^i-\overline{W}_{k\Delta_l}^i]  \\
&+ P_{k \Delta_l} C^{\top}R^{-1}\Big([Y_{(k+1)\Delta_l}-Y_{k\Delta_l}] -
\Big[C\zeta_{k\Delta_l}^i\Delta_l+R^{1/2}[\overline{V}_{(k+1)\Delta_l}^i-\overline{V}_{k\Delta_l}^i]\Big]\Big), 
\label{eq:appb_iid}
\end{align}
and the discretized NC estimator. We recall, in the limit as $N \rightarrow \infty$, the i.i.d. system coincides with discretized Kalman--Bucy diffusion whose mean,  for $(k,l)\in\mathbb{N}_0\times\mathbb{N}_0$, is defined by
\begin{equation}\label{eq:deter_mean_evol}
m_{(k+1)\Delta_l}^l =  m_{k\Delta_l}^l + Am^l_{k\Delta_l}\Delta_l + U_{k\Delta_l}^l\Big(
[Y_{(k+1)\Delta_l}-Y_{k\Delta_l}] -Cm_{k\Delta_l}^l\Delta_l\Big),
\end{equation}
We note in this appendix our results will use the notation $\overline{X}$ for the Kalman--Bucy diffusion, to keep it consistent with \cite{CJY20}. However these results also hold for the i.i.d. system \eqref{eq:appb_iid}. We require additional lemmas from \cite{CJY20}, which are discretization bias results for the discretized Kalman--Bucy diffusion. We state these as follows. The notation of the equations are modified for the multilevel, which we will discuss later.
\begin{lem}\label{lem:p_disc}
For any $T\in\mathbb{N}$ fixed and $t\in[0,T]$ there exists a $\mathsf{C}<+\infty$ such that for any $(l,j_1,j_2)\in\mathbb{N}_0\times\{1,\dots,d_x\}^2$:
$$
\Big|\mathcal{P}_{t}(j_1,j_2)-P_{\tau_t^l}^l(j_1,j_2)\Big| \leq \mathsf{C}\Delta_l.
$$
\end{lem}
\begin{lem}\label{lem:strong_error1}
For any $T\in\mathbb{N}$ fixed and $t\in[0,T-1]$ there exists a $\mathsf{C}<+\infty$ such that for any $(l,j,k_1)\in\mathbb{N}_0\times\{1,\dots,d_x\}\times\{0,1,\dots,\Delta_{l}^{-1}\}$:
$$
\mathbb{E}\Big[\Big(\overline{X}_{t+k_1\Delta_{l}}(j)-\overline{X}_{t+k_1\Delta_{l}}^l(j)\Big)^2\Big] \leq \mathsf{C}\Delta_l^2.
$$
\end{lem}

We now present our first conditional bias result, which will be the weak error of the Kalman--Bucy diffusion. This weak error will be analogous to the strong error
of \autoref{lem:strong_error1}, which was not proved, or provided, in \cite{CJY20}. However this result will be required for the current and succeeding appendix.

\begin{lem}\label{lem:weak_error1}
For any $T\in\mathbb{N}$ fixed and $t\in[0,T-1]$ there exists a $\mathsf{C}<+\infty$
 such that for any $(l,k_1)\in\mathbb{N}_0\times\{1,\dots,d_x\}\times\{0,1,\dots,\Delta_{l}^{-1}\}$:
$$
\Big|\mathbb{E}\Big[\overline{X}_{t+k_1\Delta_{l}}(j)-\overline{X}_{t+k_1\Delta_{l}}^l(j)\Big]\Big| \leq \mathsf{C}\Delta_l.
$$
\end{lem}

\begin{proof}
As before we can separate the above expression in different terms,
$$
\mathbb{E}\Big[\overline{X}_{t+k_1\Delta_{l}}(j)-\overline{X}_{t+k_1\Delta_{l}}^l(j)\Big] = T_1+T_2+T_3,
$$
such that, for $\tau_t^l=\lbrack\tfrac{t}{\Delta_l}\rbrack\Delta_l$, $t\in\mathbb{R}^+$, we have
\begin{eqnarray*}
T_1 & = & \mathbb{E}\Big[\int_{0}^{t+k_1\Delta_{l}}\Big(
\sum_{j_1=1}^{d_x} A(j,j_1)[\overline{X}_s(j_1)-\overline{X}_{\tau_s^l}^l(j_1)] + \sum_{j_1=1}^{d_x}\sum_{j_2=1}^{d_x}
\mathcal{P}_s(j,j_1)\hat{C}(j_1,j_2)[\overline{X}_s(j_2) -\overline{X}_{\tau_s^l}^l(j_2)]
\Big)ds
\Big], \\
T_2 & = & \mathbb{E}\Big[\int_{0}^{t+k_1\Delta_{l}}\Big(
 \sum_{j_1=1}^{d_x}\sum_{j_2=1}^{d_x}[\mathcal{P}_s(j,j_1)-P_{\tau_s^l}^l(j,j_1)]\hat{C}(j_1,j_2)\overline{X}_{\tau_s^l}^l(j_2)]\Big)ds\Big],\\
T_3 & = & \mathbb{E}\Big[\sum^{d_x}_{j_1=1}\sum^{d_y}_{j_2=1}\int_{0}^{t+k_1\Delta_{l}}[\mathcal{P}_{s}(j,j_1)-P_{\tau_s^l}^l(j,j_1)]\tilde{C}(j_1,j_2)dY_s(j_2)\Big],
\end{eqnarray*}
where  $\hat{C}=\tilde{C}C$ with $\tilde{C}=C^{\top}R_2^{-1}$. Now let us consider each individual term, \\
firstly for $T_1$, that \eqref{eq:bound_deter_cov} we have the following bound
\begin{equation}\label{eq:weak_T1}
|T_1| \leq \mathsf{C} \int_{0}^{t+k_1\Delta_{l}}\max_{j\in\{1,\dots,d_x\}}
\Big|\mathbb{E}\Big[\overline{X}_{s}(j)-\overline{X}_{\tau_s^l}^l(j)\Big]\Big|ds.
\end{equation}
For $T_2$, We can apply \autoref{lem:p_disc}, using the fact that $\max_{j_2\in\{1,\dots,d_x\}}|\mathbb{E}[\overline{X}_{\tau_s^l}^l(j_2)]|\leq \mathsf{C}$
we have
\begin{equation}\label{eq:weak_T2}
|T_2| \leq \mathsf{C}\Delta_l.
\end{equation}
Similarly for $T_3$, we can use \autoref{lem:p_disc} and  \autoref{lem:y_inc} which provides the bound
\begin{equation}\label{eq:weak_T3}
|T_3| \leq \mathsf{C}\Delta_l.
\end{equation}

Thus combining \eqref{eq:weak_T1}-\eqref{eq:weak_T3} leads to
$$
\max_{j\in\{1,\dots,d_x\}}\Big|\mathbb{E}\Big[\overline{X}_{t+k_1\Delta_{l}}(j)-\overline{X}_{t+k_1\Delta_{l}}^l(j)\Big]\Big| \leq \mathsf{C}\Delta_l + 
\mathsf{C} \int_{0}^{t+k_1\Delta_{l}}\max_{j\in\{1,\dots,d_x\}}
\Big|\mathbb{E}\Big[\overline{X}_{s}(j)-\overline{X}_{\tau_s^l}^l(j)\Big]\Big|ds.
$$
Finally by applying Gr\"{o}nwall's lemma, leads to the desired result.
\end{proof}

We now proceed with our result of the discretized NC estimator, which is the strong error, through the following lemma.
\begin{lem}
\label{lem:strong_error}
For any $T\in\mathbb{N}$ fixed and $t\in[0,T-1]$ there exists a $\mathsf{C}<+\infty$ such that for any $(l,k_1)\in\mathbb{N}_0\times\{0,1,\dots,\Delta_{l}^{-1}\}$:
$$
\mathbb{E}\Big[\Big([\overline{U}_{t+k_1\Delta_l}(Y) - \overline{U}_{t+k_1\Delta_l}^{l}(Y)]\Big)^2\Big]\leq  \mathsf{C} \Delta^2_l.
$$
\end{lem}

\begin{proof}
Let us first recall that, 
\begin{eqnarray*}
\overline{U}^l_{t+k_1\Delta_{l}}(Y) &=& \sum^{t\Delta^{-1}_l+k_1-1}_{k=0}\langle Cm_{k\Delta_l},R^{-1}[Y_{(k+1)\Delta_l} - Y_{k\Delta_l}]\rangle - \frac{\Delta_l}{2}\sum^{t\Delta^{-1}_l+k_1-1}_{k=0} \langle m_{k\Delta_l},S m_{k\Delta_l} \rangle, \\
\overline{U}_{t+k_1\Delta_{l}}(Y) &=& \int^{t+k_1\Delta_l}_0 \Big[\langle Cm_s,R^{-1}dY_s\rangle - \frac{1}{2}\langle m_s,Sm_s\rangle ds\Big].
\end{eqnarray*}
In order to proceed we again consider a martingale-remainder type decomposition. Therefore by setting $\tau_t^l=\lbrack\tfrac{t}{\Delta_l}\rbrack\Delta_l$, $t\in\mathbb{R}^+$, and expanding on the angle brackets, we have
 \begin{eqnarray*}
M_{t+k_1\Delta_l}(1) &=&\sum_{j_1=1}^{d_y}\sum_{j_2=1}^{d_x}\sum_{j_3=1}^{d_y} \int^{t+k_1\Delta_l}_0 R^{-1}(j_1,j_3) C(j_1,j_2)\Big[ m_s(j_2) -  m^l_{\tau_s^l}(j_2) \Big]R^{1/2}(j_1,j_3)dV_s(j_3), \\
R_{t+k_1\Delta_l} &=&\sum_{j_1=1}^{d_x}\sum_{j_2=1}^{d_x}\int^{t+k_1\Delta_l}_{0} \frac{S(j_1,j_2)}{2}\Big[ m_s(j_1)m_s(j_2) -  m^l_{\tau_s^l}(j_1)m^l_{\tau_s^l}(j_2) \Big]ds,\\
R_{t+k_1\Delta_l}(1)&=&\sum_{j_1=1}^{d_y}\sum_{j_2=1}^{d_x}\sum_{j_3=1}^{d_y} \int^{t+k_1\Delta_l}_0 R^{-1}(j_1,j_3) C(j_1,j_2)\Big[ m_s(j_1) -  m^l_{\tau_s^l}(j_1) \Big]CX_s(j_3)ds,
\end{eqnarray*}
where we have used the formula for the observational process \eqref{eq:data}, combined both remainder terms into one, and taken the scaled Brownian motion $V_t$. Let us first consider the remainder term of $R_{t+k_1\Delta_l}(1)$.  Through Jensen's inequality we have
 \begin{eqnarray*}
   \mathbb{E}[R_{t+k_1\Delta_l}(1)]^2 &=& \mathbb{E}\Big[ \sum_{j_1=1}^{d_y}\sum_{j_2=1}^{d_x}\sum_{j_3=1}^{d_y} \int^{t+k_1\Delta_l}_0 R^{-1}(j_1,j_3) C(j_1,j_2)\Big[ m_s(j_2) -  m^l_{\tau_s^l}(j_2) \Big]C X_s(j_3)ds\Big]^2 \\
&\leq& \sum_{j_1=1}^{d_y}\sum_{j_2=1}^{d_x}\sum_{j_3=1}^{d_y}\sum_{j_4=1}^{d_x}R^{-2}(j_1,j_3)C^2(j_1,j_2) \int^{t+k_1\Delta_l}_0 \mathbb{E}\Big[ [m_s(j_2) - m^l_{\tau_s}(j_2)]X_s(j_4)\Big]^2 ds. \\
 \end{eqnarray*}
 Then by using $d^2_xd^2_y$ applications of the $C_2$-inequality we get
  \begin{eqnarray*}
     \mathbb{E}[R_{t+k_1\Delta_l}(1)]^2 &\leq& \mathsf{C}\Big( \int^{t+k_1\Delta_l}_0 \max_{j_2 \in \{1,\ldots,d_x\}}  \mathbb{E}[m_s(j_2) - m^l_{\tau^l_s}(j_2)]^2 \max_{j_4 \in \{1,\ldots,d_x\}}\mathbb{E}[X_s(j_4)]^2 ds \Big).
  \end{eqnarray*}
  We know that $\max_{j_4 \in \{1,\ldots,d_x\}}\mathbb{E}[X_s(j_4)]^2 \leq \mathsf{C}$, therefore all we need is to bound $| m_s(j_1) - m^l_{\tau^l_s}(j_1)|^2$. Therefore by using the fact that $\mathbb{E}[\overline{X}_s(j)] = m_s(j)$ and $  \mathbb{E}[\overline{X}^l_{\tau^l_s}(j)] = m^l_{\tau^l_s}(j) $, we can use the weak error, i.e. \autoref{lem:weak_error1}, to conclude that $\mathbb{E}[R_{t+k_1\Delta_l}(1)]^2 \leq \mathsf{C}\Delta^2_l$.
  \\\\
  Now to proceed with $R_{t+k_1\Delta_l}(Y)$, we can split the \textit{difference of the mean} term
  \begin{eqnarray}
  \nonumber
m_s(j_1)m_s(j_2) - m^l_{\tau^l_s}(j_1)m^l_{\tau^l_s}(j_2) &=& m_s(j_1)m_s(j_2) - m^l_{\tau^l_s}(j_1)m_s(j_2) + m^l_{\tau^l_s}(j_1)m_s(j_2) - m^l_{\tau^l_s}(j_1)m^l_{\tau^l_s}(j_2) \\
\label{eq:bias_mean}
&=& \{m_s(j_1) - m^l_{\tau^l_s}(j_1)\}m_s(j_2) + m^l_{\tau^l_s}(j_1)\{m_s(j_2) - m^l_{\tau^l_s}(j_2)\}.
   \end{eqnarray}
Therefore one can substitute \eqref{eq:bias_mean} into $R_t(Y)$, and by Jensen's and the $C_2$-inequality, results in

   \begin{eqnarray}
   \nonumber
   \mathbb{E}[R_{t+k_1\Delta_l}]^2  &=& \mathbb{E}\Big[ \sum_{j_1=1}^{d_x}\sum_{j_2=1}^{d_x}\int^{t+k_1\Delta_l}_{0} \frac{S(j_1,j_2)}{2}\Big[ \{m_s(j_1) - m^l_{\tau^l_s}(j_1)\}m_s(j_2) + m^l_{\tau^l_s}(j_1)\{m_s(j_2) - m^l_{\tau^l_s}(j_2)\}\Big]ds\Big]^2 \\
   \nonumber
   &\leq& \sum_{j_1=1}^{d_x}\sum_{j_2=1}^{d_x} \frac{S^2(j_1,j_2)}{2} \int^{t+k_1\Delta_l}_0 \mathbb{E}\Big[\{m_s(j_1) - m^l_{\tau^l_s}(j_1)\}m_s(j_2)\Big]^2ds \\
\label{eq:T1_bias}
   &+& \sum_{j_1=1}^{d_x}\sum_{j_2=1}^{d_x} \frac{S^2(j_1,j_2)}{2} \int^{t+k_1\Delta_l}_0 \mathbb{E}\Big[m^l_{\tau^l_s}(j_1)\{m_s(j_2) - m^l_{\tau^l_s}(j_2)\}\Big]^2ds.
     \end{eqnarray}
Therefore by again taking the $C_2$-inequality $d^2_x$ times, and noting that \\ $ \max_{j_2\in\{1,\dots,d_x\}}m_s(j_2) = \max_{j_2\in\{1,\dots,d_x\}}\mathbb{E}[\overline{X}_s(j_2)] \leq \mathsf{C}$, from \autoref{lem:mean} and  similarly done for $m^l_{\tau^l_s}(j_1)$, leads to
   \begin{eqnarray*}
  \mathbb{E}[R_{t+k_1\Delta_l}]^2 \leq \mathsf{C} \Big( \int^{t+k_1\Delta_l}_0 \max_{j_1 \in \{1,\ldots,d_x\}} \mathbb{E}\Big[m_s(j_1) - m^l_{\tau^l_s}(j_1)\Big]^2 +  \max_{j_2 \in \{1,\ldots,d_x\}}
  \mathbb{E} \Big[m_s(j_2) - m^l_{\tau^l_s}(j_2)\Big]^2 ds \Big).
   \end{eqnarray*}
   Through the same substitution as before, and using the weak error, i.e. \autoref{lem:weak_error1},\\ we have that $ \mathbb{E}[R_{t+k_1\Delta_l}]^2 \leq \mathsf{C}\Delta^2_l$.
   \\\\
  Lastly we have the martingale term $M_{t+k_1\Delta_l}(1)$. As before we can apply Jensens inequality 
 \begin{eqnarray*}
 \mathbb{E}[M_{t+k_1\Delta_l}(1)]^2  &=& \mathbb{E}\Big[\sum_{j_1=1}^{d_y}\sum_{j_2=1}^{d_x}\sum_{j_3=1}^{d_y} \int^{t+k_1\Delta_l}_0 R^{-1}(j_1,j_3) C(j_1,j_2)\Big[ m_s(j_2) -  m^l_{\tau_s^l}(j_2) \Big]R^{1/2}(j_1,j_3)dV_s(j_3)\Big]^2 \\
& \leq &\sum_{j_1=1}^{d_y}\sum_{j_2=1}^{d_x}\sum_{j_3=1}^{d_y} {R}^{-1}(j_1,j_3) \ C^2(j_1,j_2)\int^{t+k_1\Delta_l}_0\mathbb{E}\Big[ m_s(j_2) -  m^l_{\tau_s^l}(j_2) dV_s(j_3) \Big]^2.
 \end{eqnarray*}
Then by using the Ito isometry and $d_xd^2_y$ applications of the $C_2$-inequality, we have 
\begin{eqnarray*}
\mathbb{E}[M_{t+k_1\Delta_l}(1)]^2 & \leq& \mathsf{C} \int^{t+k_1\Delta_l}_0 \max_{j_2 \in \{1,\ldots,d_x\}} \mathbb{E} \Big[m_s(j_2) - m^l_{\tau^l_s}(j_2)\Big]^2ds,
\end{eqnarray*}
         and, as before, by using \autoref{lem:weak_error1} as done for $R_{t+k_1\Delta_l}$, we can conclude that $ \mathbb{E}[M_{t+k_1\Delta_l}(1)]^2 \leq \mathsf{C} \Delta^2_l$. 
   Therefore by combining all terms and a further application of the $C_2$-inequality three times
   \begin{eqnarray*}
\mathbb{E}\Big[\Big([\overline{U}_{t+k_1\Delta_l}(Y) - \overline{U}_{t+k_1\Delta_l}^{l}(Y)]\Big)^2\Big] &\leq& \mathbb{E}[M_{t+k_1\Delta_l}(1)]^2 +    \mathbb{E}[R_{t+k_1\Delta_l}(1)]^2 +    \mathbb{E}[R_{t+k_1\Delta_l}]^2 \\
   &\leq& \mathsf{C}\Delta^2_l.
   \end{eqnarray*}
   
\end{proof}

\section{Analysis for i.i.d. MLEnKBF NC Estimator}
 \label{Appendix:C}
We now discuss the analysis, related to the variance, of both the NC estimator using the EnKBF and the i.i.d. MLEnKBF. This will lead onto the proof of 
our main result, presented as \autoref{theo:main}. We note that in our notations, we extend the case of the discretized EnKBF, to the discretized MLEnKBF, 
by adding superscripts $l$ as above. Specifically the analysis now considers the i.i.d. couple particle system
\begin{align*}
\zeta_{(k+1)\Delta_l}^{i,l} & = \zeta_{k\Delta_l}^{i,l} + A\zeta_{k\Delta_l}^{i,l}\Delta_l + Q^{1/2} [\overline{W}_{(k+1)\Delta_l}^i-\overline{W}_{k\Delta_l}^i]  
\\&+ P^{N,l}_{k \Delta_l} C^{\top}R^{-1}\Big([Y^{{i}}_{(k+1)\Delta_l}-Y^{{i}}_{k\Delta_l}]
-\Big[C\zeta_{k\Delta_l}^{i,l}\Delta_l + R^{1/2}[\overline{V}_{(k+1)\Delta_l}^i-\overline{V}_{k\Delta_l}^i]\Big]\Big), \\
\zeta_{(k+1)\Delta_{l-1}}^{i,l-1} & =  \zeta_{k\Delta_{l-1}}^{i,l-1} + A\zeta_{k\Delta_{l-1}}^{i,l-1}\Delta_{l-1} + Q^{1/2} [\overline{W}_{(k+1)\Delta_{l-1}}^i-\overline{W}_{k\Delta_{l-1}}^i] 
\\&+ P^{N,l-1}_{k \Delta_{l-1}} C^{\top}R^{-1}\Big([Y^{{i}}_{(k+1)\Delta_{l-1}}-Y^{{i}}_{k\Delta_{l-1}}] -\Big[C\zeta_{k\Delta_{l-1}}^{i,l-1}\Delta_{l-1}+ R^{1/2}[\overline{V}_{(k+1)\Delta_{l-1}}^i-\overline{V}_{k\Delta_{l-1}}^i]\Big]\Big), 
\end{align*}
 within the NC estimator.

We will use the fact that the i.i.d. system coincides with the Kalman--Bucy Diffusion $\overline{X}_t$, in the limit of $N \rightarrow \infty$. 
This implies the mean and covariance are defined through the Kalman--Bucy filter and the Ricatti equations, which allows us to use
the results from \autoref{Appendix:B}, for the process $\zeta^i_{k\Delta_l}$.

\subsection{MSE bound on EnKBF NC estimator}

Here we use the notation:  for a $d_x-$dimensional vector $x$ denote $\|x\|_2=(\sum_{j=1}^{d_x}x(j)^2)^{1/2}$.
\begin{proposition}\label{prop:var_term1}
{For any $T\in\mathbb{N}$ fixed and $t\in[0,T-1]$ there exists a $\mathsf{C}<+\infty$ such that for any $(l,N,k_1)\in\mathbb{N}_0\times\{2,3,\dots\}\times\{0,1,\dots,\Delta_l^{-1}\}$:
$$
\mathbb{E}\Big[\Big\|[\widehat{\overline{U}}_{t+k_1\Delta_l}^{N,l}-\overline{U}_{t+k_1\Delta_l}](Y)\Big\|_2^2\Big] \leq \mathsf{C}\Big(\frac{1}{N}+\Delta_l^2\Big).
$$}
\end{proposition}

\begin{proof}
Using the $C_2-$inequality one  has
\begin{align}
\mathbb{E}\Big[\Big\|[\widehat{\overline{U}}_{t+k_1\Delta_l}^{N,l}-{\overline{U}}_{t+k_1\Delta_l}](Y)\Big\|_2^2\Big] &\leq 
\label{eq:first_prop_1}
\mathsf{C}\Big(
\mathbb{E}\Big[\Big\|[\widehat{\overline{U}}_{t+k_1\Delta_l}^{N,l}-{\overline{U}}_{t+k_1\Delta_l}^l](Y)\Big\|_2^2\Big] + 
\mathbb{E}\Big[\Big\|[{\overline{U}}_{t+k_1\Delta_l}^{l}-{\overline{U}}_{t+k_1\Delta_l}](Y)\Big\|_2^2\Big] 
\Big).
\end{align}

{The first term on the R.H.S.~can be controlled by standard results for i.i.d.~sampling (recall that $\zeta_{t+k_1\Delta_l}^i|\mathscr{F}_{t+k_1\Delta_l}$ are i.i.d.~Gaussian with mean $m_{t+k_1\Delta_l}$ and covariance $P_{t+k_1\Delta_l}$), that is
\begin{equation}\label{eq:first_prop_2}
\mathbb{E}\Big[\Big\|[{\widehat{\overline{U}}}_{t+k_1\Delta_l}^{N,l}-{\overline{U}}_{t+k_1\Delta_l}^l](Y)\Big\|_2^2\Big] \leq \frac{\mathsf{C}}{N}.
\end{equation}
{The formula in \eqref{eq:first_prop_2} can be proved by using the formulae for the NC estimators, in the usual integral form, and through a simple
application of the general Minkowski inequality.}
\\
Note that it is crucial that \eqref{eq:bound_deter_cov} holds, otherwise the upper-bound can explode as a function of $l$.}
For the right-most term on the R.H.S.~of \eqref{eq:first_prop_1} by Jensen's inequality and \autoref{lem:strong_error}:
\begin{equation}\label{eq:main_bias}
\mathbb{E}\Big[\Big\|[{\overline{U}}_{t+k_1\Delta_l}^{l}-{\overline{U}}_{t+k_1\Delta_l}](Y)\Big\|_2^2\Big]  \leq \mathsf{C}\Delta_l^2.
\end{equation}
So the proof can be concluded by combining \eqref{eq:first_prop_1}, \eqref{eq:first_prop_2} and \eqref{eq:main_bias}.

\end{proof}

\subsection{Variance of i.i.d. MLEnKBF NC estimator }

\begin{proposition}
\label{prop:var_term2}{
For any  $(t,q)\in \mathbb{N}_0 \times [1,\infty)$, there exists a $\mathsf{C}<+\infty$ such that for any $(l,N,k_1)\in\mathbb{N}\times\{2,3,\dots\}\times\{0,1,\dots,\Delta_{l-1}^{-1}-1\}$:
\begin{equation}
\label{eq:mul}
\mathbb{E}\Big[\Big|[\widehat{\overline{U}}_{t+k_1\Delta_{l-1}}^{N,l}-\widehat{\overline{U}}_{t+k_1\Delta_{l-1}}^{N,l-1}] - [{\overline{U}}_{t+k_1\Delta_{l-1}}^{l}-{\overline{U}}_{t+k_1\Delta_{l-1}}^{l-1}]\Big|^q\Big]^{1/q} \leq \frac{\mathsf{C}\Delta^{1/2}_l}{\sqrt{N}}.
\end{equation}}
\end{proposition}

As before, in order to proceed we will make use of a martingale-remainder decomposition. 
\\
Recall that, for level $s \in \{l-1,l\}$, we have
\begin{align*}
\overline{U}_{t+k_1\Delta_{l-1}}^{N,s}(Y) - \overline{U}_{t+k_1\Delta_{l-1}}^{s}(Y) &= M_{t+k_1\Delta_{l-1}}^s(Y) + R_{t+k_1\Delta_{l-1}}^s \\
&= M^s_{t+k_1\Delta_{l-1}}(1) + R^s_{t+k_1\Delta_{l-1}}(1) + R^s_{t+k_1\Delta_{l-1}},
\end{align*}
Therefore substituting into the LHS of  \eqref{eq:mul}, and using Minkowski's inequality, results in
\begin{align*}
&\mathbb{E}\Big[\Big|[\widehat{\overline{U}}_{t+k_1\Delta_{l-1}}^{N,l}-\widehat{\overline{U}}_{t+k_1\Delta_{l-1}}^{N,l-1}] - [{\overline{U}}_{t+k_1\Delta_{l-1}}^{l}-{\overline{U}}_{t+k_1\Delta_{l-1}}^{l-1}]\Big|^q\Big]^{1/q}  \\ &=
\mathbb{E}\Big[\Big| (R_{t+k_1\Delta_{l-1}}^{l}(1) -R_{t+k_1\Delta_{l-1}}^{l-1}(1))  +  (M_{t+k_1\Delta_{l-1}}^{l}(1) -M_{t+k_1\Delta_{l-1}}^{l-1}(1))\\ &+(R_{t+k_1\Delta_{l-1}}^{l}(Y) -R_{t+k_1\Delta_{l-1}}^{l-1}(Y)) \Big|^q\Big]^{1/q} \\ &\leq 
\mathbb{E}\Big[\Big| (R_{t+k_1\Delta_{l-1}}^{l}(1) -R_{t+k_1\Delta_{l-1}}^{l-1}(1)) \Big|^q\Big]^{1/q}+ \mathbb{E}\Big[\Big| (M_{t+k_1\Delta_{l-1}}^{l}(1) -M_{t+k_1\Delta_{l-1}}^{l-1}(1)) \Big|^q\Big]^{1/q}  \\&+ \mathbb{E}\Big[\Big| (R_{t+k_1\Delta_{l-1}}^{l}(Y) -R_{t+k_1\Delta_{l-1}}^{l-1}(Y)) \Big|^q\Big]^{1/q}.
\end{align*}

Therefore in order to prove \autoref{prop:var_term2} we will split it into three lemmas, which are stated and proved below.

\begin{lem}
\label{lem:R1}
For any  $(t,q)\in \mathbb{N}_0 \times [1,\infty)$, there exists a $\mathsf{C}<+\infty$ such that for any $(l,N,k_1)\in\mathbb{N}\times\{2,3,\dots\}\times\{0,1,\dots,\Delta_{l-1}^{-1}-1\}$:
\begin{equation}
\label{eq:R1}
\mathbb{E}\Big[\Big| \Big(R_{t+k_1\Delta_{l-1}}^{l}(1) -R_{t+k_1\Delta_{l-1}}^{l-1}(1)\Big) \Big|^q\Big]^{1/q} \leq  \frac{\mathsf{C}\Delta^{1/2}_l}{\sqrt{N}}.
\end{equation}
\end{lem}

\begin{proof}
\begin{eqnarray*}
&&\mathbb{E}\Big[\Big| \Big(R_{t+k_1\Delta_{l-1}}^{l}(1) -R_{t+k_1\Delta_{l-1}}^{l-1}(1)\Big) \Big|^q\Big]^{1/q}\\ 
 &=& \mathbb{E}\Big[\Big|\sum^{d_y}_{j_1=1}\sum_{j_2=1}^{d_x}\sum_{j_3=1}^{d_y} \sum^{t\Delta^{-1}_{l-1}+k_1-1}_{k=0} C(j_1,j_2)  \Big(m^{N,l}_{k \Delta_{l-1}}(j_2) - m^l_{k \Delta_{l-1}}(j_2)\Big) R^{-1}(j_1,j_3)  CX_{k\Delta_{l-1}}(j_3)\Delta_l 
\\&-&  \sum^{d_y}_{j_1=1}\sum_{j_2=1}^{d_x}\sum_{j_3=1}^{d_y}\sum^{t\Delta^{-1}_{l-1}+k_1-1}_{k=0} C(j_1,j_2)  \Big(m^{N,l-1}_{k \Delta_{l-1}}(j_2) - m^{l-1}_{k \Delta_{l-1}}(j_2)\Big)R^{-1}(j_1,j_3) CX_{k\Delta_{l-1}}(j_3)\Delta_{l-1}   \Big|^q\Big]^{1/q}.
\end{eqnarray*}
Then through the generalized Minkowski inequality
\begin{eqnarray*}
\mathbb{E}\Big[\Big| \Big(R_t^{l}(1) -R_t^{l-1}(1)\Big) \Big|^q\Big]^{1/q}
 &\leq&  \sum_{j_1=1}^{d_y}\sum_{j_2=1}^{d_x}\sum_{j_3=1}^{d_y}\sum_{j_4=1}^{d_x} C(j_1,j_2)C(j_3,j_4)R^{-1}(j_1,j_3) \times \\
 && \mathbb{E}\Big[\Big|   \int^{t+k_1\Delta_{l-1}}_0 \Big( \overline{m}^{N,l}_{\tau^l_s}(j_2) X_{\tau_s^l}(j_4) - \overline{m}^{N,l-1}_{\tau^{l-1}_s}(j_2) X_{\tau_s^{l-1}}(j_4) \Big) ds\Big|^q\Big]^{1/q} \\
 \\
&\leq &   \mathsf{C}\sum_{j_2=1}^{d_x}\sum_{j_4=1}^{d_x}  \Big( {\mathbb{E}\Big[\Big| \int^{t+k_1\Delta_{l-1}}_0\Big(\overline{m}^{N,l}_{\tau^l_s}(j_2)  - \overline{m}^{N,l-1}_{\tau^{l-1}_s}(j_2) \Big)X_{\tau^{l}_s}(j_4) ds   \Big|^q\Big]^{1/q}} \\ &+&  {\mathbb{E}\Big[\Big| \int^{t+k_1\Delta_{l-1}}_0  \overline{m}^{N,l-1}_{\tau^{l-1}_s}(j_2) \Big(X_{\tau^l_s}(j_4)-X_{\tau^{l-1}_s}(j_4)\Big)
ds   \Big|^q\Big]^{1/q}}\Big) \\
&=:& T_1 + T_2,
\end{eqnarray*}
where we have used $\tau^s_t = \lfloor \frac{t}{\Delta_s} \rfloor \Delta_s $ for $t \in \mathbb{R}^+$, and
\begin{equation}
\label{eq:md}
\overline{m}^{N,s}_{k\Delta_l} = {m}^{N,s}_{k\Delta_l} - m^l_{k\Delta_s}, \quad s \in \{l-1,l\}.
\end{equation}
For $T_1$ we can express it as
\begin{eqnarray*}
T_1&=& \mathsf{C} \sum_{j_2=1}^{d_x}\sum_{j_4=1}^{d_x} \mathbb{E}\Big[\Big| \int^{t+k_1\Delta_{l-1}}_0 \Big(\overline{m}^{N,l}_{\tau^l_s}(j_2) - \overline{m}^{N,l}_{\tau^{l-1}_s}(j_2) + \overline{m}^{N,l}_{\tau^{l-1}_s}(j_2) - \overline{m}^{N,l-1}_{\tau^{l-1}_s}(j_2)\Big) X_{\tau^l_s}(j_4)ds \Big|^q\Big]^{1/q} \\ 
&\leq&  \mathsf{C} \sum_{j_2=1}^{d_x}\sum_{j_4=1}^{d_x}\int^{t+k_1\Delta_{l-1}}_0 \Big( \mathbb{E}\Big[\Big| \Big(\overline{m}^{N,l}_{\tau^l_s}(j_2) - \overline{m}^{N,l}_{\tau^{l-1}_s}(j_2)\Big)X_{\tau^l_s}(j_4) \Big|^q\Big]^{1/q} 
\\&+&   \mathbb{E}\Big[\Big| \Big( \overline{m}^{N,l}_{\tau^{l-1}_s}(j_2) - \overline{m}^{N,l-1}_{\tau^{l-1}_s}(j_2)\Big) X_{\tau^l_s}(j_4) \Big|^q\Big]^{1/q}\Big)ds,
\\&=:&T_3+T_4.
\end{eqnarray*}
again by using the generalized Minkowski inequality and Jensen's inequality.
\\\\
For $T_3$ we can apply the Marcinkiewicz--Zygmund and H\"{o}lder inequalities, and using the fact that means can be expressed as the expectations of \eqref{eq:iid1} - \eqref{eq:iid2}
\begin{eqnarray*}
T_3 &=&   \mathsf{C} \sum_{j_2=1}^{d_x}\sum_{j_4=1}^{d_x}\int^{t+k_1\Delta_{l-1}}_0 \mathbb{E} \Big[ \Big| \frac{1}{N}
 \sum^N_{i=1}\Big[\Big({\zeta}^{i,l}_{\tau^l_s}(j_2) - {\zeta}^{i,l}_{\tau^{l-1}_s}(j_2)\Big) - \Big(m^l_{\tau^{l}_s}(j_2)-m^l_{\tau^{l-1}_s}(j_2)\Big)\Big] X_{\tau^l_s}(j_4) \Big|^q\Big]^{1/q} ds \\
 &\leq&   \mathsf{C} \sum_{j_2=1}^{d_x}\sum_{j_4=1}^{d_x}\int^{t+k_1\Delta_{l-1}}_0 \mathbb{E} \Big[ \Big| \frac{1}{N}
 \sum^N_{i=1}\Big[\Big({\zeta}^{i,l}_{\tau^l_s}(j_2) - {\zeta}^{i,l}_{\tau^{l-1}_s}(j_2)\Big) - \Big(m^l_{\tau^{l}_s}(j_2)-m^l_{\tau^{l-1}_s}(j_2)\Big)\Big]  \Big|^{2q}\Big]^{1/2q} \times \\ &&\mathbb{E} \Big[ \Big|X_{\tau^l_s}(j_4) \Big|^{2q}\Big]^{1/2q} ds \\
   &\leq& \frac{\mathsf{C}_q}{\sqrt{N}} \sum_{j_2=1}^{d_x}\sum_{j_4=1}^{d_x} \int^{t+k_1\Delta_{l-1}}_0  \mathbb{E} \Big[ \Big| [{\zeta}^l_{\tau^l_s}(j_2) - {\zeta}^l_{\tau^{l-1}_s}(j_2)]\Big|^{2q}\Big]^{1/2q} \mathbb{E} \Big[ \Big|X_{\tau^l_s}(j_4) \Big|^{2q}\Big]^{1/2q} ds.
\end{eqnarray*}
The process $\mathbb{E}[|X_{\tau^l_s}]^{2q}]$ is of order $\mathcal{O}(1)$ and and the recursion is of order $\mathcal{O}({\Delta_l})$, through the strong error \autoref{lem:strong_error1}.
 For $T_4$ we know it is sufficiently small, which is of order $\mathcal{O}(\Delta_l)$. 
\\\\
For $T_2$, we use the definition of the discretized diffusion process
\begin{equation}
\label{eq:rem2}
X_{\tau^l_s} - X_{\tau^{l-1}_s} = \int^{\tau^l_s}_{\tau^{l-1}_s} AX_{u} du  + Q^{1/2}\Big(W_{\tau^{l}_s} - W_{\tau^{l-1}_s}\Big).
\end{equation}

{As before, we know the difference of the Brownian motion increment $\mathbb{E}\Big[\Big|W_{\tau^{l}_s} - W_{\tau^{l-1}_s}\Big|^q \Big]$ is of order $\mathcal{O}(\Delta^{1/2}_l)$,
and, as before, $\mathbb{E}[|X_u|^q]^{1/q} \leq \mathsf{C}$}. Therefore the integral term of \eqref{eq:rem2} is of order $\mathcal{O}(\Delta^{1/2}_l)$. Finally for $ \mathbb{E}[|\overline{m}^{N,l-1}_{\tau^{l-1}_s}|^q]^{1/q}$, as it is of order
$\mathcal{O}(N^{-\frac{1}{2}})$, therefore, combining all terms, we can deduce from that 
$$
\mathbb{E}\Big[\Big| \Big(R_{t+k_1\Delta_{l-1}}^{l}(1) -R_{t+k_1\Delta_{l-1}}^{l-1}(1)\Big) \Big|^q\Big]^{1/q} \leq \frac{\mathsf{C}\Delta^{1/2}_l}{\sqrt{N}}.
$$
\end{proof}

\begin{lem}
\label{lem:M1}
For any $(t,q)\in \mathbb{N}_0 \times [1,\infty)$, there exists a $\mathsf{C}<+\infty$ such that for any $(l,N,k_1)\in\mathbb{N}\times\{2,3,\dots\}\times\{0,1,\dots,\Delta_{l-1}^{-1}-1\}$:
\begin{equation}
\label{eq:M1}
\mathbb{E}\Big[\Big| \Big(M_{t+k_1\Delta_{l-1}}^{l}(1) -M_{t+k_1\Delta_{l-1}}^{l-1}(1)\Big) \Big|^q\Big]^{1/q} \leq  \frac{\mathsf{C}\Delta_l^{1/2}}{\sqrt{N}}.
\end{equation}
\end{lem}

\begin{proof}
As before, we set  $\tau^s_t = \lfloor \frac{t}{\Delta_s} \rfloor \Delta_s $ for $t \in \mathbb{R}^+$, and  make use of \eqref{eq:md}
\begin{eqnarray*}
&&\mathbb{E}\Big[\Big| \Big(M_{t+k_1\Delta_{l-1}}^{l}(1) -M_{t+k_1\Delta_{l-1}}^{l-1}(1)\Big) \Big|^q\Big]^{1/q} \\
 &=&\mathbb{E}\Big[\Big|\sum^{d_y}_{j_1=1}\sum_{j_2=1}^{d_x}\sum_{j_3=1}^{d_y}\Big( \sum^{t\Delta^{-1}_{l-1}+k_1-1}_{k=0} C(j_1,j_2)\Big(\overline{m}^{N,l}_{k \Delta_{l-1}}(j_2)\Big)
 R^{-1}(j_1,j_3)\times \\&&\Big([Y_{(k+1)\Delta_{l-1}}-Y_{k\Delta_{l-1}}](j_3)-CX_{k\Delta_{l-1}}(j_3)\Delta_l\Big) \\
& -&   \sum^{t\Delta^{-1}_{l-1}+k_1-1}_{k=0} C(j_1,j_2) \Big(\overline{m}^{N,l-1}_{k \Delta_{l-1}}(j_2)\Big)R^{-1}(j_1,j_3)\Big([Y_{(k+1)\Delta_{l-1}}-Y_{k\Delta_{l-1}}](j_3)-CX_{k\Delta_{l-1}}(j_3)\Delta_{l-1}\Big) \Big) \Big|^q\Big]^{1/q} \\ 
\end{eqnarray*}
Then by using generalized Minkowski and Jensen's inequality 
\begin{eqnarray*}
&&\mathbb{E}\Big[\Big| \Big(M_{t+k_1\Delta_{l-1}}^{l}(1) -M_{t+k_1\Delta_{l-1}}^{l-1}(1)\Big) \Big|^q\Big]^{1/q} \\
&=& \mathbb{E}\Big[\Big|\sum_{j_1=1}^{d_y}\sum_{j_2=1}^{d_x}\sum^{d_y}_{j_3=1}C(j_1,j_2)R^{-1}(j_1,j_3)\int^{t+k_1\Delta_{l-1}}_0 \Big( \overline{m}^{N,l}_{\tau^l_s}(j_2)- \overline{m}^{N,l-1}_{\tau^{l-1}_s}(j_2)  \Big) dY_s(j_3) \\&-&  \Big( \overline{m}^{N,l}_{\tau^l_s}(j_2) CX_{\tau_s^l}(j_3) - \overline{m}^{N,l-1}_{\tau^{l-1}_s}(j_2) CX_{\tau_s^{l-1}}(j_3) \Big) ds\Big|^q\Big]^{1/q} \\
&\leq& \mathsf{C} \sum_{j_2=1}^{d_x}\sum_{j_3=1}^{d_y}\sum_{j_4=1}^{d_x}   \mathbb{E}\Big[\Big|   \int^{t+k_1\Delta_{l-1}}_0 \Big( \overline{m}^{N,l}_{\tau^l_s}(j_2)- \overline{m}^{N,l-1}_{\tau^{l-1}_s}(j_2)  \Big) X_{\tau^{l}_s}(j_4)ds + \Big( \overline{m}^{N,l}_{\tau^l_s}(j_2)- \overline{m}^{N,l-1}_{\tau^{l-1}_s}(j_2)  \Big)dV_s(j_3)
\\ & -&  \Big( \overline{m}^{N,l}_{\tau^l_s}(j_2) X_{\tau_s^l}(j_4) - \overline{m}^{N,l-1}_{\tau^{l-1}_s}(j_2) X_{\tau_s^{l-1}}(j_4) \Big) ds\Big|^q\Big]^{1/q} \\
& \leq & \mathsf{C} \sum_{j_2=1}^{d_x}\sum_{j_3=1}^{d_y}\sum_{j_4=1}^{d_x}\int^{t+k_1\Delta_{l-1}}_0 \Big(\mathbb{E}\Big[\Big| -  \overline{m}^{N,l-1}_{\tau^{l-1}_s}(j_2)\Big( X_{\tau_s^l}(j_4) -  X_{\tau_s^{l-1}}(j_4) \Big) ds\Big|^q\Big]^{1/q} \\&+& \mathbb{E}\Big[\Big| \Big( \overline{m}^{N,l}_{\tau^l_s}(j_2)- \overline{m}^{N,l-1}_{\tau^{l-1}_s}(j_2)  \Big) dV_s(j_3)\Big|^q\Big]^{1/q}\Big)  \\
&=:& T_2+T_{3}.
\end{eqnarray*}

For $T_2$, we know it follows the same analysis as \autoref{lem:R1}. Therefore we can conclude that
$$
T_2 \leq \frac{\mathsf{C} \Delta^{1/2}_l}{\sqrt{N}}.
$$

For $T_{3}$ by using the Burkholder--Davis--Gundy and H\"{o}lder inequality, as done previously from \eqref{eq:bound1}, and we use the bound from \autoref{lem:R1}, for the
$\Big(\overline{m}^{N,l}_{\tau^l_s}- \overline{m}^{N,l-1}_{\tau^{l-1}_s}\Big)$ term. Therefore this implies that
$$
\mathbb{E}\Big[\Big| \Big(M_{t+k_1\Delta_{l-1}}^{l}(1) -M_{t+k_1\Delta_{l-1}}^{l-1}(1)\Big) \Big|^q\Big]^{1/q} \leq \frac{\mathsf{C} \Delta_l^{1/2}}{\sqrt{N}}.
$$
\end{proof}

\begin{lem}
\label{lem:R2}

For any $(t,q)\in \mathbb{N}_0 \times [1,\infty)$, there exists a $\mathsf{C}<+\infty$ such that for any $(l,N,k_1)\in\mathbb{N}\times\{2,3,\dots\}\times\{0,1,\dots,\Delta_{l-1}^{-1}-1\}$:
\begin{equation}
\label{eq:R2}
\mathbb{E}\Big[\Big| \Big(R_{t+k_1\Delta_{l-1}}^{l} -R_{t+k_1\Delta_{l-1}}^{l-1}\Big) \Big|^q\Big]^{1/q} \leq  \frac{\mathsf{C}\Delta_l}{\sqrt{N}}.
\end{equation}
\end{lem}

\begin{proof}
Again we let $\tau^s_t = \lfloor \frac{t}{\Delta_s} \rfloor \Delta_s $ for $t \in \mathbb{R}^+$, and using the generalized Minkowski inequality
\begin{eqnarray*}
&&\mathbb{E}\Big[\Big| \Big(R_{t+k_1\Delta_l}^{l} -R_{t+k_1\Delta_l}^{l-1}\Big) \Big|^q\Big]^{1/q} \\  &=&  \Big(\mathbb{E}\Big[\Big| \sum_{j_1=1}^{d_x}\sum_{j_2=1}^{d_x} \frac{S(j_1,j_2)}{2}  \sum^{t\Delta^{-1}_{l-1}+k_1-1}_{k=0}  \Big(m^{N,l}_{k \Delta_{l-1}}(j_1)m^{N,l}_{k \Delta_{l-1}}(j_2) - m^l_{k \Delta_{l-1}}(j_1)m^l_{k \Delta_{l-1}}(j_2)\Big)
\\ &-&   \sum^{t\Delta^{-1}_{l-1}+k_1-1}_{k=0}  \ \Big(m^{N,l-1}_{k \Delta_{l-1}}(j_1)m^{N,l-1}_{k \Delta_{l-1}}(j_2) - m^l_{k \Delta_{l-1}}(j_1)m^l_{k \Delta_{l-1}}(j_2)\Big)\Big|^q\Big]^{1/q} \Big) \\
&=& \sum_{j_1=1}^{d_x}\sum_{j_2=1}^{d_x} \frac{S(j_1,j_2)}{2}  \Big(\mathbb{E}\Big[\Big| \int^{t+k_1\Delta_{l-1}}_0 \Big(m^{N,l}_{\tau^l_s}(j_1)m^{N,l}_{\tau^l_s}(j_2) - m^l_{\tau^{l}_s}(j_1)m^l_{\tau^{l}_s}(j_2)\Big)  \\
&-& \Big(m^{N,l}_{\tau^{l-1}_s}(j_1)m^{N,l}_{\tau^{l-1}_s}(j_2) - m^l_{\tau^{l-1}_s}(j_1)m^l_{\tau^{l-1}_s}(j_2)\Big) \\&+& \Big(m^{N,l}_{\tau^{l-1}_s}(j_1)m^{N,l}_{\tau^{l-1}_s}(j_2) - m^l_{\tau^{l-1}_s}(j_1)m^l_{\tau^{l-1}_s}(j_2)\Big) \\
&-&   \Big(m^{N,l-1}_{\tau^{l-1}_s}(j_1)m^{N,l-1}_{\tau^{l-1}_s}(j_2) - m^l_{\tau^{l-1}_s}(j_1)m^l_{\tau^{l-1}_s}(j_2)\Big)  ds \Big|^q\Big]^{1/q} \Big) \\
&\leq&  \mathsf{C}  \sum_{j_1=1}^{d_x}\sum_{j_2=1}^{d_x}  \Big(\mathbb{E}\Big[\Big| \int^{t+k_1\Delta_{l-1}}_0 \Big(m^{N,l}_{\tau^l_s}(j_1)m^{N,l}_{\tau^l_s}(j_2) - m^l_{\tau^{l}_s}(j_1)m^l_{\tau^{l}_s}(j_2)\Big) \\
&-&\Big(m^{N,l}_{\tau^{l-1}_s}(j_1)m^{N,l}_{\tau^{l-1}_s}(j_2) - m^l_{\tau^{l-1}_s}(j_1)m^l_{\tau^{l-1}_s}(j_2)\Big) ds \Big|^q\Big]^{1/q} \Big) \\
&+& \mathsf{C} \sum_{j_1=1}^{d_x}\sum_{j_2=1}^{d_x}  \Big(\mathbb{E}\Big[\Big| \Big(m^{N,l}_{\tau^{l-1}_s}(j_1)m^{N,l}_{\tau^{l-1}_s}(j_2) - m^l_{\tau^{l-1}_s}(j_1)m^l_{\tau^{l-1}_s}(j_2)\Big) \\
&-& \Big(m^{N,l-1}_{\tau^{l-1}_s}(j_1)m^{N,l-1}_{\tau^{l-1}_s}(j_2) - m^l_{\tau^{l-1}_s}(j_1)m^l_{\tau^{l-1}_s}(j_2)\Big)  ds \Big|^q\Big]^{1/q} \Big) \\
&=:& T_1 + T_2.
\end{eqnarray*}
For $T_1$ we make use of \eqref{eq:md}
 and with the generalized Minkowski inequality
\begin{eqnarray*}
T_1 &=& \mathsf{C}  \sum_{j_1=1}^{d_x}\sum_{j_2=1}^{d_x}  \Big(\mathbb{E}\Big[\Big| \int^{t+k_1\Delta_{l-1}}_0 \Big(m^{N,l}_{\tau^l_s}(j_1)m^{N,l}_{\tau^l_s}(j_2) - m_{\tau^{l}_s}(j_1)m_{\tau^{l}_s}(j_2)\Big) \\
&-&\Big(m^{N,l}_{\tau^{l-1}_s}(j_1)m^{N,l}_{\tau^{l-1}_s}(j_2) - m_{\tau^{l-1}_s}(j_1)m_{\tau^{l-1}_s}(j_2)\Big) ds \Big|^q\Big]^{1/q} \Big) \\
 &=&  \mathsf{C} \sum_{j_1=1}^{d_x}\sum_{j_2=1}^{d_x}   \Big(\mathbb{E}\Big[\Big| \int^{t+k_1\Delta_{l-1}}_0 \Big(\overline{m}^{N,l}_{\tau^l_s}(j_1)m^{N,l}_{\tau^l_s}(j_2) 
 -\overline{m}^{N,l}_{\tau^{l-1}_s}(j_1)m^{N,l}_{\tau^{l-1}_s}(j_2) \Big) \\
 &+& \Big( m_{\tau^l_s}(j_1)\overline{m}^{N,l}_{\tau^l_s}(j_2) -m_{\tau^{l-1}_s}(j_1)\overline{m}^{N,l}_{\tau^{l-1}_s}(j_2) \Big)  ds \Big|^q\Big]^{1/q} \Big)  \\
 & \leq & \mathsf{C} \sum_{j_1=1}^{d_x}\sum_{j_2=1}^{d_x}   \mathbb{E}\Big[\Big| \int^{t+k_1\Delta_{l-1}}_0 \overline{m}^{N,l}_{\tau^l_s}(j_1)m^{N,l}_{\tau^l_s}(j_2) 
 -\overline{m}^{N,l}_{\tau^{l-1}_s}(j_1)m^{N,l}_{\tau^{l-1}_s}(j_2) ds \Big|^q\Big]^{1/q} \\
 &+&  \sum_{j_1=1}^{d_x}\sum_{j_2=1}^{d_x}   \mathbb{E}\Big[\Big| \int^{t+k_1\Delta_{l-1}}_0
  m^l_{\tau^l_s}(j_1)\overline{m}^{N,l}_{\tau^l_s}(j_2) -m^l_{\tau^{l-1}_s}(j_1)\overline{m}^{N,l}_{\tau^{l-1}_s}(j_2) ds \Big|^q \Big]^{1/q} \\
  &=:& T_3 + T_4.
\end{eqnarray*}
For $T_3$ we can use the difference of mean trick, as in \eqref{eq:bias_mean},
$$
T_3 =  \mathsf{C}
 \sum_{j_1=1}^{d_x}\sum_{j_2=1}^{d_x}   \mathbb{E}\Big[\Big| \int^{t+k_1\Delta_{l-1}}_0
\Big(\overline{m}^{N,l}_{\tau^l_s}(j_1) - \overline{m}^{N,l}_{\tau^{l-1}_s}(j_1) \Big) m^{N,l}_{\tau^l_s}(j_2) 
+ \overline{m}^{N,l}_{\tau^{l-1}_s}(j_1) \Big( {m}^{N,l}_{\tau^l_s}(j_2) - {m}^{N,l}_{\tau^{l-1}_s}(j_2) \Big)ds \Big|^q \Big]^{1/q}.
$$
We know $\mathbb{E}[|m^{N,l}_{\tau^l_s}|^q] \leq \mathsf{C}$, and $\overline{m}^{N,l}_{\tau^{l-1}_s}(j_1)$ is of oder $\mathcal{O}(N^{-\frac{1}{2}})$. The last bracket term is of order $\mathcal{O}(\Delta_l)$, arising from the strong error.  The first bracket term is the same, that appears in \autoref{lem:R1}, which is of order $\mathcal{O}(\frac{\Delta_l}{\sqrt{N}})$.
\\\\
Similarly, $T_4$ can be expressed as
$$
T_4 =  \mathsf{C}
 \sum_{j_1=1}^{d_x}\sum_{j_2=1}^{d_x}   \mathbb{E}\Big[\Big| \int^{t+k_1\Delta_{l-1}}_0
\Big(\overline{m}^{N,l}_{\tau^l_s}(j_2) - \overline{m}^{N,l}_{\tau^{l-1}_s}(j_2) \Big) m^{l}_{\tau^l_s}(j_1) 
+ \overline{m}^{N,l}_{\tau^{l-1}_s}(j_2) \Big( {m}^{l}_{\tau^l_s}(j_1) - {m}^{l}_{\tau^{l-1}_s}(j_1) \Big)ds \Big|^q \Big]^{1/q},
$$
which contains the same bounds as $T_3$, therefore implying that
$$
T_1 \leq \frac{\mathsf{C} \Delta_l}{\sqrt{N}}.
$$

For  $T_2$, we can rewrite it as
\begin{eqnarray*}
T_2 &=&\mathsf{C} \sum_{j_1=1}^{d_x}\sum_{j_2=1}^{d_x}\Big( \mathbb{E} \Big[ \Big| \int^{t+k_1\Delta_{l-1}}_0 \Big(\frac{1}{N} \sum^N_{i=1} {\zeta}^{i,l}_{\tau^{l-1}_s}(j_1)\Big)\Big(\frac{1}{N} \sum^N_{i=1} {\zeta}^{i,l}_{\tau^{l-1}_s}(j_2)\Big)
\\&-& \Big(\frac{1}{N} \sum^N_{i=1} {\zeta}^{i,l-1}_{\tau^{l-1}_s}(j_1)\Big)\Big(\frac{1}{N} \sum^N_{i=1} {\zeta}^{i,l-1}_{\tau^{l-1}_s}(j_2)\Big) \\&-& \mathbb{E}[ {\zeta}^l_{\tau^{l-1}_s}(j_1)]\mathbb{E}[ {\zeta}^l_{\tau^{l-1}_s}(j_2)]
 + \mathbb{E}[ {\zeta}^{l-1}_{\tau^{l-1}_s}(j_1)]\mathbb{E}[ {\zeta}^{l-1}_{\tau^{l-1}_s}(j_2)] ds \Big|^q \Big]^{1/q}\Big).
\end{eqnarray*}

Now to proceed, we again use $\overline{m}^{N,s}_{\tau^{l-1}_s} = {m}^{N,s}_{\tau^{l-1}_s} - m_{\tau^{l-1}_s}$, for $s \in \{l-1,l\}$,
which we can rewrite, in terms of its expectation, as
$$
\frac{1}{N}\sum^N_{i=1}\overline{\zeta}^{i,s}_{\tau^{l-1}_s}(j) = \frac{1}{N}\sum^N_{i=1}{{\zeta}}^{i,s}_{\tau^{l-1}_s}(j)  - \mathbb{E}[ {\zeta}^s_{\tau^{l-1}_s}(j)].
$$
Therefore the integral terms of $T_2$,  for $s \in \{l-1,l\}$, become
\begin{align*}
 &\Big(\frac{1}{N} \sum^N_{i=1} {\zeta}^{i,s}_{\tau^{l-1}_s}(j_1)\Big)\Big(\frac{1}{N} \sum^N_{i=1} {\zeta}^{i,s}_{\tau^{l-1}_s}(j_2)\Big) - \mathbb{E}[ {\zeta}^s_{\tau^{l-1}_s}(j_1)]\mathbb{E}[ {\zeta}^s_{\tau^{l-1}_s}(j_2)]
 \\ &= \Big(\frac{1}{N} \sum^N_{i=1}\overline{ {\zeta}}^{i,s}_{\tau^{l-1}_s}(j_1) +  \mathbb{E}[ {\zeta}^s_{\tau^{l-1}_s}(j_1)] \Big)\Big(\frac{1}{N} \sum^N_{i=1}\overline{\zeta}^{i,s}_{\tau^{l-1}_s}(j_2)  +  \mathbb{E}[ {\zeta}^s_{\tau^{l-1}_s}(j_2)] \Big) -\mathbb{E}[ {\zeta}^s_{\tau^{l-1}_s}(j_1)]\mathbb{E}[ {\zeta}^s_{\tau^{l-1}_s}(j_2)] \\
 &= \Big(\frac{1}{N} \sum^N_{i=1}\overline{\zeta}^{i,s}_{\tau^{l-1}_s}(j_1)\Big)\Big(\frac{1}{N} \sum^N_{i=1}\overline{\zeta}^{i,s}_{\tau^{l-1}_s}(j_2)\Big) + \Big(\frac{1}{N} \sum^N_{i=1}\overline{\zeta}_{\tau^{l-1}_s}^{i,s}(j_1)\Big)\mathbb{E}[ {\zeta}^s_{\tau^{l-1}_s}(j_2)] 
 +  \Big(\frac{1}{N} \sum^N_{i=1}\overline{\zeta}^{i,s}_{\tau^{l-1}_s}(j_2)\Big)\mathbb{E}[ {\zeta}^s_{\tau^{l-1}_s}(j_1)].
\end{align*}
By substituting this into $T_2$, we have

$$
T_2 = \mathsf{C}\sum_{j_1=1}^{d_x}\sum_{j_2=1}^{d_x} \Big( \mathbb{E} \Big[ \Big| \int^{t+k_1\Delta_{l-1}}_0
T_5+T_6+T_7 \ ds \Big|^q \Big]^{1/q} \Big), 
$$
where
\begin{eqnarray*}
T_5&=& \Big(\frac{1}{N} \sum^N_{i=1}\overline{\zeta}^{i,l}_{\tau^{l-1}_s}(j_1)\Big)\Big(\frac{1}{N} \sum^N_{i=1}\overline{\zeta}^{i,l}_{\tau^{l-1}_s}(j_2)\Big) - 
\Big(\frac{1}{N} \sum^N_{i=1}\overline{\zeta}^{i,l-1}_{\tau^{l-1}_s}(j_1)\Big)\Big(\frac{1}{N} \sum^N_{i=1}\overline{\zeta}^{i,l-1}_{\tau^{l-1}_s}(j_2)\Big),\\
T_6&=&\Big(\frac{1}{N} \sum^N_{i=1}\overline{\zeta}^{i,l}_{\tau^{l-1}_s}(j_1)\Big)\mathbb{E}[ {\zeta}^l_{\tau^{l-1}_s}(j_2)] -  \Big(\frac{1}{N} \sum^N_{i=1}\overline{\zeta}^{i,l-1}_{\tau^{l-1}_s}(j_1)\Big)\mathbb{E}[ {\zeta}^{l-1}_{\tau^{l-1}_s}(j_2)], \\
T_7 &=& \Big(\frac{1}{N} \sum^N_{i=1}\overline{\zeta}^{i,l}_{\tau^{l-1}_s}(j_2)\Big)\mathbb{E}[ {\zeta}^l_{\tau^{l-1}_s}(j_1)] -  \Big(\frac{1}{N} \sum^N_{i=1}\overline{\zeta}^{i,l-1}_{\tau^{l-1}_s}(j_2)\Big)\mathbb{E}[ {\zeta}^{l-1}_{\tau^{l-1}_s}(j_1)].
\end{eqnarray*}
For $T_6$ and $T_{7}$, they can be expressed as
\begin{eqnarray*}
T_6+T_{7}&=& \frac{1}{N}\sum^N_{i=1}\overline{\zeta}^{i,l}_{\tau^{l-1}_s}(j_1) \Big(\mathbb{E}[ {\zeta}^l_{\tau^{l-1}_s}(j_2)] - \mathbb{E}[ {\zeta}^{l-1}_{\tau^{l-1}_s}(j_2)]\Big) + \mathbb{E}[ {\zeta}^{l-1}_{\tau^{l-1}_s}(j_2)]\frac{1}{N}\sum^n_{i=1}\Big(\overline{\zeta}^{i,l}_{\tau^{l-1}_s}(j_1)
 - \overline{\zeta}_{\tau^{l-1}_s}^{i,l-1}(j_1)\Big) \\
&+& \frac{1}{N}\sum^N_{i=1}\overline{\zeta}_{\tau^{l-1}_s}^{i,l}(j_2) \Big(\mathbb{E}[ {\zeta}^l_{\tau^{l-1}_s}(j_1)] - \mathbb{E}[ {\zeta}^{l-1}_{\tau^{l-1}_s}(j_1)]\Big)  + \mathbb{E}[ {\zeta}^{l-1}_{\tau^{l-1}_s}(j_1)]\frac{1}{N}\sum^n_{i=1}\Big(\overline{\zeta}^{i,l}_{\tau^{l-1}_s}(j_2)
 - \overline{\zeta}^{i,l-1}(j_2)\Big).
\end{eqnarray*}
For $T_5$, we can rewrite it as
\begin{align*}
T_5 = \frac{1}{N}\sum^N_{i=1} \overline{\zeta}_{\tau^{l-1}_s}^{i,l}(j_2)\Big(\frac{1}{N} \sum^N_{i=1}\Big(\overline{\zeta}^{i,l}_{\tau^{l-1}_s}(j_1) - \overline{\zeta}^{i,l-1}_{\tau^{l-1}_s}(j_1)\Big) \Big) + \Big(\frac{1}{N}\sum^N_{i=1} \overline{\zeta}^{i,l}_{\tau^{l-1}_s}(j_1)\Big)
\frac{1}{N}\sum^N_{i=1}\Big(\overline{\zeta}^{i,l}_{\tau^{l-1}_s}(j_2) - \overline{\zeta}^{i,l-1}_{\tau^{l-1}_s}(j_2)\Big).
\end{align*}
Using these expressions, and through the generalized Minkowski and Jensen's inequality, $T_2$ can be simplified as
\begin{eqnarray*}
T_2 &\leq& \mathsf{C} \sum_{j_1=1}^{d_x}\sum_{j_2=1}^{d_x}\Big(\mathbb{E}\Big[ \Big| \int^{t+k_1\Delta_{l-1}}_0 \Big( \frac{1}{N}\sum^N_{i=1} \overline{\zeta}_{\tau^{l-1}_s}^{i,l}(j_2)\Big(\frac{1}{N} \sum^N_{i=1}\Big(\overline{\zeta}_{\tau^{l-1}_s}^{i,l}(j_1) - \overline{\zeta}_{\tau^{l-1}_s}^{i,l-1}(j_1)\Big) \Big)ds\Big|^p \Big]^{1/p}
\\ &+& \mathbb{E}\Big[ \Big|\int^{t+k_1\Delta_{l-1}}_0 \Big( \frac{1}{N}\sum^N_{i=1} \overline{\zeta}_{\tau^{l-1}_s}^{i,l}(j_1)\Big(\frac{1}{N} \sum^N_{i=1}\Big(\overline{\zeta}_{\tau^{l-1}_s}^{i,l}(j_2) - \overline{\zeta}_{\tau^{l-1}_s}^{i,l-1}(j_2)\Big) \Big)ds\Big|^p \Big]^{1/p} \\
&+& \Big|\mathbb{E}[\zeta_{\tau^{l-1}_s}^l(j_2)] - \mathbb{E}[\zeta_{\tau^{l-1}_s}^{l-1}(j_2)]\Big|\mathbb{E}\Big[\Big|\int^{t+k_1\Delta_{l-1}}_0\frac{1}{N}\sum^N_{i=1}\overline{\zeta}_{\tau^{l-1}_s}^{i,l}(j_1)ds\Big|^p \Big]^{1/p} \\
&+& \Big|\mathbb{E}[\zeta_{\tau^{l-1}_s}^{l-1}(j_2)] \Big|\mathbb{E} \Big[ \Big| \int^{t+k_1\Delta_{l-1}}_0\frac{1}{N}\sum^n_{i=1}\Big(\overline{\zeta}_{\tau^{l-1}_s}^{i,l}(j_1)  - \overline{\zeta}_{\tau^{l-1}_s}^{i,l-1}(j_1)\Big)ds \Big|^p \Big]^{1/p} \\
&+& \Big|\mathbb{E}[\zeta_{\tau^{l-1}_s}^l(j_1)] - \mathbb{E}[\zeta_{\tau^{l-1}_s}^{l-1}(j_1)]\Big|\mathbb{E}\Big[\Big| \int^{t+k_1\Delta_{l-1}}_0\frac{1}{N}\sum^N_{i=1}\overline{\zeta}_{\tau^{l-1}_s}^{i,l}(j_2)ds\Big|^p \Big]^{1/p}
\\&+& \Big|\mathbb{E}[\zeta_{\tau^{l-1}_s}^{l-1}(j_1)] \Big|\mathbb{E} \Big[ \Big| \int^{t+k_1\Delta_{l-1}}_0 \frac{1}{N}\sum^n_{i=1}\Big(\overline{\zeta}_{\tau^{l-1}_s}^{i,l}(j_2)  - \overline{\zeta}_{\tau^{l-1}_s}^{i,l-1}(j_2)\Big)ds \Big|^p \Big]^{1/p} \Big).
\end{eqnarray*}
Then by using the results from the strong and weak error of the diffusion processes, in \autoref{lem:strong_error1} and \ref{lem:weak_error1}, we
reach the following bound of
$$
\mathbb{E}\Big[\Big| \Big(R_{t+k_1\Delta_{l-1}}^{l} -R_{t+k_1\Delta_{l-1}}^{l-1}\Big) \Big|^q\Big]^{1/q} \leq \frac{\mathsf{C}\Delta_l}{\sqrt{N}}.
$$
\end{proof}
Therefore by combining all the results from \autoref{lem:R1} - \autoref{lem:R2}, leads to the desired result of  \eqref{eq:mul} in \autoref{prop:var_term2}.





\subsection{Proof of Theorem \ref{theo:main}}
\begin{proof}
Noting \eqref{eq:main_est} one has
$$
[\widehat{\overline{U}}_t^{ML}-{\overline{U}}_t](Y) = [\widehat{\overline{U}}_t^{N_0,0}-{\overline{U}}_t^0](Y) + \sum_{l=1}^L [\widehat{\overline{U}}_t^{N_l,l} -\widehat{\overline{U}}_t^{N_l,l-1}-{\overline{U}}_t^{l} +{\overline{U}}_t^{l-1}](Y) + [{\overline{U}}_t^L-{\overline{U}}_t](Y).
$$
Thus, by using three applications of the $C_2-$inequality we have
\begin{eqnarray*}
\mathbb{E}\Big[\Big\|[\widehat{\overline{U}}_t^{ML}-{\overline{U}}_t](Y)\Big\|_2^2\Big] &\leq& \mathsf{C}\Big(
\mathbb{E}\Big[\Big\|[\widehat{\overline{U}}_t^{N_0,0}-{\overline{U}}_t^0](Y)\Big\|_2^2\Big] \\ &+& \mathbb{E}\Big[\Big\|\sum_{l=1}^L [\widehat{\overline{U}}_t^{N_l,l} -\widehat{\overline{U}}_t^{N_l,l-1}-{\overline{U}}_t^{l} +{\overline{U}}_t^{l-1}](Y)\Big\|_2^2\Big]
+\mathbb{E}\Big[\Big\|[{\overline{U}}_t^L-{\overline{U}}_t](Y)\Big\|_2^2\Big]
\Big).
\end{eqnarray*}
For the first term on the R.H.S.~one can use \eqref{eq:first_prop_2} and for the last term on the R.H.S., we can use ~\eqref{eq:main_bias}. For the middle term, one has
\begin{eqnarray*}
\Big\|\sum_{l=1}^L [\widehat{\overline{U}}_t^{N_l,l} -\widehat{\overline{U}}_t^{N_l,l-1}-{\overline{U}}_t^{l} +{\overline{U}}_t^{l-1}](Y)\Big\|_2^2  &=& 
\sum_{l=1}^L\sum_{j=1}^{d_x}\Big([\widehat{\overline{U}}_t^{N_l,l} -\widehat{\overline{U}}_t^{N_l,l-1}-{\overline{U}}_t^{l} +{\overline{U}}_t^{l-1}](Y)^2\Big)(j)^2
\hspace{-1cm} \\ &+&  \sum_{l=1}^L\sum_{q=1}^L\mathbb{I}_{D^c}(l,q)\sum_{j=1}^{d_x}\Big([\widehat{\overline{U}}_t^{N_l,l} -\widehat{\overline{U}}_t^{N_l,l-1}-{\overline{U}}_t^{l} +{\overline{U}}_t^{l-1}](Y)\Big)(j) \times \\ && \Big([\widehat{\overline{U}}_t^{N_q,q} -\widehat{\overline{U}}_t^{N_q,q-1}-{\overline{U}}_t^{q} +{\overline{U}}_t^{q-1}](Y)\Big)(j).
\end{eqnarray*}
Then using a combination of the independence of the coupled particle systems along with \autoref{prop:var_term2} the proof can be concluded.

\end{proof}

\section*{Acknowledgments}
This work was supported by KAUST baseline funding.

\end{document}